\documentclass[10pt,reqno,oneside]{amsproc}
\title[A geophysical system with an ice-sheet]{A geophysical free-boundary system modeling an ice-sheet interacting with an ocean}
\author[M.~Hieber]{Matthias Hieber}
\address{Fachbereich Mathematik, Technische Universit\"at Darmstadt, Schlossgartenstr.~7, 64289 Darmstadt, Germany}
\email{hieber@mathematik.tu-darmstadt.de}
\author[I.~Kukavica]{Igor Kukavica}
\address{Department of Mathematics, University of Southern California, Los Angeles, CA 90089}
\email{kukavica@usc.edu}
\author[A.~Tuffaha]{Amjad Tuffaha}
\address{Department of Mathematics and Statistics, American University
of Sharjah, Sharjah, UAE}
\email{atufaha\char'100aus.edu}
\author[Q.~Xu]{Qi Xu}
\address{Department of Mathematics, University of Southern California, Los Angeles, CA 90089}
\email{xuqi@usc.edu}
  \chardef\forshowkeys=0
  \chardef\showllabel=0
  \chardef\refcheck=0
  \chardef\sketches=0
\usepackage{enumitem}
\usepackage{datetime}
\usepackage{fancyhdr}
\usepackage{comment}
\allowdisplaybreaks
\ifnum\forshowkeys=1

  \usepackage[notref,notcite,color]{showkeys}
\fi
\usepackage[margin=1in]{geometry}
\usepackage{amsmath, amsthm, amssymb}
\usepackage{times}
\usepackage{graphicx}
\usepackage[usenames,dvipsnames,svgnames,table]{xcolor}
\usepackage{marginnote}
\usepackage[unicode,breaklinks=true,colorlinks=true,linkcolor=black,urlcolor=black,citecolor=black]{hyperref}
\usepackage[most]{tcolorbox}
\usepackage{tikz}

\ifnum\refcheck=1
  \usepackage{refcheck}
\fi
\begin{document}
\def\YY{X}
\def\OO{\mathcal O}
\def\SS{\mathbb S}
\def\CC{\mathbb C}
\def\RR{\mathbb R}
\def\ZZ{\mathbb Z}
\def\HH{\text{H}}
\def\RSZ{\mathcal R}
\def\LL{\mathcal L}
\def\SL{\LL^1}
\def\ZL{\LL^\infty}
\def\GG{\mathcal G}
\def\tt{\langle t\rangle}
\def\erf{\mathrm{Erf}}
\def\mgt#1{\textcolor{magenta}{#1}}
\def\ff{\rho}
\def\gg{G}
\def\sqrtnu{\sqrt{\nu}}
\def\ww{w}
\def\ft#1{#1_\xi}
\def\lec{\lesssim}
\def\gec{\gtrsim}
\def\bard{\bar{\partial}}
\renewcommand*{\Re}{\ensuremath{\mathrm{{\mathbb R}e\,}}}
\renewcommand*{\Im}{\ensuremath{\mathrm{{\mathbb I}m\,}}}
\ifnum\showllabel=1
 \def\llabel#1{\marginnote{\color{lightgray}\rm\small(#1)}[-0.0cm]\notag}
% \def\llabel{\label}
%%%%%%%%%%%%%%%%%%%\def\llabel#1{\label{#1}}
%  \reversemarginpar
%  \def\llabel#1{\notag}
%  \def\llabel#1{\label{#1}}
\else
% % \def\llabel#1{\nonumber}
 \def\llabel#1{\notag}
\fi
\newcommand{\norm}[1]{\left\|#1\right\|}
\newcommand{\nnorm}[1]{\lVert #1\rVert}
\newcommand{\abs}[1]{\left|#1\right|}
\newcommand{\NORM}[1]{|\!|\!| #1|\!|\!|}
\newtheorem{theorem}{Theorem}[section]
\newtheorem{Theorem}{Theorem}[section]
\newtheorem{corollary}[theorem]{Corollary}
\newtheorem{Corollary}[theorem]{Corollary}
\newtheorem{proposition}[theorem]{Proposition}
\newtheorem{Proposition}[theorem]{Proposition}
\newtheorem{Lemma}[theorem]{Lemma}
\newtheorem{lemma}[theorem]{Lemma}
\theoremstyle{definition}
\newtheorem{definition}{Definition}[section]
\newtheorem{Remark}[Theorem]{Remark}
\def\theequation{\thesection.\arabic{equation}}
\numberwithin{equation}{section}
\definecolor{mygray}{rgb}{.6,.6,.6}
\definecolor{myblue}{rgb}{9, 0, 1}
\definecolor{colorforkeys}{rgb}{1.0,0.0,0.0}
\newlength\mytemplen
\newsavebox\mytempbox
\makeatletter
\newcommand\mybluebox{%
    \@ifnextchar[%]
       {\@mybluebox}%
       {\@mybluebox[0pt]}}
\def\@mybluebox[#1]{%
    \@ifnextchar[%]
       {\@@mybluebox[#1]}%
       {\@@mybluebox[#1][0pt]}}
\def\@@mybluebox[#1][#2]#3{
    \sbox\mytempbox{#3}%
    \mytemplen\ht\mytempbox
    \advance\mytemplen #1\relax
    \ht\mytempbox\mytemplen
    \mytemplen\dp\mytempbox
    \advance\mytemplen #2\relax
    \dp\mytempbox\mytemplen
    \colorbox{myblue}{\hspace{1em}\usebox{\mytempbox}\hspace{1em}}}
\makeatother
%Igor's macros  varmac
\def\PPs{\mathcal{P}_{\text{s}}}
\def\JJ{J}
\def\TT{\tilde{T}}
\def\bold{\colu {\bf OLD:}}
\def\eold{\colb {}}
\def\Omegaatm{\Omega_{\text{atm}}}
\def\Omegaoce{\Omega_{\text{oce}}}

\def\uatm{u_{\text{atm}}}
\def\vatm{v_{\text{atm}}}
\def\watm{w_{\text{atm}}}
\def\patm{p_{\text{atm}}}
\def\uoce{u_{\text{oce}}}
\def\voce{v_{\text{oce}}}
\def\vatm{v_{\text{atm}}}
\def\woce{w_{\text{oce}}}
\def\watm{w_{\text{atm}}}
\def\poce{p_{\text{oce}}}
\def\patm{p_{\text{atm}}}

\def\nablah{\nabla_{\text{H}}}
\def\nablaa{\nabla_{a}}
\def\nablaah{\nabla_{a,\text{H}}}
\def\diva{\div_{a}}
\def\divah{\div_{a,\text{H}}}
\def\divh{\div_{\text{H}}}
\def\rr{r}
\def\weaks{\text{\,\,\,\,\,\,weakly-* in }}
\def\inn{\text{\,\,\,\,\,\,in }}
\def\cof{\mathop{\rm cof\,}\nolimits}
\def\Dn{\frac{\partial}{\partial N}}
\def\Dnn#1{\frac{\partial #1}{\partial N}}
\def\tdb{\tilde{b}}
\def\tda{b}
\def\qqq{u}
\def\lat{\Delta_2}
\def\biglinem{\vskip0.5truecm\par==========================\par\vskip0.5truecm}
\def\inon#1{\hbox{\ \ \ \ \ \ \ }\hbox{#1}}                %in or on
\def\onon#1{\inon{on~$#1$}}
\def\inin#1{\inon{in~$#1$}}
\def\FF{F}
\def\andand{\text{\indeq and\indeq}}
\def\ww{w(y)}
\def\ll{{\color{red}\ell}}
\def\ee{\epsilon_0}
\def\startnewsection#1#2{ \section{#1}\label{#2}\setcounter{equation}{0}}   
\def\nnewpage{ }
\def\sgn{\mathop{\rm sgn\,}\nolimits}    
\def\Tr{\mathop{\rm Tr}\nolimits}    
\def\div{\mathop{\rm div}\nolimits}
\def\curl{\mathop{\rm curl}\nolimits}
\def\dist{\mathop{\rm dist}\nolimits}  
\def\supp{\mathop{\rm supp}\nolimits}
\def\indeq{\quad{}}           
\def\period{.}                       
\def\semicolon{\,;}                  
\def\nts#1{{\cor #1\cob}}
\def\colr{\color{red}}
\def\colrr{\color{black}}
\def\colb{\color{black}}
\def\coly{\color{lightgray}}
\definecolor{colorgggg}{rgb}{0.1,0.5,0.3}
\definecolor{colorllll}{rgb}{0.0,0.7,0.0}
\definecolor{colorhhhh}{rgb}{0.3,0.75,0.4}
\definecolor{colorpppp}{rgb}{0.7,0.0,0.2}
\definecolor{coloroooo}{rgb}{0.45,0.0,0.0}
\definecolor{colorqqqq}{rgb}{0.1,0.7,0}
\def\colg{\color{colorgggg}}
\def\collg{\color{colorllll}}
\def\cole{\color{black}}
\def\coleo{\color{colorpppp}}
\def\colu{\color{blue}}
\def\colc{\color{colorhhhh}}
\def\colW{\colb}   %color for weight
\definecolor{coloraaaa}{rgb}{0.6,0.6,0.6}%%%out
\def\colw{\color{coloraaaa}}
\def\comma{ {\rm ,\qquad{}} }            
\def\commaone{ {\rm ,\quad{}} }          
\def\nts#1{{\color{blue}\hbox{\bf ~#1~}}} 
\def\ntsf#1{\footnote{\color{colorgggg}\hbox{#1}}} 
\def\blackdot{{\color{red}{\hskip-.0truecm\rule[-1mm]{4mm}{4mm}\hskip.2truecm}}\hskip-.3truecm}
\def\bluedot{{\color{blue}{\hskip-.0truecm\rule[-1mm]{4mm}{4mm}\hskip.2truecm}}\hskip-.3truecm}
\def\purpledot{{\color{colorpppp}{\hskip-.0truecm\rule[-1mm]{4mm}{4mm}\hskip.2truecm}}\hskip-.3truecm}
\def\greendot{{\color{colorgggg}{\hskip-.0truecm\rule[-1mm]{4mm}{4mm}\hskip.2truecm}}\hskip-.3truecm}
\def\cyandot{{\color{cyan}{\hskip-.0truecm\rule[-1mm]{4mm}{4mm}\hskip.2truecm}}\hskip-.3truecm}
\def\reddot{{\color{red}{\hskip-.0truecm\rule[-1mm]{4mm}{4mm}\hskip.2truecm}}\hskip-.3truecm}
\def\tdot{{\color{green}{\hskip-.0truecm\rule[-.5mm]{3mm}{3mm}\hskip.2truecm}}\hskip-.1truecm}
\def\vert{\Vert}
\def\gdot{\greendot}
\def\bdot{\bluedot}
\def\pdot{\purpledot}
\def\ydot{\cyandot}
\def\rdot{\cyandot}
\def\fractext#1#2{{#1}/{#2}}
\def\ii{\hat\imath}
\def\fei#1{\textcolor{blue}{#1}}
\def\vlad#1{\textcolor{cyan}{#1}}
\def\igor#1{\text{{\textcolor{colorqqqq}{#1}}}}
\def\igorf#1{\footnote{\text{{\textcolor{colorqqqq}{#1}}}}}
\newcommand{\p}{\partial}
\newcommand{\UE}{U^{\rm E}}
\newcommand{\PE}{P^{\rm E}}
\newcommand{\KP}{K_{\rm P}}
\newcommand{\uNS}{u^{\rm NS}}
\newcommand{\vNS}{v^{\rm NS}}
\newcommand{\pNS}{p^{\rm NS}}
\newcommand{\omegaNS}{\omega^{\rm NS}}
\newcommand{\uE}{u^{\rm E}}
\newcommand{\vE}{v^{\rm E}}
\newcommand{\pE}{p^{\rm E}}
\newcommand{\omegaE}{\omega^{\rm E}}
\newcommand{\ua}{u_{\rm   a}}
\newcommand{\va}{v_{\rm   a}}
\newcommand{\omegaa}{\omega_{\rm   a}}
\newcommand{\ue}{u_{\rm   e}}
\newcommand{\ve}{v_{\rm   e}}
\newcommand{\omegae}{\omega_{\rm e}}
\newcommand{\omegaeic}{\omega_{{\rm e}0}}
\newcommand{\ueic}{u_{{\rm   e}0}}
\newcommand{\veic}{v_{{\rm   e}0}}
\newcommand{\up}{u^{\rm P}}
\newcommand{\vp}{v^{\rm P}}
\newcommand{\tup}{{\tilde u}^{\rm P}}
\newcommand{\bvp}{{\bar v}^{\rm P}}
\newcommand{\omegap}{\omega^{\rm P}}
\newcommand{\tomegap}{\tilde \omega^{\rm P}}
\renewcommand{\up}{u^{\rm P}}
\renewcommand{\vp}{v^{\rm P}}
\renewcommand{\omegap}{\Omega^{\rm P}}
\renewcommand{\tomegap}{\omega^{\rm P}}
\newcommand{\lot}{\text{l.o.t.}}

\begin{abstract}
We consider a free-boundary model for the ice-sheet interacting with an ocean.
The model captures the coupling between a viscous geophysical fluid and an elastic interface through kinematic and dynamic boundary conditions that account for hydrodynamic loading.
Using the ALE formulation, we derive a system on a fixed reference domain and establish local-in-time \emph{a priori} estimates for strong solutions with initial data in $H^2$. The main analytical difficulties arise from the nonlinear terms involving vertical derivatives and from high-order pressure contributions on the interface.
\colb
\end{abstract}

\colb
\maketitle
\setcounter{tocdepth}{2} 
\tableofcontents
\startnewsection{Introduction}{sec01}
In this paper, we consider a mathematical model for the interaction between ocean flow dynamics and a large structure. This model comprises the free boundary viscous primitive equations typically used to model fluid flow when the horizontal scales are by orders of magnitude larger than the vertical scale. The structural dynamics of the large structure are modeled by the Euler-Bernoulli equations describing the vertical deflection, while the interaction is captured by the velocity-matching boundary condition and the hydrodynamic loading forces in the structure equation.

The model describes the dynamics of a large floating structure in the
ocean. The Euler-Bernoulli thin beam theory is widely used in the
modeling of offshore or ocean engineering structures; see, for
instance, \cite{HC, PB, PV, XEK, XW}.  These structures could be
man-made structures or natural formations such as ice sheets. While
more intricate models for sea ice, incorporating viscoplastic and
thermoelastic effects, are readily available, e.g.~\cite{H, Hun},
other simpler thin plate or beam-theory based formulations have also
been considered \cite{Hu, HRLTCT, K, P, S, SDWRL, THAB}, where it is
common to use the Euler-Bernoulli beam for the deflection of the ice
floe.  On the other hand, the primitive equations are widely used to
model ocean and atmosphere dynamics. For example, the  Max-Planck
Global Sea/Ice Ocean model \cite{MHJLR}
(see also \cite{J.})
combines the primitive equations with the Hibler model of sea ice dynamics. A comprehensive review of the mathematical modeling of very large floating structures (VLFS) in the scientific community can be found in~\cite{KPV}. 

The main purpose of this paper is to study local-in-time solutions to the free boundary primitive equations when the dynamics of the free boundary are governed by the Euler-Bernoulli equations. One can find extensive mathematical literature on the primitive equations; see e.g.~\cite{BH, CLT1,  CLT2, CT, HIRZ, HK, IKZ, KZ, LT1, LTW1, LTW2, LTW3,LTW4,PTZ}, and for the inviscid version, see
\cite{CINT,KMVW,KTVZ,LT2,MW,R,W} to mention only few references.

. However, currently, the only mathematical treatment of the primitive equations with a free-boundary is due to Li and Liang~\cite{LL}. 

The problem we study lies broadly in a class of free boundary viscous fluid-structure interactions, which involves the study of a fluid (typically modeled by the Navier-Stokes equations) interacting with elastic structure governed by a systems of elasticity. The mathematical study of such models traces back to~\cite{CDEG, DEGL, GM} and \cite{B} where weak solutions were established in both 2D and 3D for a model comprising the Navier-Stokes equations coupled with a damped beam equation governing the motion of the moving interface. Strong local-in-time solutions and global-in-time solutions for small data up to contact were also established by several authors in different settings and configurations
\cite{CCS, CS, KLT1,KLT2,L1, L2, GH, GHL, SS, SSu}, with and without damping in the structure. Other models involving compressible flows were also studied more recently; see e.g.~\cite{BS, HIRZ, Tr, TW, Mi}.  More intricate models involving the coupling of the Navier-Stokes equations with linear and nonlinear Koiter shell models governing interface dynamics were also considered by several authors, see e.g.~\cite{L, LR, MRR, MC1, MC2, MS}.  For inviscid flow-structure interaction problems where fluid dynamics are modeled by potential flow equations or the Euler equations, we refer the reader to more recent works \cite{AS,BKMT, KNT, KT1, KT2, LA, WY}. 
 
Our main result in this paper is the establishment of local-in-time a priori estimates satisfied by strong solutions when the initial data for the fluid velocity and the interface displacement are in the $H^{2}$ Sobolev space. We transform the PDE system to a fixed domain using the ALE (Arbitrary-Lagrangian-Eulerian) transformation. It is known that the primitive equations derived from the Navier-Stokes equations under the small vertical-to-horizontal length-scale assumption exhibit one derivative loss of regularity in the vertical velocity component.

This intrinsic loss of regularity makes the free-boundary coupling particularly delicate. In particular, under the ALE transformation, the nonlinear terms involving the vertical derivative of the vertical velocity $\partial_3 v_3$ cannot be eliminated directly by the incompressibility condition. As a result, higher-order vertical derivatives appear in the transformed system, which prevents a direct closure of the energy estimates. To summarize, compared to the primitive equations in a fixed domain, the free-boundary primitive-equation exhibit a one-derivative loss.

To overcome this difficulty, we design a series of anisotropic $L^2$-based estimates that carefully combine horizontal and temporal derivatives, together with interpolation inequalities, to effectively minimize the order of the terms containing vertical spatial derivatives. A second difficulty stems from the pressure estimate: in the ALE formulation, additional high-order terms emerge that cannot be absorbed by the dissipation alone, particularly the first (boundary) term in~\eqref{EQ89}. Instead of relying solely on an elliptic estimate for the pressure, we exploit precise bounds on the ALE coefficients and integrate the pressure estimate into the hierarchy of $L^2$-energy inequalities. This strategy allows the pressure contribution to be systematically canceled within the overall energy framework, thereby closing the estimates and ensuring the robustness of the a priori bounds. These precise bounds are dictated by the assumption of small time, which maintains the interface close to the initial configuration.

Another feature of the problem under consideration is the elliptic system satisfied by the 2D pressure variable, from which we derive pressure estimates in terms of the height function. 
A particular anomaly arises from the variable Laplacian term of the vertical velocity, a high-order term, which appears in the pressure equation, and from which one can obtain only $L^2$ in time control of the $H^{2}$ space norm of pressure, but no pointwise in time regularity in any Sobolev space norm. This further complicates the time derivative estimates, since no control of the time derivative of the pressure is possible, and this results in pointwise terms that could only be absorbed by a careful reliance on the smallness of the variable coefficients for small time, and eventually leads to a delicate non-standard type Gronwall inequality.  Additionally, careful Agmon-Douglis-Nirenberg type estimates have to be derived to fit the situation at hand in order to obtain the elliptic estimate required  to recover the full regularity of the velocity.

The paper is structured as follows. Section~\ref{sec02} introduces the model and the notation.  The following section reformulates the system in the ALE variables and presents the main result, Theorem~\ref{T01}. Section~\ref{sec6} presents several preliminary lemmas.  Sections~\ref{sec13}, \ref{sec14}, \ref{sec15}, and~\ref{sec16} contain the tangential, elliptic, pressure, and time-derivative bounds, respectively, with the summaries given in Lemmas~\ref{L06}, \ref{L05}, \ref{L07}, and~\ref{L08}.  Finally, Section~\ref{sec17} provides the proof of the main theorem.

\section{The model}\label{sec02}
We consider a system of PDEs consisting of the primitive equations of the ocean, defined on an evolving domain $\Omega(t)\subseteq \mathbb{R}^{3} $, coupled with a fourth-order Euler-Bernoulli equation describing the dynamics of the moving two-dimensional interface. The variables for the flow are the velocity vector $(v,w)$ and a scalar fluid pressure~$p$; here, $v=(v_{1}, v_{2})$ represents the horizontal components of the velocity while  $w$ denotes the vertical velocity. The primitive equations then read
  \begin{align}
  \begin{split}
  &
  v_{t} - \Delta v + (v \cdot \nabla_{\HH}) v + w \partial_{z} v+ \nablah p = 0
  \inin{\Omega(t)},
  \\&
  \div_{\HH} v = -\partial_{z}w 
  \inin{\Omega(t)},
  \\&
  \partial_{z} p =0
  \inin{\Omega(t)}  
   .
  \end{split}
   \label{EQ01}
  \end{align}
Here, the operator $\nablah$ denotes the horizontal gradient $(\partial_{1}, \partial_{2})$.
The domain $\Omega(t) $ is defined as 
  \begin{align}
   \Omega(t)
     = \bigl\{ (x_{1}, x_{2}, z)
        :  (x_{1}, x_{2} ) \in \mathbb{T}^{2},   0 \leq z \leq h( x_{1},x_{2},t)
       \bigr\}
   ,
   \label{EQ02}
  \end{align}
where $h$ represents the height of the interface.
We denote by $\Gamma_{1}$ the top boundary parameterized by $h$, i.e.,
  \begin{align}
   \Gamma_{1}(t)= \{ (x_{1}, x_{2}, h(x_{1},x_{2}, t)) :  (x_{1},x_{2} ) \in \mathbb{T}^{2} \}
   ,
   \label{EQ03}
  \end{align}
while the bottom boundary of $\Omega(t)$, denoted by $\Gamma_{0}$,  is fixed and is defined as
  \begin{align}
   \Gamma_{0}= \bigl\{(x_{1}, x_{2}, 0) :  (x_{1},x_{2} ) \in \mathbb{T}^{2} \bigr\}.
   \label{EQ04}
  \end{align}
On the top, we impose the no-slip boundary condition on the horizontal
velocity component, i.e.,
  \begin{align}
v =0    \onon{\Gamma_{1}(t) \times [0,T]} .
   \label{EQ05}
  \end{align}
We denote by $\nu$ the unit dynamic normal on~$\Gamma_{1}(t)$ which is given explicitly by
  \begin{align}
   \nu(t) =  \frac{1}{ \sqrt{|\nablah h|^{2}+1}}  (-\partial_{1} h, -\partial_{2} h ,1)^{T}.
   \label{EQ06}
  \end{align}
The dynamics of the free surface are governed by a fourth order structure equation, with the fluid pressure acting as a lifting force,
  \begin{align}
     h_{tt}+ \Delta_{\HH}^{2} h = p  \onon{\Gamma_{1}(t) \times [0,T]}  
   ,
   \label{EQ07}
  \end{align}
where $p$ is evaluated at the point~$(x_1,x_2,h)$.
In addition, the  kinematic condition provides the velocity matching between the fluid and the structure
and is given by 
  \begin{align}
   h_{t}  =w  \onon{\Gamma_{1}(t) \times [0,T]}.
   \label{EQ08}
  \end{align}
On the rigid bottom boundary, we impose the no-slip boundary condition
  \begin{align}
   (v,w) =(0,0,0)    \onon{\Gamma_{0} \times [0,T]}   
   .
   \label{EQ09}
  \end{align}
We prescribe periodic boundary conditions in the $x_{1}$, $x_{2}$ directions,
i.e., the horizontal component of the domain is~$\mathbb{T}^2$.
We also denote the initial conditions by 
  \begin{equation}
    (v,w,h,h_{t})(0,\cdot)
    = (v_0, w_0, 0, h_{1})
    .
   \llabel{EQ10}
  \end{equation}
In the reminder of the section, we derive an elliptic equation
for the pressure.
Applying the horizontal divergence operator $\div_{\HH}$ to the first equation in \eqref{EQ01}, we obtain
\begin{align}
\Delta_{\HH} p = -\div_{\HH} v_{t} + \Delta \div_{\HH} v - (v \cdot \nabla_{\HH}) \div_{\HH} v 
      -  \nabla_{\HH} v : \nabla^{T}_{\HH}  v - w \partial_{z} \div_{\HH} v  
      -   \nabla_{\HH} w \cdot \partial_{z}  v  
   \llabel{EQ80}
\end{align}
on the variable domain~$\Omega(t)$.
We next note that $\div_{\HH} v = -\partial_{z} w $ from the divergence condition in \eqref{EQ01}, so that we have
\begin{align}
\Delta_{\HH} p = \partial_{z}  w_{t} - \partial_{z} \Delta w +\partial_{z} ( (v \cdot \nabla_{\HH})  w )
      -  \nabla_{\HH} v : \nabla^{T}_{\HH}  v + w \partial^{2}_{z} w  
      -   2 \nabla_{\HH} w \cdot \partial_{z}  v  
.
\label{EQ81}      
\end{align}

\startnewsection{Harmonic change of variable}{sec03}
For simplicity, we assume that the interface is initially flat, i.e.,
  \begin{align}
   h(x_{1}, x_{2}, t=0) = 1
   \inin{\Omega}
   ,
   \label{EQ11}
  \end{align}
where $\Omega = \Omega(0)= \mathbb{T}^{2} \times [0,1]$.
We define a change of variable which fixes the domain as follows.
First, let $\phi$ be the harmonic extension of $h$, defined as
  \begin{align}
  \begin{split}
  &\Delta \phi =0
  \inin{\Omega}
  ,
  \\&
  \phi(x_{1}, x_{2}, 1,t) =  h(x_{1},x_{2}, t)
    \onon{\Gamma_1}
    ,
  \\&
  \phi(x_{1}, x_{2}, 0,t) =  0
    \onon{\Gamma_0}
    ,
  \end{split}
   \label{EQ12}
  \end{align}
and then
define the map $\eta(x_{1}, x_{2}, t)\colon \Omega \to \Omega(t)$  as
  \begin{align}
   \begin{split}
   \eta (x_{1}, x_{2}, x_{3}, t)= ( x_{1}, x_{2}, \phi(x_{1}, x_{2}, x_{3}, t) )
   .   
   \end{split}
   \llabel{EQ13}
   \end{align}
We next introduce the inverse $a$ of the Jacobian matrix $\nabla
\eta$, which reads
  \begin{align}
   a = 
   \begin{pmatrix}
    1 & 0 & 0 \\
    0& 1 & 0 \\
    \fractext{-\partial_{1} \phi}{\partial_{3} \phi} & \fractext{-\partial_{2} \phi}{\partial_{3} \phi} & \fractext{1}{\partial_{3} \phi} 
   \end{pmatrix}
   \label{EQ14}
  \end{align}
and by $b$ the cofactor matrix 
  \begin{align}
   b
   = J a
   = (\partial_{3} \phi)  a 
   =
   \begin{pmatrix}
    \partial_{3} \phi  & 0 & 0 \\
    0& \partial_{3} \phi & 0 \\
    -\partial_{1} \phi & -\partial_{2} \phi & 1
   \label{EQ16}
   \end{pmatrix}
     ,
   \end{align}
where 
  \begin{equation}
   J=\det \nabla\eta=\partial_{3} \phi
   \llabel{EQ15}
  \end{equation}
stands for the Jacobian of the matrix~$\nabla \eta$.
Note that $\phi|_{t=0} = x_{3}$, and thus $\nabla \eta |_{t=0} = I $.

\subsection{Fixing the domain}
We now introduce the new variables for the velocity $(\bar v, \bar w)$
and the pressure~$\bar p$ as
  \begin{align}
   \begin{split}
&
   \bar{v} = v  \circ \eta  \inin{\Omega},
   \\&
   \bar{w} = w  \circ \eta  \inin{\Omega},
   \\&
   \bar{p}=   p  \circ \eta   \inin{\Omega}
   ,
  \end{split}
   \llabel{EQ17}
  \end{align}
which are defined on a fixed domain~$\Omega=\Omega(0)$.
We next perform the harmonic change of variables to the forthcoming differential operators, which leads to
  \begin{align}
   \partial_{t} &\to \partial_{t} -  a_{33}\partial_{t} \phi \partial_{3}, \llabel{EQ82}
   \\
   \partial_{z} &\to a_{33} \partial_{3 }= \frac{1}{\partial_{3} \phi} \partial_{3},\llabel{EQ83}\\
   \Delta   \to \Delta_{a} & \equiv  \Delta_{x_{1}, x_{2}} +\sum_{i,j=1}^{3}   a_{ji} \partial_{j}( a_{3i}\partial_{3})  + \sum_{i=1}^{2} a_{3i} \partial_{i}\partial_{3},
   \llabel{EQ84}\\
   \Delta_{\HH} \to \Delta_{a,\HH} & \equiv  \Delta_{x_{1}, x_{2}} + \underbrace{ \sum_{j=1}^{3} \sum_{i=1}^{2}  a_{ji} \partial_{j}( a_{3i}\partial_{3})  + \sum_{i=1}^{2} a_{3i} \partial_{i}\partial_{3} }_{\tilde{\Delta}_{a}}, \llabel{EQ85}\\
   \nabla_{\HH} \to \nabla_{a,\HH} & \equiv  \nabla_{x_{1}, x_{2}} + (a_{31} \partial_{3}, a_{32} \partial_{3})
   .
   \llabel{EQ86}
\end{align}
Here and below, we use $x_3$ and $z$ to denote the third variable interchangeably.
Applying this change of variable to the primitive equations,  we get
  \begin{align}
   &
   \bar{v}_{t}
   -    \Delta_{x_{1}, x_{2}} \bar{v}
       -   a_{kl} \partial_{k}(a_{3l} \partial_{3} \bar{v})
         - \sum_{i=1,2} a_{3i} \partial_{3} \partial_{i} \bar{v}
     +\bar{v}_\gamma a_{j\gamma}\partial_j\bar{v}
     + \frac{1}{\partial_{3} \phi} (\bar{w}- \phi_{t}) \partial_{3} \bar{v} 
     + \nabla_{\HH} \bar{p} = 0 \inin{\Omega},
    \llabel{EQ18a}
   \\&
\frac{1}{\partial_{3} \phi} \partial_{3} \bar{w}=   \div_{a,\HH} \bar{v} 
   \inin{\Omega}
   ,
   \label{EQ19a}
  \end{align}
with the boundary  conditions
  \begin{align}
  \bar{v}=0    \onon{\Gamma_{1} \times [0,T]}  
  \llabel{EQ20a}
  \end{align}
and 
\begin{align}
 \begin{split}
  h_{t}(t,x_{1}, x_{2})= \bar{w}(t, x_{1},x_{2},1)  \onon{\mathbb{T}^{2} \times [0,T]}   
  .
 \end{split}
  \llabel{EQ21a}
\end{align}
After the change of variable, the plate equation still reads 
\begin{align}
 \begin{split}
  h_{tt}+\Delta_{\HH}^{2}h = \bar{p}  \onon{\Gamma_{1} \times [0,T]}   
  ,
 \end{split}
  \llabel{EQ2a}
\end{align}
as it is originally written in the Lagrangian coordinates.
The no-slip boundary condition on $\Gamma_{0}$ satisfied by the new variables are
  \begin{align}
  \bar{v}=0    \onon{\Gamma_{0} \times [0,T]}  ,
  \llabel{EQ23a}  \\
  \bar{w}=0    \onon{\Gamma_{0} \times [0,T]}  
  .
  \llabel{EQ24a}  
  \end{align}
Note that the condition for the pressure \eqref{EQ01}$_3$ becomes
  \begin{align}
  \begin{split}
   \partial_{3} \bar{p} = 0,
  \end{split}
   \llabel{EQ29a}
  \end{align}
which implies that  $ \bar{p}$ depends only on $x_{1}$ and~$x_{2}$.

Note that \eqref{EQ81} in the new variables takes the form
  \begin{align}
    \begin{split}
   \Delta_{x_{1}, x_{2}}  \bar{p}
   &= \frac{1}{\partial_{3}\phi}
   \biggl(
   \partial_{3} ( \bar{w}_{t}- a_{33} \phi_{t} \partial_{3} \bar{w} ) - \partial_{3} \Delta_{a} \bar{w} 
   + \partial_{3}( (\bar{v} \cdot \nabla_{a,\HH})  \bar{w} )
   \\&\indeq\indeq\indeq\indeq\indeq\indeq
   -  \partial_{3}\phi \nabla_{a,\HH} \bar{v} : \nabla^{T}_{a, \HH}  \bar{v} + \bar{w} \partial_{3}\left(\frac{1}{\partial_{3} \phi} \partial_{3}\bar{w}\right)   
   -  2 \nabla_{a,\HH} \bar{w} \cdot \partial_{3}  \bar{v}
   \biggr)
   .
  \end{split}
  \label{EQ120}
  \end{align}

\subsection{The setting after the variable change}
Since we shall not return to the original Eulerian variables, we
now omit bars over the variables. Here we state precisely the problem
in the new variables and formulate the main result.

The momentum and divergence equations read
  \begin{align}
   &
   {v}_{t}
     - \diva ( \nablaa {v})
     + v \cdot \nabla_{a,\HH} v
     + \frac{1}{\partial_{3} \phi} ({w}- \phi_{t}) \partial_{3} {v} 
     + \nabla_{\HH} {p} = 0 \inin{\Omega},
    \label{EQ18}
   \\&
\frac{1}{\partial_{3} \phi} \partial_{3} {w}=   \div_{a,\HH} {v} 
   \inin{\Omega}
   ,
   \label{EQ19}
  \end{align}
where $p$ depends only on $x_1$ and~$x_2$.
The boundary condition for the velocity on $\Gamma_{1}$ is
given by
  \begin{align}
  {v}=0    \onon{\Gamma_{1} \times [0,T]}  
  \label{EQ20}  
  \end{align}
and 
\begin{align}
 \begin{split}
  h_{t}(t,x_{1}, x_{2})= {w}(t, x_{1},x_{2},1)  \onon{\mathbb{T}^{2} \times [0,T]}   
  .
 \end{split}
  \llabel{EQ21}  
\end{align}
On the other hand, the plate equation reads 
\begin{align}
 \begin{split}
  h_{tt}+\Delta_{\HH}^{2}h = {p}  \onon{\Gamma_{1} \times [0,T]}   
  .
 \end{split}
  \llabel{EQ22}  
\end{align}
The no-slip boundary condition on $\Gamma_{0}$ satisfied by the new variables becomes
  \begin{align}
  {v}=0    \onon{\Gamma_{0} \times [0,T]}  ,
  \label{EQ23}  \\
  {w}=0    \onon{\Gamma_{0} \times [0,T]}  
  .
  \llabel{EQ24}  
  \end{align}
Above, we used the following definitions.
The variable Laplacian is defined as
  \begin{align}
  \begin{split}
   \diva ( \nablaa {v})
     &\equiv
     \Delta_{x_{1}, x_{2}} {v}
       +   a_{kl} \partial_{k}(a_{3l} \partial_{3} {v})
         + \sum_{i=1,2} a_{3i} \partial_{3} \partial_{i} {v}
   ,
  \end{split}
   \llabel{EQ25}
  \end{align}
while the variable horizontal gradient reads
  \begin{align}
  \begin{split}
   [\nablaah{v}]_i
   \equiv  \nablah {v}_i 
       + \begin{pmatrix} a_{31} \\ a_{32}
     \end{pmatrix}  
       \partial_{3} {v}_i
         \comma i=1,2
   ,
  \end{split}
   \llabel{EQ26}
  \end{align}
which is consistent with
  \begin{equation}
   [v \cdot \nabla_{a,\HH} v]_i
   =
   v_\alpha\partial_\alpha v_i+
   v_\alpha a_{3\alpha}\partial_3 v_i
        \comma i=1,2
        .
   \llabel{EQ30}
  \end{equation}
The variable horizontal divergence reads
  \begin{align}
  \begin{split}
  \div_{a,\HH} {v}\equiv 
      \divh {v} 
      + a_{31}\partial_{3}{v}_{1}
        +a_{32}\partial_{3}{v}_{2}
  .
  \end{split}
   \llabel{EQ28}
  \end{align}
In $x_1,x_2$ directions, we assume 1-periodic boundary conditions on $v$, $w$,
$p$, and~$h$.

We also restate \eqref{EQ120} without bars, i.e.,
  \begin{align}
    \begin{split}
   \Delta_{x_{1}, x_{2}}  {p}
   &= \frac{1}{\partial_{3}\phi}
   \biggl(
   \partial_{3} ( {w}_{t}- a_{33} \phi_{t} \partial_{3} {w} ) - \partial_{3} \Delta_{a} {w} 
   + \partial_{3}( ({v} \cdot \nabla_{a,\HH})  {w} )
   \\&\indeq\indeq\indeq\indeq\indeq\indeq
   -  \partial_{3}\phi \nabla_{a,\HH} {v} : \nabla^{T}_{a, \HH}  {v} + {w} \partial_{3}\left(\frac{1}{\partial_{3} \phi} \partial_{3}{w}\right)   
   -  2 \nabla_{a,\HH} {w} \cdot \partial_{3}  {v}
   \biggr)
   .
  \end{split}
   \label{EQ87}
  \end{align}

Note that we shall use $x_3$ and $z$ interchangably; similarly, we
shall use $v_3$ for $w$ without mention.

The following is the main result of this paper.

\cole
\begin{Theorem}
\label{T01}
Given initial data $(v_{0}, w_{0}) \in H^{2}(\Omega)$ and $h_{0}=1$
and $h_{1} \in   H^{4}(\mathbb{T}^{2})$,  assume that
there exists
a smooth solution $({v},{w},{p}, h,h_{t})$ to the system
\eqref{EQ01}--\eqref{EQ09}
on the time interval $[0,\TT]$ for some~$\TT$.
Assume that
$\|v_0\|^2_{H^2(\Omega)}
  +\|w_0\|^2_{H^2(\Omega)}
  +\|h\|^2_{H^4(\mathbb{T}^2)}\leq M_0$,
for some~$M_0$.
  Then the smooth solution to the system
\eqref{EQ01}--\eqref{EQ09}
on the time interval $[0,\TT]$ satisfies 
\begin{align}
\begin{split}
    &\sup_{0\leq t\leq \TT} \left(\|\partial''v\|^2_{L^2}
    +\|\partial'v\|^2_{L^2}+\|v\|^2_{L^2}+\|v_t\|^2_{L^2}
    +\|h_{tt}\|^2_{L^2}+\|h_t\|_{H^2}^2+\|h\|^2_{H^4}\right)
    \\&\indeq
    +\int_0^{\TT} \left(\|\partial''v\|^2_{H^1}+\|\partial'\partial_3v\|^2_{L^2}
    +\|\partial_3v\|^2_{L^2}+\|v_t\|^2_{H^1}\right)\, dt
    \leq K
    ,
\end{split}
   \llabel{EQ1111}
\end{align}
where $K$ is an explicit polynomial of~$M_0$.
In other words, the above energy inequality yields the Sobolev regularity of the solution on 
$[0,\TT]$:
\begin{align}
\begin{split}
&v\in L^2((0,{\TT});H^3(\Omega)),
\ v_t\in L^\infty((0,{\TT});L^2(\Omega))\cap L^2((0,{\TT});H^1(\Omega)),\\
&h\in L^\infty((0,{\TT});H^4(\mathbb{T}^2)),
\ h_t\in L^\infty((0,{\TT});H^2(\mathbb{T}^2)),
\ h_{tt}\in L^\infty((0,{\TT});L^2(\mathbb{T}^2)).
\end{split}
   \llabel{EQ123}
\end{align}
\end{Theorem}
\colb

The remainder of the paper is devoted to proving this theorem.
Note that the theorem provides an a priori estimate; we will address the construction of solutions in a future work.

\startnewsection{Preliminary lemmas}{sec6}
The first statement asserts the regularity of the cofactor matrix~$b$ based on regularity of the height function.

\cole
\begin{Lemma}
\label{L01}
Let $(h, {v}) \in L^2([0,T];H^4(\mathbb{T}^2)\times H^3(\Omega))$.  Then the ALE map $\phi$ and the coefficient matrices $a$ and $b$ satisfy the following inequalities:\\
(i) $\|\phi(t) \|_{H^{r}} \lesssim
    \| h(t) \|_{H^{r-1/2}(\mathbb{T}^{2})}$,
    for $1 \leq r \leq 2$,\\
(ii) $\|b(t) \Vert_{H^{r}} \lesssim \| h(t) \|_{H^{r+1/2}(\mathbb{T}^{2})}$, for $r \geq 0$, and \\
(iii) $ \| \partial_{3} \phi(t)\|_{L^{\infty}} \lesssim
    \| h(t) \Vert_{H^{s+1/2}(\mathbb{T}^{2})}$,
   for $s > 3/2$,\\
for $t\in[0,T]$.
\end{Lemma}
\colb

Above, $T>0$ is the time of existence.

\begin{proof}[Proof of Lemma~\ref{L01}]
(i) Since $\phi$ is the harmonic extension of $h $ from the boundary to the interior of $\Omega$,
the inequality follows by using elliptic regularity. \\
(ii) The estimates follow from the definition of $b$ and part~(i). \\
(iii) By Sobolev's inequality,  we have
$\Vert \partial_{3} \phi \Vert_{L^{\infty}} \lesssim \Vert \phi \Vert_{H^{s+1}}$,
for  $s>3/2$, and the bound then follows from~(i).
\end{proof}

\cole
\begin{Lemma}
\label{L02}
Let $\epsilon\in(0,\frac{1}{2}]$. Assume that $\phi$ is smooth and
\begin{align}
    \|h\|_{H^4(\Gamma_1)},\|h_t\|_{H^2(\Gamma_1)}\leq K,
    \indeq\indeq 
    t\in[0,T],
    \label{EQ32}
\end{align}
where $K$
is as in the statement of main theorem and $0<T\leq{\TT}$. Then we have
\begin{align}
    \|a-I\|_{H^r},\|b-I\|_{H^r},\|J-1\|_{H^r}\leq\epsilon,
    \indeq 
    t\in[0,\min\{T_0,T\}]
    \label{EQ33}
\end{align}
and \begin{align}
    \|J-1\|_{L^\infty}\leq \epsilon,
    \indeq\indeq
    t\in[0,\min\{T_0,T\}]
    ,
    \label{EQ34}
\end{align}
where $T_0$ satisfies
\begin{align}
  \begin{split}
    0<T_0\leq\frac{\epsilon}{CK},
  \end{split}
   \llabel{EQ35}
\end{align}
with $C$ depending on dimension three only, and the exponent 
$r$ can be chosen arbitrarily from the interval $[1.5,3.5)$.
\end{Lemma}
\colb

From here on, we denote $T_0$ simply by $T$ and assume that \eqref{EQ33} and \eqref{EQ34} hold with a sufficiently small constant $\epsilon>0$ to be determined below. Note that, by \eqref{EQ33}, we also have
\begin{align}
    \|a-I\|_{L^\infty},\|b-I\|_{L^\infty}\lesssim\epsilon,
    \indeq\indeq
    t\in[0,T],
    \llabel{EQ36}
\end{align}
while \eqref{EQ34} gives
\begin{align}
    \frac{1}{2}\leq J\leq \frac{3}{2},
    \indeq\indeq
    t\in[0,T]
    ;
    \label{EQ37}
\end{align}
in particular, $J=\partial_3\phi$ is positive and stays away from~$0$.

\begin{proof}[Proof of Lemma~\ref{L02}]
By \eqref{EQ11} and \eqref{EQ32}, we have 
$h\in C\left([0,T],H^2(\Omega)\right)$.
By the Sobolev embedding theorem, there exists $T_0>0$ such that
\begin{align}
      \|h-1\|_{H^2}\lesssim\epsilon
    \indeq\indeq 
    t\in [0,\min\{T_0,T\}]
    ,
    \llabel{EQ38}
\end{align}
where $T_0$ satisfies
\begin{align}
  \begin{split}
    0<T_0\leq\frac{\epsilon}{CK}
    .
  \end{split}
   \llabel{EQ39}
\end{align}
Using the elliptic regularity,
\begin{align}
    \|b-I\|_{H^{1.5}(\Omega)}
    \lesssim
    \|h-1\|_{H^2(\Gamma_1)}
    \lesssim\epsilon
    ,
    \llabel{EQ40}
\end{align}
and, again by \eqref{EQ32},
\begin{align}
    \|b-I\|_{H^{3.5}(\Omega)}
    \lesssim
    \|h-1\|_{H^4(\Gamma_1)}
    \lesssim K+1
    .
    \llabel{EQ41}
\end{align}
The second inequality in \eqref{EQ33} then follows by interpolation.
The remaining part of \eqref{EQ33} follows from \eqref{EQ14} and~\eqref{EQ16}. 
Finally, the $L^\infty$
estimate in \eqref{EQ34} follows directly from Sobolev’s inequality.
\end{proof}
\colb

\cole
\begin{Lemma}
\label{L03}
Let $v(x_1,x_2,z)$ be a smooth function satisfying our boundary conditions \eqref{EQ20} and~\eqref{EQ23}. Then
\begin{align}
    \left\Vert \partial_z^{m\gamma}\partial'^\beta v\right\Vert _{L^2}
    \lesssim
    \Vert \partial_z^{m} v\Vert ^{\gamma}_{L^2}
    \Vert (\partial')^{\frac{\beta}{1-\gamma}} v\Vert ^{1-\gamma}_{L^2}
    ,
    \label{EQ42}
\end{align}
where
$m\in\mathbb{N}$ and
$\beta>0$ and $0<\gamma<1$ are such that
$m \gamma \in \mathbb{N}$.
\end{Lemma}
\colb

Above and in the sequel, for smooth $v$
such that
    \begin{align}
      \begin{split}
    v(x_1,x_2,x_3)
    =\sum_{k\in \mathbb{Z}^2}\sum_{n=1}^\infty
    \hat{v}_{k,n}
    e^{ik\cdot x}
    \sin(n\pi x_3)
    ,
  \end{split}
   \label{EQ43}
    \end{align}
where $x=(x_1,x_2)$ and $k=(k_1,k_2)\in \mathbb{Z}^2$,
and $\alpha>0$, $\beta>0$, we denote
  \begin{equation}
    \Vert (\partial')^\beta v\Vert_{L^2}^2
    =
    (2\pi)^2
    \sum_{k,n}|k|^{\beta}|\hat{v}_{k,n}|^2
        .
   \llabel{EQ121}
  \end{equation}

\begin{proof}[Proof of Lemma~\ref{L03}]
Let \eqref{EQ43} be the Fourier series expansion of~$v$,
and denote $\alpha=m\gamma\in\mathbb{N}$.
For simplicity, we only address the case when $\alpha$ is even, where
    \begin{align}
      \begin{split}
        \partial_z^\alpha\partial'^\beta v
        =
        c_0
        \sum_{k,n}(n\pi)^\alpha
        |k|^\beta
        \hat{v}_{k,n}
        e^{ikx}\sin(n\pi x_3)
        .
  \end{split}
   \label{EQ44}
    \end{align}
By Parseval's Theorem, we have
    \begin{align}
        \Vert \partial_z^\alpha\partial'^\beta v\Vert ^2_{L^2}
        =
	c_1
        \sum_{k,n}|k|^{2\beta}n^{2\alpha}
        |\hat{v}_{k,n}|^2
        \label{EQ45}
    \end{align}
and
    \begin{align}
        \Vert \partial_z^{\frac{\alpha}{\gamma}}v\Vert_{L^2}^2
        &
        =
	c_2
        \sum_{k,n}n^{\frac{2\alpha}{\gamma}}|\hat{v}_{k,n}|^2
        ,
        \label{EQ46}
    \end{align}
where $c_0$, $c_1$, and $c_2$ are explicit constants depending on $m$ and~$\gamma$.
Similarly, the analogous expression for the horizontal derivatives reads
    \begin{align}
        \Vert \partial'^{\frac{\beta}{1-\gamma}}v\Vert^2_{L^2}
        &
        =
	c_3
        \sum_{k,n}|k|^{\frac{2\beta}{1-\gamma}}|\hat{v}_{k,n}|^2
        ,
        \label{EQ47}
    \end{align}
where $c_3$ depends on $\beta$ and~$\gamma$.
By H\"older's inequality, together with
\eqref{EQ45}, \eqref{EQ46}, and \eqref{EQ47},
we obtain
    \begin{align}
    \begin{split}
        \Vert \partial_z^\alpha\partial'^\beta v\Vert^2_{L^2}
        &\lec
        \sum_{k,n}n^{2\alpha}|k|^{2\beta}
        |\hat{v}_{k,n}|^2
        =
        \sum_{k,n}n^{2\alpha}|\hat{v}_{k,n}|^{2\gamma}
        \cdot
        |k|^{2\beta}|\hat{v}_{k,n}|^{2(1-\gamma)}
        \\&\lesssim
        \sum_{k,n}(n^{\frac{2\alpha}{\gamma}}
        |\hat{v}_{k,n}|^2)^\gamma
        \sum_{k,n}(|k|^{\frac{2\beta}{1-\gamma}}
        |\hat{v}_{k,n}|^2)^{1-\gamma}
        \lesssim
        \Vert \partial_z^{m}v\Vert^{2\gamma}_{L^2}
        \Vert \partial'^{\frac{\beta}{1-\gamma}}v\Vert ^{2(1-\gamma)}_{L^2}
        .
    \end{split}
    \llabel{EQ48}
    \end{align}
If $\alpha$ is odd, then we only need to replace $\sin (n\pi x_3)$
with $\cos (n\pi x_3)$ in~\eqref{EQ44}, but otherwise the proof is the same.
\end{proof}
\colb

\cole
\begin{Lemma}
\label{L04}
Let $v(x_1,x_2,z)$ be a smooth function satisfying our boundary conditions \eqref{EQ20} and~\eqref{EQ23}. Then
\begin{align}
    \Vert v\Vert_{L^\infty}
    \lesssim
    \Vert v\Vert_{L^2}^\frac{1}{4}
    \Vert \partial_3v\Vert_{L^2}^\frac{1}{4}
    \Vert \partial''v\Vert_{L^2}^\frac{1}{4}
    \Vert \partial_3\partial''v\Vert_{L^2}^\frac{1}{4}
    +\lot
    \label{EQ151}
\end{align}
In the case where the left-hand side is also taken in an anisotropic
norm, we have
\begin{align}
    \left\Vert \Vert v \Vert_{L^\infty_z}\right\Vert_{L^4_\HH}
    \lesssim
    \Vert (\partial')^\frac{1}{2}v\Vert_{L^2}^\frac{1}{2}
    \Vert \partial_3(\partial')^\frac{1}{2}v\Vert_{L^2}^\frac{1}{2}
    +\lot
    \label{EQ152}
\end{align}
\end{Lemma}
\colb

We note that the inequality~\eqref{EQ151}, when the lower-order terms are included, reads
\begin{align}
    \begin{split}
    \Vert v\Vert_{L^\infty}
    &\lesssim
    (\Vert v\Vert_{L^2}^\frac{1}{4}\Vert \partial''v\Vert_{L^2}^\frac{1}{4}+\Vert v\Vert_{L^2}^\frac{1}{2})
    (\Vert\partial_3v\Vert_{L^2}^\frac{1}{4}\Vert\partial_3\partial''v\Vert_{L^2}^\frac{1}{4}+\Vert\partial_3v\Vert_{L^2}^\frac{1}{2})
    .
   \label{EQ136}
    \end{split}
\end{align}
Similar adjustments can be made for \eqref{EQ152}; however, the proof shows that the lower-order terms are of lower order than the leading ones and can be safely neglected.
From here on we simply neglect these lower order terms.

\begin{proof}[Proof of Lemma~\ref{L04}]
Since $v=0$ on $\Gamma_0 \cup \Gamma_1$, we obtain
\begin{align}
    \begin{split}
        &v^2(x_1,x_2,z,t)=v^2(x_1,x_2,z,t)-v^2(x_1,x_2,0,t)
        \\&\indeq
        =2\int_0^zv(x_1,x_2,z',t)\partial_3v(x_1,x_2,z',t)\,dz'\lesssim
        \Vert v\Vert_{L^2_z}^\frac{1}{2}
        \Vert\partial_3v\Vert_{L^2_z}^\frac{1}{2}
        ,
   \llabel{EQ124}
    \end{split}
\end{align}
for $(x_1,x_2,z,t)\in \mathbb{T}^2\times[0,1]\times[0,T]$. Therefore, using Agmon's inequality, we have
\begin{align}
    \begin{split}
        &\Vert v\Vert_{L^\infty}
        =\left\Vert\Vert v\Vert_{L^\infty_z}\right\Vert_{L^\infty_\HH}
        \lesssim 
        \left\Vert\Vert v\Vert_{L^2_z}^\frac{1}{2}
        \Vert\partial_3v\Vert_{L^2_z}^\frac{1}{2}
        \right\Vert_{L^\infty_\HH}
        \lesssim
        \left\Vert\Vert v\Vert_{L^2_z}
        \right\Vert_{L^\infty_\HH}^\frac{1}{2}
        \left\Vert
        \Vert\partial_3v\Vert_{L^2_z}
        \right\Vert_{L^\infty_\HH}^\frac{1}{2}
        \\&\indeq
        \lesssim
        \left\Vert\Vert v\Vert_{L^\infty_\HH}
        \right\Vert_{L^2_z}^\frac{1}{2}
        \left\Vert
        \Vert\partial_3v\Vert_{L^\infty_\HH}
        \right\Vert_{L^2_z}^\frac{1}{2}
        \lesssim
        \left\Vert\Vert v\Vert_{L^2_\HH}^\frac{1}{2}
        \Vert \partial''v\Vert_{L^2_\HH}^\frac{1}{2}
        +\Vert v\Vert_{L^2_\HH}
        \right\Vert_{L^2_z}^\frac{1}{2}
        \left\Vert
        \Vert \partial_3v\Vert_{L^2_\HH}^\frac{1}{2}
        \Vert \partial_3\partial''v\Vert_{L^2_\HH}^\frac{1}{2}
        +\Vert\partial_3v\Vert_{L^2_\HH}
        \right\Vert_{L^2_z}^\frac{1}{2}
        .
   \llabel{EQ125}
    \end{split}
\end{align}
Applying Minkowski’s inequality together with Hölder’s inequality, we deduce
\begin{align}
    \Vert v\Vert_{L^\infty}
    \lesssim
    \Vert v\Vert_{L^2}^\frac{1}{4}
    \Vert \partial_3v\Vert_{L^2}^\frac{1}{4}
    \Vert \partial''v\Vert_{L^2}^\frac{1}{4}
    \Vert \partial_3\partial''v\Vert_{L^2}^\frac{1}{4}
    +\lot
    \llabel{EQ155}
\end{align}
In a similar manner, we obtain for $\left\Vert \Vert v \Vert_{L^\infty_z}\right\Vert_{L^4_\HH}$ that
\begin{align}
    \left\Vert \Vert v \Vert_{L^\infty_z}\right\Vert_{L^4_\HH}
    \lesssim
    \left(\Vert (\partial')^\frac{1}{2}v\Vert_{L^2}^\frac{1}{2}
    +
    \Vert v\Vert_{L^2}^\frac{1}{2}
    \right)
    \left(\Vert \partial_3(\partial')^\frac{1}{2}v\Vert_{L^2}^\frac{1}{2}
    +
    \Vert \partial_3v\Vert_{L^2}^\frac{1}{2}
    \right)
    \lesssim
    \Vert (\partial')^\frac{1}{2}v\Vert_{L^2}^\frac{1}{2}
    \Vert \partial_3(\partial')^\frac{1}{2}v\Vert_{L^2}^\frac{1}{2}
    +\lot
    ,
    \llabel{EQ156}
\end{align}
completing the proof.
\end{proof}

\startnewsection{Tangential estimates}{sec13}
Throughout this paper, repeated Latin indices are summed over from~1 to 3, while repeated Greek indices go from~1 to 2. We denote by $\partial'$ the horizontal derivative in the $x_1$ or $x_2$ direction. Similarly, $\partial''$ denotes the double horizontal derivative. Throughout, let 
$\bard'=(-\Delta_{\HH}+I)^\frac{1}{2}$, where $-\Delta_\HH$ denotes the nonnegative Laplacian, so its Fourier symbol is~$|\xi|^2$. On $\mathbb{T}^2$, $\bard'$ is the Fourier multiplier with symbol 
$(|\xi|^2+1)^\frac{1}{2}$.

Before commencing with the tangential estimates,
we introduce our polynomial notation. Throughout the remainder of the paper, $\mathcal{P}$ denotes a generic positive polynomial  of the variable
    \begin{align}
        \begin{split}
            \sum_{\alpha=1,2}
	    (\|\partial''v_\alpha\|_{L^2}^2
            +\|\partial'v_\alpha\|^2
            +\|v_\alpha\|^2
            +\|\partial_tv_\alpha\|^2
	    )
            +\|h\|_{H^4}^2
            +\|h_t\|_{H^2}^2
            +\|h_{tt}\|_{L^2}^2
            .
        \end{split}
        \llabel{EQ55}
    \end{align}
Also, we use $\PPs$ to denote a positive polynomial tending
to zero as $t\to0$ depending on the same variable as above. 
The notation $\mathcal{P}_0$ refers to the evaluation of this polynomial at $t=0$, i.e., when all time–dependent variables are taken at their initial values.

\cole%%%out
\begin{Lemma}
\label{L06}
Under the assumptions of Theorem~\ref{T01}
and with $\epsilon\in(0,1/2]$,
we have
the tangential estimate
\begin{align}
    \begin{split}
        &\frac{1}{2}\frac{d}{dt}\int J  \partial''v_{\alpha}\partial''v_{\alpha}
        +\int \nablah\partial''v_\alpha\nablah\partial''v_\alpha
        +\int\partial_3\partial''v_\alpha\partial_3\partial''v_\alpha
        \\&\indeq
        +\frac{d}{dt}\int (\partial'' \partial_th)^2
        +\frac{d}{dt}\int (\partial''\Delta_{\HH}h)^2
	\\&
        \lesssim 
        \mathcal{P}
        +\epsilon(\Vert v_t\Vert_{H^1}^2+\Vert v\Vert_{H^3}^2)
        +\Vert p\Vert_{H^2}\mathcal{P}
        +\Vert p\Vert_{H^2}\Vert v\Vert_{H^2}\PPs
        .
    \end{split}
    \label{EQ64}
\end{align}
\end{Lemma}
\colb%%%out

\begin{proof}[Proof of Lemma~\ref{L06}]
First, we write \eqref{EQ18}$_1$ in the coordinate form,
\begin{align}
  \begin{split}
    &
    \partial_t v_\alpha
    -\Delta_{\HH}v_\alpha
    -a_{kl}\partial_k(a_{3l}\partial_3v_\alpha)
    -a_{3\beta}\partial_{3\beta} v_\alpha
    +v_\gamma a_{j\gamma}\partial_jv_\alpha
    \\&\indeq
    +\frac{1}{\partial_3\phi}(w-\phi_t)\partial_3v_\alpha
    +a_{k\alpha}\partial_kp=0
    ,
  \end{split}
   \label{EQ49}
\end{align}
for $\alpha=1,2$.
We multiply both sides of the equation~\eqref{EQ49} with $J=\partial_3 \phi$, which leads to
  \begin{align}
  \begin{split}
      &
      J\partial_t v_\alpha
      -J\Delta_{\HH}v_\alpha
      -2b_{3\beta}\partial_{3\beta} v_\alpha
      -\frac{1+b_{31}^2+b_{32}^2}{J}\partial^2_3v_\alpha
      -b_{kl}\partial_k a_{3l}\partial_3 v_\alpha
      \\&\indeq
      +v_\gamma b_{j\gamma}\partial_jv_\alpha
      +(w-\phi_t)\partial_3v_\alpha+b_{k\alpha}\partial_k p=0
     .  
  \end{split}
   \label{EQ50}
  \end{align}
We apply $\partial''$ to the equation \eqref{EQ50} and then take the  $L^2$-inner product with $\partial''v_\alpha$. Summing $\alpha$ from 1 to 2, the energy estimate for $\partial''v_\alpha$ may be written as
\begin{align}
    \begin{split}
        &\frac{1}{2}\frac{d}{dt}\int J\partial''v_{\alpha}\partial''v_{\alpha}
        +\int J\nablah\partial''v_\alpha
        \nablah\partial''v_\alpha
        +\int\frac{1+b_{31}^2+b_{32}^2}{J}
        \partial_3\partial''v_\alpha
        \partial_3\partial''v_\alpha
	\\&\indeq
	=\int\left(J\partial''\partial_tv_\alpha
        -\partial''(J\partial_tv_\alpha)\right)\partial''v_\alpha 
        -\int\nablah J\nablah \partial''v\partial''v_\alpha
        +\int \left(\partial''(J\Delta_{\HH}v_\alpha)
        -J\partial''\Delta_{\HH}v_\alpha\right)\partial''v_\alpha
	\\&\indeq\indeq
        -\int\partial_3\left(\frac{1+b_{31}^2+b_{32}^2}{J}\right)\partial''\partial_3 v_\alpha\partial''v_\alpha
        \\&\indeq\indeq
        +\int\left(\partial''\left(\frac{1+b_{31}^2+b_{32}^2}{J}\partial_{33}v_\alpha\right)
        -\frac{1+b_{31}^2+b_{32}^2}{J}\partial''\partial_{33}v_\alpha\right)\partial''v_\alpha
	\\&\indeq\indeq
        +\int\partial''\left(2b_{3\beta}\partial_{3\beta} v_\alpha
        +b_{kl}
        \partial_ka_{3l}\partial_3v_\alpha\right)\partial''v_\alpha
	\\&\indeq\indeq
	-\int\partial''\left(v_1b_{j1}\partial_jv_\alpha
        +v_2b_{j2}\partial_jv_\alpha
        +(w-\phi_t)\partial_3v_\alpha\right)\partial''v_\alpha
	\\&\indeq\indeq
        +\int \left(v_1b_{j1}\partial_j\partial''v_\alpha
        +v_2b_{j2}\partial_j\partial''v_\alpha
        +(w-\phi_t)\partial_3\partial''v_\alpha\right)\partial''v_\alpha
	\\&\indeq\indeq
	-\int (\partial''(b_{k\alpha}\partial_kp)
        -b_{k\alpha}\partial''\partial_kp)\partial''v_\alpha
        -\int (\partial''(b_{k\alpha}\partial_kv_\alpha)
        -b_{k\alpha}\partial''\partial_kv_\alpha)\partial''p
        -\int_{\mathbb{T}^2}\partial''w\partial''p
	.
    \end{split}
    \label{EQ51}
\end{align}
Note that the differentiation of the product
$\frac{1}{2}J\partial''v_\alpha\partial''v_\alpha$
with respect to $t$
generates a term 
$\frac{1}{2}J_t\partial''v_\alpha\partial''v_\alpha$.
Furthermore, observing that integration by parts in combination with Piola's identity and divergence free condition yields 
\begin{align}
    \begin{split}
    &\int \left(v_1b_{j1}\partial_j\partial''v_\alpha
    +v_2b_{j2}\partial_j\partial''v_\alpha
    +(w-\phi_t)\partial_3\partial''v_\alpha\right)\partial''v_\alpha
    \\&
    =-\frac{1}{2}\int (b_{j1}\partial_jv_1+b_{j2}\partial_jv_2+\partial_3w)\partial''v_\alpha\partial''v_\alpha
    -\frac{1}{2}\int (\partial_jb_{j1}v_1+\partial_jb_{j2}v_2)\partial''v_\alpha\partial''v_\alpha
    \\&\indeq\indeq
    +\frac{1}{2}\int \partial_3\phi_t\partial''v_\alpha\partial''v_\alpha
    =\frac{1}{2}J_t\partial''v_\alpha\partial''v_\alpha
    .
    \end{split}
\end{align}
Here we explain the appearance of the last three terms in \eqref{EQ51},
which result from the term
$\int\partial''(b_{k\alpha}\partial_k p)\partial''v_\alpha$.
Namely, first 
\begin{align}
    \begin{split}
        &\int\partial''(b_{k\alpha}\partial_k p)\partial''v_\alpha
        =\int \left(\partial''(b_{k\alpha}\partial_kp)
        -b_{k\alpha}\partial''\partial_kp\right)\partial''v_\alpha
        -\int \partial_k(b_{k\alpha}\partial'' v_\alpha)\partial'' p
        .
    \end{split}
    \label{EQ52}
\end{align}
The last term in \eqref{EQ52} may be rewritten as
\begin{align}
\begin{split}
        &-\int \partial_k(b_{k\alpha}\partial''v_\alpha)\partial'' p
        =
        -\int b_{k\alpha}\partial'' \partial_k v_\alpha\partial'' p
        \\&\indeq
        =\int \left(\partial''(b_{k\alpha}\partial_kv_\alpha)
        -b_{k\alpha}\partial''\partial_kv_\alpha\right)\partial'' p
        +\int \partial''\partial_3 w\partial''p
     ,
    \end{split}
    \llabel{EQ52b}
\end{align}
and applying the Fundamental Theorem of Calculus to the last term, we get
\begin{align}
\begin{split}
    &
\int\partial''(b_{k\alpha}\partial_k p)\partial''v_\alpha
\\&\indeq
	=-\int (\partial''(b_{k\alpha}\partial_kp)
        -b_{k\alpha}\partial''\partial_kp)\partial''v_\alpha
        -\int (\partial''(b_{k\alpha}\partial_kv)
        -b_{k\alpha}\partial''\partial_kv_\alpha)\partial''p
        -\int_{\mathbb{T}^2}\partial''w\partial''p
	,
    \end{split}
    \llabel{EQ52c}
\end{align}
explaining the last three terms in~\eqref{EQ51}.
The purpose of the above simplifications (integration by parts, Piola’s identity, and the incompressibility condition) is to isolate the term whose differentiation generates excessively high–order derivatives. These contributions cannot be absorbed by the dissipation terms, and are
canceled by a high-order term in the plate equation.

 We now  apply $\partial''$ to the equation \eqref{EQ07}, then take the  $L^2$-inner product with $\partial''\partial_th$, and then combine the resulting equality with the equation \eqref{EQ08} to obtain
\begin{align}
    \frac{1}{2}\frac{d}{dt}\int (\partial'' \partial_th)^2
    +\frac{1}{2}\frac{d}{dt}\int (\partial''\Delta_{\HH}h)^2
    =\int_{\mathbb{T}^2}\partial''p\partial''w
    .
    \label{EQ53}
\end{align}
Summing \eqref{EQ51} and~\eqref{EQ53}, we obtain our initial tangential estimate
\begin{align}
    \begin{split}
       &\frac{1}{2}\frac{d}{dt}\int J \partial''v_{\alpha}\partial''v_{\alpha}
        +\int J\nablah\partial''v_\alpha\nablah\partial''v_\alpha
        +\int\frac{1+b_{31}^2+b_{32}^2}{J}\partial_3\partial''v_\alpha\partial_3\partial''v_\alpha
        \\&\indeq
        +\frac{1}{2}\frac{d}{dt}\int (\partial'' \partial_th)^2
        +\frac{1}{2}\frac{d}{dt}\int (\partial''\Delta_{\HH}h)^2
	\\&\indeq
	=\int\left(J\partial''\partial_tv_\alpha
        -\partial''(J\partial_tv_\alpha)\right)\partial''v_\alpha 
        -\int\nablah J\nablah\partial''v_\alpha\partial''v_\alpha
        +\int \left(\partial''(J\Delta_{\HH}v_\alpha)
        -J\partial''\Delta_{\HH}v_\alpha\right)\partial''v_\alpha
	\\&\indeq\indeq
        -\int\partial_3\left(\frac{1+b_{31}^2+b_{32}^2}{J}\right)\partial''\partial_3 v_\alpha\partial''v_\alpha
        \\&\indeq\indeq
        +\int\left(\partial''\left(\frac{1+b_{31}^2+b_{32}^2}{J}\partial_{33}v_\alpha\right)
        -\frac{1+b_{31}^2+b_{32}^2}{J}\partial''\partial_{33}v_\alpha\right)\partial''v_\alpha
	\\&\indeq\indeq
        +\int\partial''\left(2b_{3\beta}\partial_{3\beta} v
        +b_{kl}
        \partial_ka_{3l}\partial_3v_\alpha\right)\partial''v_\alpha
	\\&\indeq\indeq
	-\int\partial''\left(v_1b_{j1}\partial_jv_\alpha
        +v_2b_{j2}\partial_jv_\alpha
        +(w-\phi_t)\partial_3v_\alpha\right)\partial''v_\alpha
	\\&\indeq\indeq
        +\int \left(v_1b_{j1}\partial_j\partial''v_\alpha
        +v_2b_{j2}\partial_j\partial''v_\alpha
        +(w-\phi_t)\partial_3\partial''v_\alpha\right)\partial''v_\alpha
	\\&\indeq\indeq
        -\int (\partial''(b_{k\alpha}\partial_kv_\alpha)
        -b_{k\alpha}\partial''\partial_kv_\alpha)\partial''p
        -\int (\partial''(b_{k\alpha}\partial_kp)
        -b_{k\alpha}\partial''\partial_kp)\partial''v_\alpha
        \\&\indeq
        =I_1+I_2+\cdots+I_{10}
	.
    \end{split}
    \label{EQ54}
\end{align}
We estimate the terms above in order, starting with
\begin{align}
\begin{split}
    I_1
    &=\int\left(J\partial''\partial_tv_\alpha
    -\partial''(J\partial_tv_\alpha)\right)\partial''v_\alpha
    \\&
    \lesssim\|\partial'J\|_{L^\infty}\|\partial'v_t\|_{L^2}\|\partial''v\|_{L^2}
    +\|\partial''J\|_{L^6}\|v_t\|_{L^3}\|\partial''v\|_{L^2}
    \\&
    \lesssim\|h\|_{H^4}\|v_t\|_{H^1}\|\partial''v\|_{L^2}+\|h\|_{H^{3.5}}\|v_t\|^{\frac{1}{2}}_{L^2}\|v_t\|^{\frac{1}{2}}_{H^1}
    \|\partial''v\|_{L^2}
    \\&\lesssim
    \mathcal{P}
   +\epsilon(\|v_t\|_{H^1}^2+\|v\|_{H^3}^2)
   .
   \label{EQ551}
\end{split}
\end{align}
The sum of the second and third terms on the right-hand side of \eqref{EQ54} may be estimated by Hölder’s inequality as
\begin{align}
\begin{split}
    I_2+I_3
    &=
    -\int\nablah J\nablah\partial''v_\alpha\partial''v_\alpha
    +\int \left(\partial''(J\Delta_{\HH}v_\alpha)
    -J\partial''\Delta_{\HH}v_\alpha\right)\partial''v_\alpha
    \\&
    \lesssim\|\nablah J\|_{L^\infty}\|\nablah\partial''v\|_{L^2}\|\partial''v\|_{L^2}
    +\|\partial''J\|_{L^3}\|\Delta_\HH v\|_{L^6}\|\partial''v\|_{L^2}
    \\&\indeq
    +\|\partial'J\|_{L^\infty}\|\partial'\Delta_\HH v\|_{L^2}\|\partial''v\|_{L^2}
    \\&
    \lesssim\|h\|_{H^4}\|\nablah\partial''v\|\|\partial''v\|_{L^2}
    +\|h\|_{H^{3}}\|\Delta_\HH v\|_{H^1}\|\partial''v\|_{L^2}
    +\|h\|_{H^4}\|\partial'\Delta_\HH v\|_{L^2}\|\partial''v\|_{L^2}
    \\&
    \lesssim \mathcal{P}+\epsilon \|v\|_{H^3}^2
    .
\end{split}
\label{EQ552}
\end{align}
The fourth term can again be estimated, in the same manner as the first two terms, by applying H\"older’s inequality in the 
$L^\infty$-$L^2$-$L^2$ form together with Lemma~\ref{L01},
leading to
  \begin{equation*}
   I_4\lec     \mathcal{P}+\epsilon \|v\|_{H^3}^2
   .
   \llabel{EQ112}
  \end{equation*}
However, for $I_5$, such a technique fails due to the higher boundary regularity required by the Sobolev embedding into~$L^\infty$.
Therefore, we resort to anisotropic estimates and write
\begin{align}
    \begin{split}
    I_5
    &=\int\left(\partial''\left(\frac{1+b_{31}^2+b_{32}^2}{J}\partial_{33}v_\alpha\right)
        -\frac{1+b_{31}^2+b_{32}^2}{J}\partial''\partial_{33}v_\alpha\right)\partial''v_\alpha
    \\&
    \lesssim
    \left\|\|\partial''J^{-1}+\partial''(b_{33}(a_{31}^2+a_{32}^2))\|_{L^{\infty}_z}\right\|_{L^4_\HH}\left\|
    \|\partial_{33}v\|_{L^2_z}\right\|_{L^4_\HH}\left\|
    \|\partial''v\|_{L^2_z}\right\|_{L^2_\HH}
    \\&\indeq
    +\|\partial' J^{-1}+\partial'(b_{33}(a_{31}^2+a_{32}^2))\|_{L^\infty}
    \|\partial_{33}\partial'v\|_{L^2}
    \|\partial''v\|_{L^2}
    \\&\lesssim 
    P(\|h\|_{H^4})\|\partial_{33}v\|_{L^2}^\frac{1}{2}\|\partial_{33}\partial'v\|_{L^2}^\frac{1}{2}\|\partial_i^2v\|_{L^2}
    +P(\|h\|_{H^4})\|\partial_3^2\partial_iv\|_{L^2}\|\partial_i^2v\|_{L^2}
    \\&\lesssim
    \mathcal{P}+\epsilon \|v\|_{H^3}^2
    .
    \end{split}
    \label{EQ553}
\end{align}
To estimate $I_6$, we apply integration by parts by writing
\begin{align}
\begin{split}
        I_6&=\int\partial''\left(2b_{3\beta}\partial_{3\beta} v_\alpha
        +b_{kl}
        \partial_ka_{3l}\partial_3v_\alpha\right)\partial''v_\alpha
        \\&=\int 2\partial''b_{3\beta}\partial_{3\beta} v_\alpha\partial''v_\alpha
        +4\int\partial'b_{3\beta}\partial'\partial_{3\beta} v_\alpha\partial''v_\alpha
        -2\int\partial'\partial_{3\beta} v_\alpha\partial'(b_{3\beta}\partial''v_\alpha)
        \\&\indeq
        +\int \partial''(b_{kl}\partial_k a_{3l})\partial_3v_\alpha\partial''v_\alpha
        +2\int\partial'(b_{kl}\partial_ka_{3l})\partial'\partial_3v_\alpha\partial''v_\alpha
        +\int b_{kl}\partial_ka_{3l}\partial''\partial_3v_\alpha\partial''v_\alpha
        \\&
        \lesssim
        \left\|\Vert\partial''b_{3\beta}\|_{L^\infty_z}\right\|_{L^4_\HH}
        \left\|\partial_{3\beta} v\|_{L^2_z}\right\|_{L^4_\HH}
        \|\partial''v\|_{L^2}
        +\|\partial'b_{3\beta}\|_{L^\infty}
        \|\partial'\partial_{3\beta} v\|_{L^2}
        \|\partial''v\|_{L^2}
        \\&\indeq
        +\|b_{3\beta}\|_{L^\infty}\|\partial'\partial_{3\beta} v\|_{L^2}\|\partial'''v\|_{L^2}
        +\left\|\|\partial''(b_{kl}
        \partial_ka_{3l})\|_{L^\infty_z}\right\|_{L^4_\HH}
        \left\|\|\partial_3 v\|_{L^2_z}\right\|_{L^4_\HH}
        \|\partial''v\|_{L^2}
       \\&\indeq
        +\|\partial'(b_{kl}
        \partial_ka_{3l})\|_{L^\infty}
        \|\partial'\partial_3 v\|_{L^2}
        \|\partial''v\|_{L^2}
        +\|(b_{kl}
        \partial_ka_{3l})\|_{L^\infty}
        \|\partial''\partial_3 v\|_{L^2}\|\partial''v\|_{L^2}
        \\&
        \lesssim
        \mathcal{P}+\epsilon \|v\|_{H^3}^2
        ,
        \end{split}
        \label{EQ554}
\end{align}
where we apply Lemma~\ref{L04}. Collecting the above estimates \eqref{EQ551}, \eqref{EQ552}, \eqref{EQ553}, and \eqref{EQ554}, we arrive at
  \begin{equation}
   \sum_{k=1}^6 I_k
   \leq
   \mathcal{P}
   +(\epsilon+\PPs)(\|v_t\|_{H^1}^2+\|v\|_{H^3}^2)
   .
   \llabel{EQ56}
  \end{equation}
The sum
\begin{align}
  \begin{split}
      I_7 + I_8
        &=
        -\int\partial''\left(v_1b_{j1}\partial_jv_\alpha
        +v_2b_{j2}\partial_jv_\alpha
        +(w-\phi_t)\partial_3v_\alpha\right)\partial''v_\alpha
    \\&\indeq
        +\int \left(v_1b_{j1}\partial_j\partial''v_\alpha
        +v_2b_{j2}\partial_j\partial''v_\alpha
        +(w-\phi_t)\partial_3\partial''v_\alpha\right)\partial''v_\alpha
  ,
  \end{split}  
  \label{EQ57}
\end{align}
which makes the highest-order contribution,
cannot be absorbed by the dissipation terms.
To estimate it, we employ Lemma~\ref{L04}.
To proceed with a treatment of the sum in \eqref{EQ57},
we first rewrite it as
\begin{align}
  \begin{split}
      I_7 + I_8
        &=
        -\int\partial''b_{j\beta}v_{\beta}
    \partial_j v_\alpha\partial''v_\alpha
    -2    \int\partial'b_{j\beta}\partial'v_{\beta}
    \partial_j v_\alpha\partial''v_\alpha
   \\&\indeq
    -2\int\partial'b_{j\beta}v_{\beta}
    \partial'\partial_j v_\alpha\partial''v_\alpha
    -2\int b_{j\beta}\partial''v_{\beta}
    \partial_j v_\alpha\partial''v_\alpha
    -2\int b_{j\beta}\partial'v_{\beta}
    \partial_j \partial'v_\alpha\partial''v_\alpha
   \\&\indeq
    - \int \partial''w\partial_3v_\alpha\partial''v_\alpha
    -2\int \partial'w\partial_3\partial'v_\alpha\partial''v_\alpha
    +    \int(
        \partial''\phi_t\partial_3v_\alpha\partial''v_\alpha
        +2\partial'\phi_t\partial'\partial_3v_\alpha\partial''v_\alpha
          )
   \\&
   = J_1 + \ldots + J_8
   .
  \end{split}  
   \label{EQ27}
\end{align}
Using the anisotropic estimate~\eqref{EQ151}, we have
\begin{align}
\begin{split}
    J_1
    &=\int\partial''b_{j\beta}v_{\beta}
    \partial_j v_\alpha\partial''v_\alpha
    \\&
    \lesssim
    \left\Vert \Vert \partial''b_{j\beta}\Vert_{L^\infty_z}\right\Vert_{L^2_{\HH}}
    \Vert v_\beta\Vert_{L^\infty}
    \left\Vert \Vert \partial'v_\alpha\Vert_{L^2_z}\Vert \right\Vert_{L^\infty_{\HH}}
    \Vert \partial''v\Vert_{L^2}
    \\&\indeq
    +\left\Vert \Vert \partial''b_{j\beta}\Vert_{L^\infty_z}\right\Vert_{L^2_{\HH}}\left\Vert 
    \Vert v_{\beta}\Vert_{L^\infty_z}\right\Vert_{L^\infty_{\HH}}
    \left\Vert \Vert \partial_3v_\alpha\Vert_{L^2_z}\right\Vert_{L^4_{\HH}}
    \left\Vert \Vert \partial''v_\alpha\Vert_{L^2_z}\right\Vert_{L^4_{\HH}}
    \\&
    \lesssim 
    \Vert h\Vert_{H^4}
    \left(
    \Vert v\Vert ^\frac{1}{4}_{L^2}
    \Vert \bard''v\Vert ^\frac{1}{4}_{L^2}
    \Vert \partial_3v\Vert ^\frac{1}{4}_{L^2}
    \Vert \partial_3\bard''v\Vert ^\frac{1}{4}_{L^2}
    \right)\left(
    \Vert \partial'v\Vert_{L^2}^\frac{1}{2}
    \Vert \partial'''v\Vert_{L^2}^\frac{1}{2}
    \right)
    \Vert \partial''v\Vert_{L^2_{\HH}}
    \\&\indeq
    +\Vert h\Vert_{H^4}
    \left(\Vert v\Vert ^\frac{1}{4}_{L^2}
    \Vert \bard''v\Vert ^\frac{1}{4}_{L^2}
    \Vert \partial_3v\Vert ^\frac{1}{4}_{L^2}
    \Vert \partial_3\bard''v\Vert ^\frac{1}{4}_{L^2}\right)
    \left(
    \Vert \partial_3v\Vert_{L^2}^\frac{1}{2}
    \Vert \partial_3\partial'v\Vert ^\frac{1}{2}_{L^2}
    \right)\left(
    \Vert \partial''v\Vert_{L^2}^\frac{1}{2}
    \Vert \partial'''v\Vert_{L^2}^\frac{1}{2}
    \right)
    ,
    \end{split}
   \llabel{EQ110}
    \end{align}
from where
\begin{align}
\begin{split}
    J_1
    &
    \lesssim
    \Vert h\Vert_{H^4}
    \Vert v\Vert ^\frac{1}{4}_{L^2}
    \Vert \bard''v\Vert ^\frac{5}{4}_{L^2}
    \Vert \partial_3v\Vert ^\frac{1}{4}_{L^2}
    \Vert \partial_3\bard''v\Vert ^\frac{1}{4}_{L^2}
    \Vert \partial'v\Vert_{L^2}^\frac{1}{2}
    \Vert \partial'''v\Vert_{L^2}^\frac{1}{2}
    \\&\indeq
    +\Vert h\Vert_{H^4}
    \Vert v\Vert ^\frac{5}{8}_{L^2}
    \Vert \bard''v\Vert ^\frac{3}{4}_{L^2}
    \Vert \partial_3\bard''v\Vert ^\frac{1}{4}_{L^2}
    \Vert \partial_{33}v\Vert_{L^2}^\frac{3}{8}
    \Vert \partial_3\partial'v\Vert ^\frac{1}{2}_{L^2}
    \Vert \partial'''v\Vert_{L^2}^\frac{1}{2}
    \\&
    \lesssim
    \mathcal{P}
    +\epsilon\Vert v\Vert_{H^3}^2
    ,
\end{split}
\llabel{EQ59}
\end{align}
using Young's inequality in the final step.
The estimates for the next four terms follow from a common argument based on anisotropic estimates.
First, we have
\begin{align}
    \begin{split}
    J_{2}
    &=
    -2\int\partial'b_{j\beta}\partial'v_{\beta}
    \partial_j v_\alpha\partial''v_\alpha
    \\&
    \lesssim 
    \PPs\left\Vert \Vert \partial'v\Vert_{L^\infty_z}\right\Vert_{L^2_\HH}
    \left\Vert \Vert \partial'v\Vert_{L^2_z}\right\Vert_{L^\infty_\HH}
    \Vert \partial''v\Vert_{L^2}
    +\PPs\left\Vert \Vert \partial'v\Vert_{L^\infty_z}\right\Vert_{L^2_\HH}
    \left\Vert \Vert \partial_3v\Vert_{L^2_z}\right\Vert_{L^\infty_\HH}\Vert \partial''v\Vert_{L^2}
    \\&\lesssim 
    \PPs\Vert \partial'v\Vert_{L^2}\Vert \partial_3\partial'v\Vert_{L^2}^\frac{1}{2}
    \Vert \bard'''v\Vert ^\frac{1}{2}_{L^2}
    \Vert \partial''v\Vert_{L^2}
    \\&\indeq
    +\PPs\Vert \partial'v\Vert_{L^2}^\frac{1}{2}
    \Vert \partial_3\partial'v\Vert_{L^2}^\frac{1}{2}
    \Vert \partial_3v\Vert ^\frac{1}{2}_{L^2}
    \Vert \partial_3\bard''v\Vert_{L^2}^\frac{1}{2}
    \Vert \partial''v\Vert_{L^2}
    \\&\lesssim 
    \mathcal{P}
    +\epsilon\Vert v\Vert_{H^3}^2
    .
    \end{split}
   \llabel{EQ126}
   \end{align}
By repeating the same anisotropic estimate for the other three terms, we obtain 
\begin{align}
\begin{split}
    J_{3}
    &=\int\partial'b_{j\beta}v_{\beta}
    \partial'\partial_j v_\alpha\partial''v_\alpha
    \lesssim 
    \PPs\Vert v\Vert_{L^\infty}\Vert \partial'\nabla v\Vert_{L^2}\Vert \partial''v\Vert_{L^2}
    \\&
    \lesssim 
    \PPs\Vert v\Vert_{L^2}^\frac{1}{4}
    \Vert \bard''v\Vert ^\frac{5}{4}_{L^2}
    \Vert \partial_3v\Vert_{L^2}^\frac{1}{4}
    \Vert \partial_3\bard''v\Vert_{L^2}^\frac{1}{4}
    (\Vert \partial''v\Vert_{L^2}+\Vert \partial_3\partial'v\Vert_{L^2})
    \lesssim 
    \mathcal{P}
    +\epsilon\Vert v\Vert_{H^3}^2
    .
    \end{split}
   \llabel{EQ127}
\end{align}
For the next term, we write
\begin{align}
\begin{split}
    J_{4}&=\int b_{j\beta}\partial''v_{\beta}
    \partial_j v_\alpha\partial''v_\alpha
    \lesssim
    \PPs
    \left\Vert \Vert \partial''v\Vert_{L^\infty_z}\Vert \right\Vert_{L^2_\HH}
    \left\Vert \Vert \nabla v\Vert_{L^2_z}\right\Vert_{L^\infty_\HH}
    \Vert \partial''v\Vert_{L^2}
    \\&\lesssim 
    \PPs\Vert \partial''v\Vert_{H^1}\Vert v\Vert_{H^1}^\frac{1}{2}\Vert \partial''v\Vert ^\frac{3}{2}_{L^2}
    \lesssim 
    \mathcal{P}
    +\epsilon\Vert v\Vert_{H^3}^2
    ,
\end{split}
   \llabel{EQ128}\end{align}
while Lemma~\ref{L04} yields
\begin{align}
    \begin{split}
    J_{5}
    &=-2\int b_{j\beta}\partial'v_{\beta}
    \partial_j \partial'v_\alpha\partial''v_\alpha
    \\&\lesssim 
    \PPs\left\Vert \Vert \partial'v\Vert_{L^\infty_z}\right\Vert_{L^4_\HH}
    \left\Vert \Vert \partial''v\Vert_{L^2_z}\right\Vert_{L^4_\HH}
    \Vert \partial''v\Vert_{L^2}
    +\PPs\left\Vert \Vert \partial'v\Vert_{L^\infty_z}\right\Vert_{L^4_\HH}
    \left\Vert \Vert \partial_3\partial'v\Vert_{L^2_z}\right\Vert_{L^4_\HH}\Vert \partial''v\Vert_{L^2}
    \\&\lesssim 
    \PPs\Vert \partial'v\Vert_{L^2}^\frac{1}{4}
    \Vert \partial_3\partial'v\Vert_{L^2}^\frac{1}{4}
    \Vert \partial''v\Vert_{L^2}^\frac{1}{4}
    \Vert \partial_3\partial''v\Vert_{L^2}^\frac{1}{4}
    \Vert \partial''v\Vert_{L^2}^\frac{1}{2}
    \Vert \partial'''v\Vert ^\frac{1}{2}_{L^2}
    \Vert \partial''v\Vert_{L^2}
    \\&\indeq
    +\PPs\Vert \partial'v\Vert_{L^2}^\frac{1}{4}
    \Vert \partial_3\partial'v\Vert_{L^2}^\frac{1}{4}
    \Vert \partial''v\Vert_{L^2}^\frac{1}{4}
    \Vert \partial_3\partial''v\Vert_{L^2}^\frac{1}{4}
    \Vert \partial_3\partial'v\Vert_{L^2}^\frac{1}{2}
    \Vert \partial_3\partial''v\Vert ^\frac{1}{2}_{L^2}
    \Vert \partial''v\Vert_{L^2}
    \\&\lesssim
    \mathcal{P}
    +\epsilon\Vert v\Vert_{H^3}^2
    .
    \end{split}
   \llabel{EQ129}\end{align}
Next, we estimate 
  \begin{equation*}
   J_6 + J_7
    =
    -
    \int \partial''w\partial_3v_\alpha\partial''v_\alpha
    -2\int \partial'w\partial_3\partial'v_\alpha\partial''v_\alpha
  .   
  \end{equation*}
Using anisotropic analysis~\eqref{EQ152},
we have
\begin{align}
\begin{split}
    J_6 + J_7
    &
    =-\int \partial''w\partial_3v_\alpha\partial''v_\alpha
    -2\int \partial'w\partial_3\partial'v_\alpha\partial''v_\alpha
    \\&
    \lesssim 
    \left\Vert \Vert \partial''w\Vert_{L^\infty_z}\right\Vert_{L^2_{\HH}}
    \left\Vert \Vert \partial_3v_\alpha\Vert_{L^2_z}\right\Vert_{L^\infty_{\HH}}
    \Vert \partial''v_\alpha\Vert_{L^2}
    +\left\Vert \Vert \partial'w\Vert_{L^4_{\HH}}\right\Vert_{L^\infty_z}
    \left\Vert \Vert \partial_3\partial'v_\alpha\Vert_{L^4_{\HH}}\right\Vert_{L^2_z}
    \Vert \partial''v_\alpha\Vert_{L^2}
    \\&
    \lesssim 
    \Vert \partial''w\Vert ^\frac{1}{2}_{L^2}
    \Vert \partial''\partial_3w\Vert ^\frac{1}{2}_{L^2}
    \Vert \partial_3v_\alpha\Vert ^\frac{1}{2}_{L^2}
    \Vert \partial_3\bard''v_\alpha\Vert ^\frac{1}{2}_{L^2}
    \Vert \partial''v_\alpha\Vert_{L^2}
    \\&\indeq
    +\Vert \partial' w\Vert ^\frac{1}{4}_{L^2}
    \Vert \partial'' w\Vert ^\frac{1}{4}_{L^2}
    \Vert \partial''\partial_3w\Vert ^\frac{1}{4}_{L^2}
    \Vert \partial'\partial_3w\Vert_{L^2}^\frac{1}{4}
    \Vert \partial_3\partial'v_\alpha\Vert_{L^2}^\frac{1}{2}
    \Vert \partial_3\partial''v_\alpha\Vert_{L^2}^\frac{1}{2}
    \Vert \partial''v_\alpha\Vert_{L^2}
    \\&
    =J_{61}+J_{62}
    .
\end{split}
\llabel{EQ61}
\end{align}
By exploiting the incompressibility condition, 
$w$ can be represented in integral form,
\begin{align}
    \begin{split}
    \Vert \partial''w\Vert_{L^2}
    &=\left\Vert \int_0^z
    \partial''
    (b_{11}\partial_1 v_1
    +b_{22}\partial_2 v_2
    -b_{31}\partial_3 v_1
    -b_{32}\partial_3 v_2
    )d\bar{z}\right\Vert_{L^2_{H,z}}
    \\&
    \leq
    \left\Vert \int_0^1
    \partial''
    (b_{11}\partial_1 v_1
    +b_{22}\partial_2 v_2
    -b_{31}\partial_3 v_1
    -b_{32}\partial_3 v_2
    )d\bar{z}\right\Vert_{L^2_{H,z}}
    \\&
    \leq \left\Vert \Vert 
    \partial''
    (b_{11}\partial_1 v_1
    +b_{22}\partial_2 v_2
    -b_{31}\partial_3 v_1
    -b_{32}\partial_3 v_2
    )\Vert_{L^1_z}
    \right\Vert_{L^2_\HH}
    \\&
    \lesssim
    \Vert 
    \partial''
    (b_{11}\partial_1 v_1
    +b_{22}\partial_2 v_2
    -b_{31}\partial_3 v_1
    -b_{32}\partial_3 v_2
    )\Vert_{L^2_{H,z}}
    .
    \end{split}
\llabel{EQ130}
\end{align}
However, owing to the fact that $\partial_3w$
is naturally involved in the incompressibility condition \eqref{EQ19a},
$\partial''\partial_3w$ can be estimated by the same upper bound as $\partial''w$.
Therefore, using
the divergence-free condition and \eqref{EQ42}, we have
\begin{align}
\begin{split}
    J_{61}&\lesssim \Vert 
    \partial''
    (b_{11}\partial_1 v_1
    +b_{22}v_2
    -b_{31}\partial_3 v_1
    -b_{32}\partial_3 v_2
    )\Vert_{L^2}
    \\&\indeq
    \times
    \Vert \partial_3v_\alpha\Vert ^\frac{1}{2}_{L^2}
    \Vert \partial_3\bard''v_\alpha\Vert ^\frac{1}{2}_{L^2}
    \Vert \partial''v_\alpha\Vert_{L^2}
    \\&
    \lesssim
    \Bigl(
    \Vert b\Vert_{L^\infty}\Vert \partial'''v\Vert_{L^2}
    +\Vert \partial'b\Vert_{L^\infty}\Vert \partial''v\Vert_{L^2}
    +\left\Vert \Vert \partial''b\Vert_{L^\infty_z}\right\Vert_{L^2_{\HH}}
    \left\Vert \Vert \partial'v\Vert_{L^2_z}\right\Vert_{L^\infty_{\HH}}
    \\&\indeq
    +\left\Vert \Vert \partial''b\Vert_{L^\infty_z}\right\Vert_{L^2_{\HH}}
    \left\Vert \Vert \partial_3v\Vert_{L^2_z}\right\Vert_{L^\infty_{\HH}}
    +\Vert \partial'b\Vert_{L^\infty}
    \Vert \partial_3\partial'v\Vert_{L^2}
    +\Vert b\Vert_{L^\infty}
    \Vert \partial_3\partial''v\Vert_{L^2}
    \Bigr)
    \\&\indeq\indeq
    \times
    \Vert v\Vert_{L^2}^\frac{1}{4}
    \Vert \partial_{33}v\Vert_{L^2}^\frac{1}{4}
    \Vert \partial_3\bard''v\Vert ^\frac{1}{2}_{L^2}
    \Vert \partial''v\Vert_{L^2}
    \\&
    \lesssim
    \mathcal{P}
    +\epsilon\Vert v\Vert_{H^3}^2
    .
\end{split}
\llabel{EQ62}
\end{align}
Using analogous estimates, we also have
  \begin{align}
  \begin{split}
    J_{62}&
    =\Vert \partial' w\Vert ^\frac{1}{4}_{L^2}
    \Vert \partial'' w\Vert ^\frac{1}{4}_{L^2}
    \Vert \partial''\partial_3w\Vert ^\frac{1}{4}_{L^2}
    \Vert \partial'\partial_3w\Vert_{L^2}^\frac{1}{4}
    \Vert \partial_3\partial'v_\alpha\Vert_{L^2}^\frac{1}{2}
    \Vert \partial_3\partial''v_\alpha\Vert_{L^2}^\frac{1}{2}
    \Vert \partial''v_\alpha\Vert_{L^2}
    \\&\lesssim \Vert 
    \partial''
    (b_{11}\partial_1 v_1
    +b_{22}v_2
    -b_{31}\partial_3 v_1
    -b_{32}\partial_3 v_2
    )\Vert_{L^2}
    \\&\indeq
    \times
    \Vert \partial_3\partial'v_\alpha\Vert ^\frac{1}{2}_{L^2}
    \Vert \partial_3\bard''v_\alpha\Vert ^\frac{1}{2}_{L^2}
    \Vert \partial''v_\alpha\Vert_{L^2}
    \\&
    \lesssim
    \Bigl(
    \Vert b\Vert_{L^\infty}\Vert \partial'''v\Vert_{L^2}
    +\Vert \partial'b\Vert_{L^\infty}\Vert \partial''v\Vert_{L^2}
    +\left\Vert \Vert \partial''b\Vert_{L^\infty_z}\right\Vert_{L^2_{\HH}}
    \left\Vert \Vert \partial'v\Vert_{L^2_z}\right\Vert_{L^\infty_{\HH}}
    \\&\indeq
    +\left\Vert \Vert \partial''b\Vert_{L^\infty_z}\right\Vert_{L^2_{\HH}}
    \left\Vert \Vert \partial_3v\Vert_{L^2_z}\right\Vert_{L^\infty_{\HH}}
    +\Vert \partial'b\Vert_{L^\infty}
    \Vert \partial_3\partial'v\Vert_{L^2}
    +\Vert b\Vert_{L^\infty}
    \Vert \partial_3\partial''v\Vert_{L^2}
    \Bigr)
    \\&\indeq\indeq
    \times
    \Vert \partial'v\Vert_{L^2}^\frac{1}{4}
    \Vert \partial_{33}\partial'v\Vert_{L^2}^\frac{1}{4}
    \Vert \partial_3\bard''v\Vert ^\frac{1}{2}_{L^2}
    \Vert \partial''v\Vert_{L^2}
    \\&
    \lesssim
    \mathcal{P}
    +\epsilon\Vert v\Vert_{H^3}^2
    .
   \llabel{EQ114}
   \end{split}
  \end{align}

We still need to estimate the term in \eqref{EQ57} involving $\phi_t$,
for which we have
  \begin{align}
    \begin{split}
    J_8
    &=
    -
    \int(
        \partial''\phi_t\partial_3v_\alpha\partial''v_\alpha
        +2\partial'\phi_t\partial'\partial_3v_\alpha\partial''v_\alpha
          )
    \\&\lesssim
    \left\Vert \Vert \partial''\phi_t\Vert_{L^2_z}\right\Vert_{L^4_\HH}
    \left\Vert \Vert \partial_3 v\Vert_{L^2_z}\right\Vert_{L^4_\HH}
    \left\Vert \Vert \partial''v\Vert_{L^\infty_z}\right\Vert_{L^2_\HH}
    +\left\Vert \Vert \partial'\phi_t\Vert_{L^2_z}\right\Vert_{L^\infty_\HH}
    \left\Vert \Vert \partial_3\partial'v\Vert_{L^2_z}\right\Vert_{L^2_\HH}
    \left\Vert \Vert \partial''v\Vert_{L^\infty_z}\right\Vert_{L^2_\HH}
    \\&\lesssim 
    \mathcal{P}\Vert \partial_3v\Vert ^\frac{1}{2}_{L^2}
    \Vert \partial_3\partial'v\Vert_{L^2}^\frac{1}{2}
    \Vert \partial_3\partial''v\Vert_{L^2}^\frac{1}{2}
    \Vert \partial''v\Vert ^\frac{1}{2}_{L^2}
    +\mathcal{P}\Vert \partial_3\partial'v\Vert_{L^2}
    \Vert \partial''v\Vert_{L^2}^\frac{1}{2}
    \Vert \partial''\partial_3v\Vert_{L^2}^\frac{1}{2}
    \\&
    \lesssim \mathcal{P}+\epsilon\Vert v\Vert ^2_{H^3}
    .
    \end{split}
   \llabel{EQ132}
\end{align}
Summarizing all the estimates on \eqref{EQ57}, we thus
obtain
  \begin{equation*}
   I_7
   + I_8
   \lec
   \mathcal{P}
   +\epsilon(\Vert v_t\Vert_{H^1}^2+\Vert v\Vert_{H^3}^2)
   .
  \end{equation*}
For~$I_9$, we again use that the pressure $p$ is independent of the variable $z$, together with Lemma~\ref{L02}.
We have
\begin{align}
\begin{split}
     I_9
     &
    =-\int (\partial''(b_{k\alpha}\partial_kv_\alpha)
        -b_{k\alpha}\partial''\partial_kv_\alpha)\partial''p
    \\&
    =-\int (\partial''(b_{\beta\alpha}\partial_\beta v_\alpha)
    -b_{\beta\alpha}\partial''\partial_\beta v_\alpha)\partial''p
    -\int(\partial''(b_{3\alpha}\partial_3v_\alpha)
    -b_{3\alpha}\partial''\partial_3v_\alpha)\partial''p
    \\&
    \lesssim
    \Vert p\Vert_{H^2}\mathcal{P}
    +\Vert p\Vert_{H^2}\Vert v\Vert_{H^2}\PPs
    .
\end{split}
\llabel{EQ63}
\end{align}
Proceeding similarly for $I_{10}$, we
obtain
  \begin{align}
    \begin{split}
    I_{10}
    &=-\int (\partial''(b_{k\alpha}\partial_kp)
    -b_{k\alpha}\partial''\partial_kp)\partial''v_\alpha
    \\&\lesssim 
    \left\Vert \Vert \partial''b\Vert_{L^2_z}\right\Vert_{L^4_\HH}
    \left\Vert \Vert \partial'p\Vert_{L^\infty_z}\right\Vert_{L^4_\HH}
    \Vert \partial''v\Vert_{L^2}
    +\Vert \partial'b\Vert_{L^\infty}\Vert \partial''p\Vert_{L^2}\Vert \partial''v\Vert_{L^2}\\
    &\lesssim 
    \Vert h\Vert_{H^4}\Vert \partial'p\Vert ^\frac{1}{2}_{L^2}
    \Vert \partial''p\Vert_{L^2}^\frac{1}{2}
    \Vert \partial''v\Vert_{L^2}
    +
    \Vert h\Vert_{H^4}\Vert \partial''p\Vert_{L^2}\Vert \partial''v\Vert_{L^2}
    \\&
    \lesssim \Vert p\Vert_{H^2}\mathcal{P}
    .
    \end{split}
   \llabel{EQ131}
   \end{align}
Summarizing, we 
arrive at the tangential estimate~\eqref{EQ64}.
Note that we have used \eqref{EQ37} to remove~$J$
from two of the terms on the left-hand side of~\eqref{EQ54}.
\end{proof}

\startnewsection{$L^2$ bounds on the velocity}{sec14}
The main goal of this section is to prove the following statement.

\cole%%%out
\begin{Lemma}
\label{L05}
Under the assumptions of Theorem~\ref{T01}
and with $\epsilon\in(0,1/2]$,
we have for the
two normal derivatives of the velocity,
\begin{align}
\begin{split}
    \Vert \partial_{33}v\Vert_{L^2}\lesssim
    (\PPs+\epsilon)\Vert \partial_{33}v\Vert_{L^2}
    +\Vert p\Vert_{L^2}\mathcal{P}
    +\mathcal{P}
    .
\end{split}
\label{EQ69}
\end{align}
and 
\begin{align}
    \begin{split}
    \Vert \partial_{33}\partial'v\Vert_{L^2} 
    &\lesssim
    \Vert \partial'''v\Vert_{L^2}
    +\Vert \partial_t\partial'v\Vert_{L^2}
    +\Vert p\Vert_{H^2}
    \mathcal{P}
    \\&\indeq
    +(\PPs+\epsilon)
    (\Vert \partial_{33}v\Vert_{L^2}
    +\Vert \partial_3\partial''v\Vert_{L^2}
    +\Vert \partial_{33}\partial'v\Vert_{L^2})
    +\mathcal{P}
    .
    \end{split}
   \label{EQ74}
\end{align}
while for the three derivatives, we have
\begin{align}
    \begin{split}
    \Vert \partial_{333}v\Vert_{L^2}
    \lesssim&
    (\PPs+\epsilon)
    (\Vert\partial_{33}v\Vert_{L^2}
    +\Vert\partial_{33}\partial'v\Vert_{L^2}
    +\Vert\partial_{333}v\Vert_{L^2})
    \\&\indeq
    +\Vert \partial_3\partial''v\Vert_{L^2}
    +\Vert\partial_t\partial_3v\Vert_{L^2}
    +\epsilon\Vert\partial'''v\Vert_{L^2}
    +\mathcal{P}
    .
    \end{split}
    \label{EQ79}
\end{align}
\end{Lemma}
\colb%%%out

\begin{proof}[Proof of Lemma~\ref{L05}]
We first isolate $\partial_{33}v$
from the velocity equation. Then, taking the $L^2$-norm on both sides provides a convenient starting point for quantitative control. By carrying out this procedure successively for three quantities, $\partial_{33}v$, $\partial_{33}\partial'v$, $\partial_{333}v$, we generate a chain of $L^2$-estimates in which each higher-order contribution can be bounded by lower-order terms together with the pressure. This sequential approach allows us to reduce the complexity of the expressions while maintaining sharp control over the relevant quantities.

After rearranging the terms in \eqref{EQ49}, we obtain
\begin{align}
\begin{split}
     \partial_{33}v_\alpha=&
     -\Delta_{\HH} v_\alpha
     -2a_{3\beta}\partial_{3\beta} v_\alpha
     -(a_{31}^2+a_{32}^2+a_{33}^2-1)\partial_{33}v_\alpha
     \\&
     +\partial_tv_\alpha
     -a_{kl}\partial_ka_{3l}\partial_3v_\alpha
     +v_\gamma a_{j\gamma}\partial_jv_\alpha
     +\frac{1}{\partial_3\phi}(w-\phi_t)\partial_3v_\alpha
     +a_{k\alpha}\partial_kp
\end{split}
\label{EQ65}
\end{align}
for $\alpha=1,2.$
By taking the $L^2$
norm of both sides of the equation and estimating each term using the triangle inequality, we get
\begin{align}
\begin{split}
     \Vert \partial_{33}v_\alpha\Vert_{L^2}\leq\sum_{m=1}^8 J^{(1)}_m
     ,
\end{split}
\llabel{EQ66}
\end{align}
where $J^{(1)}_1,\ldots,J^{(1)}_8$ denote the $L^2$ norms of the eight terms on the right-hand side of~\eqref{EQ65}.
By Lemma~\ref{L02}, all the linear terms in $v$ can be bounded as
\begin{align}
\begin{split}
    \sum_{m=1}^5 J^{(1)}_m+J^{(1)}_8
    &\lesssim
    \Vert\Delta_{\HH} v_\alpha\Vert_{L^2}
     +\Vert2a_{3\beta}\Vert_{L^\infty} 
     \Vert \partial_{3\beta}v_\alpha\Vert_{L^2}
     +\Vert a_{31}^2+a_{32}^2+a_{33}^2-1\Vert_{L^\infty}
     \Vert\partial_{33}v_\alpha\Vert_{L^2}
     \\&\indeq\indeq
     +\Vert\partial_tv_\alpha\Vert_{L^2}
     +\Vert a_{kl}\partial_ka_{3l}\Vert_{L^\infty}\Vert\partial_3v_\alpha\Vert_{L^2}
     +\Vert a_{k\alpha}\Vert_{L^\infty} \Vert\partial_kp\Vert_{L^2}
    \\&
    \lesssim
    (\Vert \partial''v\Vert_{L^2}
    +\Vert \partial_3v\Vert_{L^2}
    +\Vert \partial_3\partial'v\Vert_{L^2}
    +\Vert v_t\Vert_{L^2}
    +\Vert p\Vert_{H^2})
    \mathcal{P}
    +\Vert \partial_{33}v\Vert_{L^2}
    \PPs
    .
    \end{split}
    \label{EQ67}
\end{align}
An application of the interpolation inequality Lemma~\ref{L03} in combination with Young’s inequality yields the estimate
\begin{align}
    \begin{split}
    \sum_{m=1}^5 J^{(1)}_m+J^{(1)}_8
    &\lesssim
    (\Vert \partial''v\Vert_{L^2}
    +\Vert \partial_{33}v\Vert_{L^2}^\frac{1}{2}
     \Vert v\Vert_{L^2}^\frac{1}{2}
    +\Vert \partial_{33}v\Vert_{L^2}^\frac{1}{2}
     \Vert \partial''v\Vert_{L^2}^\frac{1}{2}
    +\Vert v_t\Vert_{L^2}
    +\Vert p\Vert_{H^2})
    \mathcal{P}
    +\Vert \partial_{33}v\Vert_{L^2}
    \PPs
    \\&
    \lesssim
    (\PPs+\epsilon)\Vert \partial_{33}v\Vert_{L^2}
    +\Vert p\Vert_{L^2}\mathcal{P}
    +\mathcal{P}
    .
    \end{split}
    \llabel{EQ67b}
\end{align}
In the treatment of the nonlinear terms, using Young’s inequality alone is not sufficient to obtain the desired estimate; we wish to employ interpolation so that the 
$\partial_3$ derivatives are concentrated as much as possible on a single term. Thus, by using Lemma ~\ref{L04} and the incompressibility condition, we write
\begin{align}
    \begin{split}
    J_7^{(1)}
    &=
    \left\Vert \frac{1}{\partial_3\phi}(w-\phi_t)\partial_3v_\alpha\right\Vert_{L^2}
    \\&
    =
    \left\Vert \frac{1}{\partial_3\phi}\right\Vert_{L^\infty}
    \left\Vert \Vert w\Vert_{L^\infty_z}\right\Vert_{L^4_{\HH}}
    \left\Vert \Vert \partial_3v\Vert_{L^2_z}\right\Vert_{L^4_{\HH}}
    +\left\Vert \frac{1}{\partial_3\phi}\right\Vert_{L^\infty}
    \Vert \phi_t\Vert_{L^3}
    \Vert \partial_3v\Vert_{L^6}
    \\&
    \lesssim\mathcal{P}
    \left\Vert
    \Vert b_{11}\partial_1 v_1
    +b_{22}\partial_2 v_2
    -b_{31}\partial_3 v_1
    -b_{32}\partial_3 v_2
    \Vert_{L^2_z}
    \right\Vert_{L_\HH^4}
    \Vert \partial_3v\Vert_{L^2}^\frac{1}{2}
    \Vert \partial_3\partial'v\Vert_{L^2}^\frac{1}{2}
    +\Vert \partial_3v\Vert_{L^2}^\frac{1}{2}
    \Vert \partial_{33}v\Vert_{L^2}^\frac{1}{2}
    \mathcal{P}
    \\&
    \lesssim
    \Vert \partial'v\Vert_{L^2}^\frac{1}{2}
    \Vert \partial''v\Vert_{L^2}^\frac{1}{2}
    \Vert \partial_3v\Vert_{L^2}^\frac{1}{2}
    \Vert \partial_3\partial'v\Vert_{L^2}^\frac{1}{2}
    \mathcal{P}
    +\Vert \partial_3v\Vert_{L^2}
    \Vert \partial_3\partial'v\Vert_{L^2}
    \PPs
    +\Vert \partial_3v\Vert_{L^2}^\frac{1}{2}
    \Vert \partial_{33}v\Vert_{L^2}^\frac{1}{2}
    \mathcal{P}
    .
    \end{split}
    \label{EQ68}
\end{align}
We employ Lemma~\ref{L02} here to control the coefficient $b$, while applying Sobolev and interpolation inequalities to obtain the final result. We note that in the expansion of $w$, the coefficient in front of $\partial_3v$ becomes sufficiently small as time tends to zero. We then make further use of Lemma~\ref{L03} to carry out an additional interpolation
\begin{align}
    \begin{split}
    J_7^{(1)}
    &\lesssim
    \left(
    \Vert v\Vert_{L^2}^\frac{1}{2}
    \Vert \partial_{33} v\Vert_{L^2}^\frac{1}{2}
    \right)
    \left(
    \Vert \partial''v\Vert_{L^2}^\frac{1}{2}
    \Vert \partial_{33} v\Vert_{L^2}^\frac{1}{2}
    \right)
    \PPs
    \\&\indeq
    +\Vert \partial'v\Vert_{L^2}^\frac{1}{2}
    \Vert \partial''v\Vert_{L^2}^\frac{1}{2}
    \left(
    \Vert \partial_{33}v\Vert_{L^2}^\frac{1}{4}
    \Vert v \Vert_{L^2}^\frac{1}{4}
    \right)
    \left(
    \Vert \partial_{33}v\Vert_{L^2}^\frac{1}{4}
    \Vert \partial''v \Vert_{L^2}^\frac{1}{4}
    \right)
    \mathcal{P}
    +\left(
    \Vert \partial_{33}v\Vert_{L^2}^\frac{1}{4}
    \Vert v \Vert_{L^2}^\frac{1}{4}
    \right)
    \Vert \partial_{33}v\Vert_{L^2}^\frac{1}{2}
    \mathcal{P}
    \\&
    \lesssim
    (\PPs+\epsilon)\Vert \partial_{33}v\Vert_{L^2}
    +\mathcal{P}
    .
    \llabel{EQ68b}
    \end{split}
\end{align}
In the present set of estimates, the highest-order contributions arise from the nonlinear term involving~$w$. Under the free boundary condition, the incompressibility constraint further produces terms containing~$\partial_3v$. However, the coefficients in front of such $\partial_3v$
-terms necessarily vanish as $t\to 0$. This structural feature makes it possible to incorporate the $\PPs$-type estimates, which capture the smallness in time and allow us to control the nonlinear interactions effectively.

The estimate for $J^{(1)}_6$
  is simpler and more favorable than that for $J_7^{(1)}$
  since it involves more tangential derivatives $\partial'$
rather than the normal derivative~$\partial_3$. Using Lemma~\ref{L04} and Young's inequality, we obtain
\begin{align}
    \begin{split}
        J_6^{(1)}&=\Vert v_\gamma a_{j\gamma}\partial_jv_\alpha\Vert_{L^2}
        \\&\lesssim
        \sum_{j=1,2}
        \left\Vert \Vert v_\gamma\Vert_{L^\infty_z}
        \right\Vert_{L^4_\HH} 
        \Vert a_{j\gamma}\Vert_{L^\infty}
        \left\Vert
        \Vert \partial_jv_\alpha\Vert_{L^2_z}\right\Vert_{L^4_\HH}
        +\left\Vert \Vert v_\gamma\Vert_{L^2_z}
        \right\Vert_{L^\infty_\HH} 
        \Vert a_{3\gamma}\Vert_{L^\infty}
        \left\Vert
        \Vert \partial_3 v_\alpha\Vert_{L^\infty_z}\right\Vert_{L^2_\HH}
        \\&\lesssim \mathcal{P}
        +\epsilon\Vert\partial_{33}v\Vert_{L^2}^2
    .
    \end{split}
    \llabel{EQ68c}
\end{align}
This yields the desired bound, and combining \eqref{EQ67} and \eqref{EQ68} leads to
\begin{align}
\begin{split}
    \Vert \partial_{33}v\Vert_{L^2}\lesssim
    (\PPs+\epsilon)\Vert \partial_{33}v\Vert_{L^2}
    +\Vert p\Vert_{L^2}\mathcal{P}
    +\mathcal{P}
    ,
\end{split}
\label{EQ69a}
\end{align}
which is \eqref{EQ69}.

By applying the tangential derivative $\partial'$ to the equation \eqref{EQ65}, we derive
\begin{align}
    \begin{split}
        \partial_{33}\partial'v_\alpha
        =&
        -\Delta_{\HH}\partial'v_\alpha
        -2\partial'(a_{3\beta}\partial_{3\beta} v_\alpha)
        \\&
        +2\partial'(a_{31}^2+a_{32}^2+a_{33}^2-1)\partial_{33}v_\alpha
        +2(a_{31}^2+a_{32}^2+a_{33}^2-1)\partial_{33}\partial'v_\alpha
        \\&
        +\partial_t\partial'v_\alpha
        +\partial'(a_{kl}\partial_ka_{3l}\partial_3v_\alpha)
        +\partial'(v_\gamma a_{j\gamma}\partial_jv_\alpha)
        \\&
        +\partial'(a_{33}(w-\phi_t)\partial_3v_\alpha)
        +\partial'(a_{k\alpha}\partial_kp)
    \end{split}
    \label{EQ70}
\end{align}
for $\alpha=1,2.$
Analogously, taking the $L^2$
norm on both sides of the equation \eqref{EQ70} yields
\begin{align}
    \Vert \partial_{33}\partial'v_\alpha\Vert_{L^2}\leq\sum_{m=1}^9 J^{(2)}_m.
    \llabel{EQ71}
\end{align}
Utilizing Lemma~\ref{L02} once more, we derive estimates for the linear terms,
\begin{equation}
\begin{split}
    &\sum_{m=1}^6 J^{(2)}_m+J^{(2)}_9\\
    &\lesssim
    \Vert\Delta_{\HH}\partial'v_\alpha\Vert_{L^2}
    +\Vert\partial'a_{3\beta}\Vert_{L^2}
    \Vert\partial_{3\beta} v_\alpha\Vert_{L^2}
    +\Vert a_{3\beta}\Vert_{L^2}
    \Vert\partial'\partial_{3\beta} v_\alpha\Vert_{L^2}
    \\&\indeq
    +\Vert\partial'(a_{31}^2+a_{32}^2+a_{33}^2-1)\Vert_{L^\infty}
    \Vert\partial_{33}v_\alpha\Vert_{L^2}
    +\Vert(a_{31}^2+a_{32}^2+a_{33}^2-1)\Vert_{L^\infty}
    \Vert\partial_{33}\partial'v_\alpha\Vert_{L^2}
    \\&\indeq
    +\Vert\partial_t\partial'v_\alpha\Vert_{L^2}
    +\Vert\partial'(a_{kl}\partial_ka_{3l})\Vert_{L^\infty}
    \Vert\partial_3v_\alpha\Vert_{L^2}
    +\Vert a_{kl}\partial_ka_{3l}\Vert_{L^\infty}
    \Vert\partial'\partial_3v_\alpha\Vert_{L^2}
    \\&\indeq
    +\Vert\partial'a_{k\alpha}\Vert_{L^\infty}
    \Vert\partial_kp\Vert_{L^2}
    +\Vert a_{k\alpha}\Vert_{L^\infty}
    \Vert\partial'\partial_kp\Vert_{L^2}
    \\&\lesssim
    \Vert \partial'''v\Vert_{L^2}
    +\Vert \partial_t\partial'v\Vert_{L^2}
    +(
    \Vert \partial_3v\Vert_{L^2}
    +\Vert \partial_3\partial'v\Vert_{L^2}
    +\Vert p\Vert_{H^2})
    \mathcal{P}
    \\&\indeq\indeq
    +(\Vert \partial_3\partial''v\Vert_{L^2}
    +\Vert \partial_{33}v\Vert_{L^2}
    +\Vert \partial_{33}\partial'v\Vert_{L^2})
    \PPs
    ,
    \label{EQ72}
    \end{split}
\end{equation}
whence, using in particular $\Vert a_{3\beta}\Vert_{L^2}\leq \PPs$,
\begin{equation}
\begin{split}
    \sum_{m=1}^6 J^{(2)}_m+J^{(2)}_9
    \lesssim
    \Vert \partial'''v\Vert_{L^2}
    +\Vert \partial_t\partial'v\Vert_{L^2}
    +\Vert p\Vert_{H^2}
    \mathcal{P}
    +(\Vert \partial_3\partial''v\Vert_{L^2}
    +\Vert \partial_{33}v\Vert_{L^2}
    +\Vert \partial_{33}\partial'v\Vert_{L^2})
    \PPs
    +\mathcal{P}
   .
   \end{split}
   \llabel{EQ116}
\end{equation}
In the treatment of the nonlinear terms under the $L^2$ estimate, the main tools we employ are Lemmas~\ref{L03} and \ref{L04},
\begin{align}
    \begin{split}
    J_8^{(2)}
    &=\Vert \partial'(a_{33}(w-\phi_t)\partial_3v_\alpha)\Vert_{L^2}
    \\&
    =
    \Vert \partial'a_{33}\Vert_{L^\infty}
    \left\Vert \Vert w\Vert_{L^\infty_z}\right\Vert_{L^4_{\HH}}
    \left\Vert \Vert \partial_3v\Vert_{L^2_z}\right\Vert_{L^4_{\HH}}
    +\Vert \partial'a_{33}\Vert_{L^\infty}
    \Vert \phi_t\Vert_{L^6}
    \Vert \partial_3v\Vert_{L^3}
    +\Vert a_{33}\Vert_{L^\infty}
    \Vert \partial'\phi_t\Vert_{L^6}
    \Vert \partial_3v\Vert_{L^3}
    \\&\indeq\indeq
    +\Vert a_{33}\Vert_{L^\infty}
    \left(
    \left\Vert \Vert \partial'w\Vert_{L^\infty_z}\right\Vert_{L^2_{\HH}}
    \left\Vert \Vert \partial_3v\Vert_{L^2_z}\right\Vert_{L^\infty_{\HH}}
    +\left\Vert \Vert w\Vert_{L^\infty_z}\right\Vert_{L^4_{\HH}}
    \left\Vert \Vert \partial'\partial_3v\Vert_{L^2_z}\right\Vert_{L^4_{\HH}}
    \right)
    \\&\indeq\indeq
    +\Vert a_{33}\Vert_{L^\infty}
    \Vert \phi_t\Vert_{L^\infty}
    \Vert \partial_3\partial'v\Vert_{L^2}
    ,    
    \end{split}
   \llabel{EQ58}
\end{align}
from where
\begin{align}
    \begin{split}
    J_8^{(2)}
    &
    \lec
    \PPs \left\Vert
    \Vert b_{11}\partial_1 v_1
    +b_{22}\partial_2 v_2
    -b_{31}\partial_3 v_1
    -b_{32}\partial_3 v_2
    \Vert_{L^2_z}
    \right\Vert_{L_\HH^4}
    \Vert \partial_3(\partial')^\frac{1}{2}v\Vert_{L^2}
    \\&\indeq
    +\mathcal{P}
    \left(
    \Vert \partial_{3}v\Vert_{L^2}
    +
    \Vert\partial_3v\Vert_{L^2}^\frac{1}{2}
    \Vert \partial_{33}v\Vert_{L^2}^\frac{1}{2}
    +
    \Vert \partial_{3}(\partial')^\frac{1}{2}v\Vert_{L^2}
    \right)
    \\&\indeq
    +\mathcal{P}\left\Vert
    \Vert 
    \partial'(
    b_{11}\partial_1 v_1
    +b_{22}\partial_2 v_2
    -b_{31}\partial_3 v_1
    -b_{32}\partial_3 v_2)
    \Vert_{L^2_z}
    \right\Vert_{L_\HH^2}
    \Vert \partial_3(\partial')^\frac{1}{2}v\Vert_{L^2}^\frac{1}{2}
    \Vert \partial_3(\partial')^\frac{3}{2}v\Vert_{L^2}^\frac{1}{2}
    \\&\indeq
    +\mathcal{P}
    \left\Vert
    \Vert b_{11}\partial_1 v_1
    +b_{22}\partial_2 v_2
    -b_{31}\partial_3 v_1
    -b_{32}\partial_3 v_2
    \Vert_{L^2_z}
    \right\Vert_{L_\HH^4}
    \Vert \partial_3(\partial')^\frac{3}{2}v\Vert_{L^2}
    \\&\indeq
    +\mathcal{P}
    \Vert \partial_{33}\partial'v\Vert_{L^2}^\frac{1}{2}
    \Vert \partial'v\Vert_{L^2}^\frac{1}{2}
    .
    \end{split}
    \label{EQ73}
\end{align}
\eold
Hence, by Lemma~\ref{L02}, we have $b_{31},b_{32}\to0$ as $t\to 0$. 
Combining this with the interpolation inequality and the two-dimensional Sobolev inequality, we have
\begin{align}
    \begin{split}
    J_8^{(2)}
    &\lesssim
    \PPs \left(
    \Vert \partial'v\Vert_{L^2}^\frac{1}{2}
    \Vert \partial''v\Vert_{L^2}^\frac{1}{2}
    +
    \Vert \partial_{3}(\partial')^\frac{1}{2}v\Vert_{L^2}
    \right)
    \Vert \partial_{33}\partial'v\Vert_{L^2}^\frac{1}{2}
    \Vert v\Vert_{L^2}^\frac{1}{2}
    \\&\indeq
    +\mathcal{P}
    \left(
    \Vert \partial_{33}v\Vert_{L^2}^\frac{1}{2}
    \Vert v\Vert_{L^2}^\frac{1}{4}
    \Vert\partial_{33}v\Vert_{L^2}^\frac{1}{4}
    +
    \Vert \partial_{33}\partial'v\Vert_{L^2}^\frac{1}{2}
    \Vert v\Vert_{L^2}^\frac{1}{2}
    \right)
    \\&\indeq
    +
    \left(
    \mathcal{P}
    \Vert \partial''v\Vert_{L^2}
    +
    \PPs
    \Vert \partial_{3}\partial'v\Vert_{L^2}
    \right)
    \Vert \partial_3(\partial')^\frac{1}{2}v\Vert_{L^2}^\frac{1}{2}
    \Vert \partial_3(\partial')^\frac{3}{2}v\Vert_{L^2}^\frac{1}{2}
    \\&\indeq
    +
    \left(
    \mathcal{P}
    \Vert \partial'v\Vert_{L^2}^\frac{1}{2}
    \Vert \partial''v\Vert_{L^2}^\frac{1}{2}
    +
    \PPs
    \Vert \partial_{3}(\partial')^\frac{1}{2}v\Vert_{L^2}
    \right)
    \Vert \partial_{33}\partial'v\Vert_{L^2}^\frac{1}{2}
    \Vert \partial''v\Vert_{L^2}^\frac{1}{2}
    \\&\indeq
    +\mathcal{P}
    \Vert \partial_{33}\partial'v\Vert_{L^2}^\frac{1}{2}
    \Vert \partial'v\Vert_{L^2}^\frac{1}{2}
    .
    \end{split}
    \llabel{EQ73b}
\end{align}
An integration by parts argument gives the inequality 
\begin{align}
    \begin{split}
        \Vert\partial_3(\partial')^{s}v\Vert_{L^2}
        \lesssim
        \Vert \partial_{33}\partial'v\Vert^\frac{1}{2}
        \Vert (\partial')^{2s-1}v\Vert^\frac{1}{2}
        ,
        \llabel{EQ150}
    \end{split}
\end{align}
where $s\in\left[\frac{1}{2},\frac{3}{2}\right]$
. Employing this in the estimate, we obtain
\begin{align}
    \begin{split}
    J_8^{(2)}
    &\lesssim
    \PPs \left(
    \Vert \partial'v\Vert_{L^2}^\frac{1}{2}
    \Vert \partial''v\Vert_{L^2}^\frac{1}{2}
    +
    \Vert \partial_{33}\partial'v\Vert_{L^2}^\frac{1}{2}
    \Vert v\Vert_{L^2}^\frac{1}{2}
    \right)
    \Vert \partial_{33}\partial'v\Vert_{L^2}^\frac{1}{2}
    \Vert v\Vert_{L^2}^\frac{1}{2}
    \\&\indeq
    +\mathcal{P}
    \left(
    \Vert \partial_{33}v\Vert_{L^2}^\frac{1}{2}
    \Vert v\Vert_{L^2}^\frac{1}{4}
    \Vert\partial_{33}v\Vert_{L^2}^\frac{1}{4}
    +
    \Vert \partial_{33}\partial'v\Vert_{L^2}^\frac{1}{2}
    \Vert v\Vert_{L^2}^\frac{1}{2}
    \right)
    \\&\indeq
    +
    \left(
    \mathcal{P}
    \Vert \partial''v\Vert_{L^2}
    +
    \PPs
    \Vert \partial_{33}\partial'v\Vert_{L^2}^\frac{1}{2}
    \Vert \partial'' v\Vert_{L^2}^\frac{1}{2}
    \right)
    \left(
    \Vert \partial_{33}\partial'v\Vert_{L^2}^\frac{1}{2}
    \Vert v\Vert_{L^2}^\frac{1}{2}
    \right)^\frac{1}{2}
    \left(
    \Vert \partial_{33}\partial'v\Vert_{L^2}^\frac{1}{2}
    \Vert \partial''v\Vert_{L^2}^\frac{1}{2}
    \right)^\frac{1}{2}
    \\&\indeq
    +
    \left(
    \mathcal{P}
    \Vert \partial'v\Vert_{L^2}^\frac{1}{2}
    \Vert \partial''v\Vert_{L^2}^\frac{1}{2}
    +
    \PPs
    \Vert \partial_{33}\partial'v\Vert_{L^2}^\frac{1}{2}
    \Vert v\Vert_{L^2}^\frac{1}{2}
    \right)
    \Vert \partial_{33}\partial'v\Vert_{L^2}^\frac{1}{2}
    \Vert \partial''v\Vert_{L^2}^\frac{1}{2}
    \\&\indeq
    +\mathcal{P}
    \Vert \partial_{33}\partial'v\Vert_{L^2}^\frac{1}{2}
    \Vert \partial'v\Vert_{L^2}^\frac{1}{2}
    .
    \end{split}
   \llabel{EQ115}
\end{align}
\eold
After combining some terms and applying Young’s inequality, we obtain
\begin{align}
    \begin{split}
        J_8^{(2)}\lesssim 
        (\epsilon+\PPs)
        (\Vert\partial_{33}v\Vert_{L^2}+\Vert\partial_{33}\partial'v\Vert_{L^2})
        +\epsilon\Vert\partial'''v\Vert_{L^2}
        +\mathcal{P}
        .
        \label{EQ73c}
    \end{split}
\end{align}
By employing the same approach to $J^{(2)}_7$, we derive 
\begin{align}
    \begin{split}
    J_{7}^{(2)}&= \Vert\partial' (v_\gamma a_{j\gamma}\partial_jv_\alpha)\Vert_{L^2}
    \\&\lesssim
    \Vert a_{j\gamma}\Vert_{L^\infty}
    \left\Vert\Vert \partial'v\Vert_{L^\infty_z}\right\Vert_{L^4_\HH}
    \left(
    \left\Vert\Vert \partial'v\Vert_{L^2_z}\right\Vert_{L^4_\HH}
    +
    \left\Vert\Vert \partial_3v\Vert_{L^2_z}\right\Vert_{L^4_\HH}
    \right)
    \\&\indeq
    +
    \Vert a_{j\gamma}\Vert_{L^\infty}
    \left\Vert\Vert v\Vert_{L^\infty_z}\right\Vert_{L^4_\HH}
    \left(
    \left\Vert\Vert \partial''v\Vert_{L^2_z}\right\Vert_{L^4_\HH}
    +
    \left\Vert\Vert \partial'\partial_3v\Vert_{L^2_z}\right\Vert_{L^4_\HH}
    \right)
    \\&\indeq
    +\Vert \partial'a_{j\gamma}\Vert_{L^\infty}
    \left\Vert\Vert v\Vert_{L^\infty_z}\right\Vert_{L^4_\HH}
    \left(\left\Vert\Vert \partial'v\Vert_{L^2_z}\right\Vert_{L^4_\HH}
    +
    \left\Vert\Vert \partial_3v\Vert_{L^2_z}\right\Vert_{L^4_\HH}
    \right)
    .
    \llabel{EQ73d}
    \end{split}
\end{align}
At this point, we invoke Lemma~\ref{L04} and deduce
\begin{align}
    \begin{split}
    J_{7}^{(2)}
    &\lec
    \PPs
    \Vert \partial'v\Vert_{L^2}^\frac{1}{4}
    \Vert \partial''v\Vert_{L^2}^\frac{1}{4}
    \Vert \partial_3\partial'v\Vert_{L^2}^\frac{1}{4}
    \Vert \partial_3\partial''v\Vert_{L^2}^\frac{1}{4}
    \left(\Vert \partial'v\Vert_{L^2}^\frac{1}{2}
    \Vert \partial''v\Vert_{L^2}^\frac{1}{2}
    +\Vert\partial_3(\partial')^\frac{1}{2}v\Vert_{L^2}
    \right)
    \\&\indeq
    +\PPs
    \Vert v\Vert_{L^2}^\frac{1}{4}
    \Vert \partial'v\Vert_{L^2}^\frac{1}{4}
    \Vert \partial_3v\Vert_{L^2}^\frac{1}{4}
    \Vert \partial_3\partial'v\Vert_{L^2}^\frac{1}{4}
    \left(\Vert \partial''v\Vert_{L^2}^\frac{1}{2}
    \Vert \partial'''v\Vert_{L^2}^\frac{1}{2}
    +\Vert\partial_3(\partial')^\frac{3}{2}v\Vert_{L^2}
    \right)
    \\&\indeq
    +\PPs
    \Vert v\Vert_{L^2}^\frac{1}{4}
    \Vert \partial'v\Vert_{L^2}^\frac{1}{4}
    \Vert \partial_3v\Vert_{L^2}^\frac{1}{4}
    \Vert \partial_3\partial'v\Vert_{L^2}^\frac{1}{4}
    \left(\Vert \partial'v\Vert_{L^2}^\frac{1}{2}
    \Vert \partial''v\Vert_{L^2}^\frac{1}{2}
    +\Vert\partial_3(\partial')^\frac{1}{2}v\Vert_{L^2}
    \right)
    .
   \label{EQ117}
    \end{split}
  \end{align}
According to Lemma~\ref{L03}, we obtain 
\begin{align}
    \begin{split}
    J_{7}^{(2)}
    &\lesssim
    \PPs
    \Vert\partial'v\Vert_{L^2}^\frac{1}{4}
    \Vert \partial'' v\Vert_{L^2}^\frac{1}{4}
    \Vert \partial_{33}\partial'v\Vert_{L^2}^\frac{1}{8}
    \Vert \partial'v\Vert_{L^2}^\frac{1}{8}
    \Vert\partial_3\partial''v\Vert_{L^2}^\frac{1}{4}
    \left(\Vert \partial'v\Vert_{L^2}^\frac{1}{2}
    \Vert \partial''v\Vert_{L^2}^\frac{1}{2}
    +\Vert\partial_{33}\partial'v\Vert_{L^2}^\frac{1}{2}
    \Vert v\Vert_{L^2}^\frac{1}{2}
    \right)
    \\&\indeq
    +\PPs
    \Vert v\Vert_{L^2}^\frac{1}{4}
    \Vert \partial'v\Vert_{L^2}^\frac{1}{4}
    \Vert \partial_{33}v\Vert_{L^2}^\frac{1}{8}
    \Vert v\Vert_{L^2}^\frac{1}{8}
    \Vert \partial_{33}\partial'v\Vert_{L^2}^\frac{1}{8}
    \Vert \partial'v\Vert_{L^2}^\frac{1}{8}
    \left(\Vert \partial''v\Vert_{L^2}^\frac{1}{2}
    \Vert \partial'''v\Vert_{L^2}^\frac{1}{2}
    +\Vert\partial_{33}\partial'v\Vert_{L^2}^\frac{1}{2}
    \Vert \partial'' v\Vert_{L^2}^\frac{1}{2}
    \right)
    \\&\indeq
    +\PPs
    \Vert v\Vert_{L^2}^\frac{1}{4}
    \Vert \partial'v\Vert_{L^2}^\frac{1}{4}
    \Vert \partial_{33}v\Vert_{L^2}^\frac{1}{8}
    \Vert v\Vert_{L^2}^\frac{1}{8}
    \Vert \partial_{33}\partial'v\Vert_{L^2}^\frac{1}{8}
    \Vert \partial'v\Vert_{L^2}^\frac{1}{8}
    \left(\Vert \partial'v\Vert_{L^2}^\frac{1}{2}
    \Vert \partial''v\Vert_{L^2}^\frac{1}{2}
    +\Vert\partial_{33}\partial'v\Vert_{L^2}^\frac{1}{2}
    \Vert v\Vert_{L^2}^\frac{1}{2}
    \right)
    \\&\lesssim
    \epsilon
    \left(\Vert\partial_{33}v\Vert_{L^2}
    +\Vert\partial_{33}\partial'v\Vert_{L^2}
    +\Vert\partial'''v\Vert_{L^2}
    +\Vert\partial_3\partial''v\Vert_{L^2}
    \right)
    +\PPs
    .
        \label{EQ73e}
    \end{split}
\end{align}
By combining the estimates in \eqref{EQ72}, \eqref{EQ73c}, and
  \eqref{EQ73e}, we arrive at an
$L^2$ estimate for $\partial_{33}\partial'v$, which reads
\begin{align}
    \begin{split}
    \Vert \partial_{33}\partial'v\Vert_{L^2} 
    &\lesssim
    \Vert \partial'''v\Vert_{L^2}
    +\Vert \partial_t\partial'v\Vert_{L^2}
    +\Vert p\Vert_{H^2}
    \mathcal{P}
    \\&\indeq
    +(\PPs+\epsilon)
    (\Vert \partial_{33}v\Vert_{L^2}
    +\Vert \partial_3\partial''v\Vert_{L^2}
    +\Vert \partial_{33}\partial'v\Vert_{L^2})
    +\mathcal{P}
    .
    \end{split}
   \label{EQ74a}
\end{align}
Applying $\partial_3$ to both sides of \eqref{EQ65} and proceeding similarly, we obtain
\begin{align}
    \begin{split}
        \partial_{333}v_\alpha
        =&
        -\Delta_{\HH}\partial_3v_\alpha
        -2\partial_3(a_{3\beta}\partial_{3\beta} v_\alpha)
        \\&
        +2\partial_3(a_{31}^2+a_{32}^2+a_{33}^2-1)\partial_{33}v_\alpha
        +2(a_{31}^2+a_{32}^2+a_{33}^2-1)\partial_{333}v_\alpha
        \\&
        +\partial_t\partial_3v_\alpha
        +\partial_3(a_{kl}\partial_ka_{3l}\partial_3v_\alpha)
        +\partial_3(v_\gamma a_{j\gamma}\partial_jv_\alpha)
        \\&
        +\partial_3(a_{33}(w-\phi_t)\partial_3v_\alpha)
        +\partial_3(a_{k\alpha}\partial_kp)
    \end{split}
    \label{EQ75}
\end{align}
for $\alpha=1,2.$
Taking the $L^2$ norm on both sides of the equation \eqref{EQ75}, the following estimate holds
\begin{align}
    \Vert \partial_{333}v_\alpha\Vert_{L^2}\leq\sum_{m=1}^9 J^{(3)}_m.
    \llabel{EQ76}
\end{align}
Since $\partial_3 p=0$, we have
\begin{align}
    J_{9}^{(3)}=\partial_3(a_{k\alpha}\partial_kp)=\sum_{k=1,2}\partial_3a_{k\alpha}\partial_kp.
   \llabel{EQ133}
\end{align}
According to the definition of $a$, when $k=1,2$ one has $a_{k\alpha}\in\{0,1\}$; therefore, $J_9^{(3)}=0$.
By applying Lemma~\ref{L02}, we may derive an estimate for the linear terms,
\begin{align}
    \begin{split}
    &\sum_{m=1}^6 J^{(3)}_m+J^{(3)}_9
    \\&\lesssim
        \Vert \Delta_{\HH}\partial_3v_\alpha\Vert_{L^2}
        +\Vert\partial_3a_{3\beta}\Vert_{L^\infty}
        \Vert\partial_{3\beta} v_\alpha\Vert_{L^2}
        +\Vert a_{3\beta}\Vert_{L^\infty}
        \Vert\partial_{33\beta} v_\alpha\Vert_{L^2}
        \\&\indeq
        +\Vert\partial_3(a_{31}^2+a_{32}^2+a_{33}^2-1)\Vert_{L^\infty}
        \Vert\partial_{33}v_\alpha\Vert_{L^2}
        +\Vert a_{31}^2+a_{32}^2+a_{33}^2-1\Vert_{L^\infty}
        \Vert\partial_{333}v_\alpha\Vert_{L^2}
        \\&\indeq
        +\Vert\partial_t\partial_3v_\alpha\Vert_{L^2}
        +\Vert\partial_3(a_{kl}\partial_ka_{3l})\Vert_{L^\infty}
        \Vert\partial_3v_\alpha\Vert_{L^2}
        +\Vert a_{kl}\partial_ka_{3l}\Vert_{L^\infty}
        \Vert\partial_{33}v_\alpha\Vert_{L^2}
    \\&\lesssim \PPs(\Vert\partial_{33}v\Vert_{L^2}
    +\Vert\partial_{33}\partial'v\Vert_{L^2}
    +\Vert\partial_{333}v\Vert_{L^2})
    +\Vert \partial_3\partial''v\Vert_{L^2}
    +\Vert\partial_t\partial_3v\Vert_{L^2}
    .
    \end{split}
    \label{EQ77}
\end{align}
Using the incompressibility condition~\eqref{EQ19} once more, we deduce the estimate for the nonlinear term,
\begin{align}
    \begin{split}
    J_8^{(3)}
    &=\Vert \partial_{3}(a_{33}(w-\phi_t)\partial_3v_\alpha)\Vert_{L^2}
    \\&
    =
    \Vert \partial_3a_{33}\Vert_{L^\infty}
    \left\Vert \Vert w\Vert_{L^\infty_z}\right\Vert_{L^4_{\HH}}
    \left\Vert \Vert \partial_3v\Vert_{L^2_z}\right\Vert_{L^4_{\HH}}
    +\Vert \partial_3a_{33}\Vert_{L^\infty}
    \Vert \phi_t\Vert_{L^6}
    \Vert \partial_3v\Vert_{L^3}
    +\Vert a_{33}\Vert_{L^\infty}
    \Vert \partial_3\phi_t\Vert_{L^6}
    \Vert \partial_3v\Vert_{L^3}
    \\&\indeq\indeq
    +\Vert a_{33}\Vert_{L^\infty}
    \left(
    \left\Vert \Vert \partial_3w\Vert_{L^\infty_z}\right\Vert_{L^4_{\HH}}
    \left\Vert \Vert \partial_3v\Vert_{L^2_z}\right\Vert_{L^4_{\HH}}
    +\left\Vert \Vert w\Vert_{L^\infty_z}\right\Vert_{L^4_{\HH}}
    \left\Vert \Vert \partial_{33}v\Vert_{L^2_z}\right\Vert_{L^4_{\HH}}
    \right)
    \\&\indeq\indeq
    +\Vert a_{33}\Vert_{L^\infty}
    \Vert \phi_t\Vert_{L^\infty}
    \Vert \partial_{33}v\Vert_{L^2}
    ,
    \end{split}
   \llabel{EQ122}
\end{align}
from where
\begin{align}
    \begin{split}
    J_8^{(3)}
    &\lesssim
    \PPs \left\Vert
    \Vert b_{11}\partial_1 v_1
    +b_{22}\partial_2 v_2
    -b_{31}\partial_3 v_1
    -b_{32}\partial_3 v_2
    \Vert_{L^2_z}
    \right\Vert_{L_\HH^4}
    \Vert \partial_3(\partial')^\frac{1}{2}v\Vert_{L^2}
    \\&\indeq
    +\mathcal{P}
    (
    \Vert v\Vert_{L^2}^\frac{1}{2}
    \Vert \partial_{333}v\Vert_{L^2}^\frac{1}{2}
    +
    \Vert \partial_{3}(\partial')^\frac{1}{2}v\Vert_{L^2}
    )
    \\&\indeq
    +\mathcal{P}\left\Vert
    \Vert 
    b_{11}\partial_1 v_1
    +b_{22}\partial_2 v_2
    -b_{31}\partial_3 v_1
    -b_{32}\partial_3 v_2
    \Vert_{L^2_z}
    \right\Vert_{L_\HH^4}
    \Vert \partial_3(\partial')^\frac{1}{2}v\Vert_{L^2}
    \\&\indeq
    +\mathcal{P}
    \left\Vert
    \Vert b_{11}\partial_1 v_1
    +b_{22}\partial_2 v_2
    -b_{31}\partial_3 v_1
    -b_{32}\partial_3 v_2
    \Vert_{L^2_z}
    \right\Vert_{L_\HH^4}
    \Vert \partial_{33}(\partial')^\frac{1}{2}v\Vert_{L^2}
    \\&\indeq
    +\mathcal{P}
    \Vert \partial_{333}v\Vert_{L^2}^\frac{2}{3}
    \Vert v\Vert_{L^2}^\frac{1}{3}
    .
    \end{split}
    \llabel{EQ78}
\end{align}
\eold
Applying Lemma~\ref{L03}, we obtain
\begin{align}
    \begin{split}
    J_8^{(3)}
    &\lesssim
    \PPs \left(
    \Vert \partial'v\Vert_{L^2}^\frac{1}{2}
    \Vert \partial''v\Vert_{L^2}^\frac{1}{2}
    +
    \Vert \partial_{3}(\partial')^\frac{1}{2}v\Vert_{L^2}
    \right)
    \Vert \partial_{333}v\Vert_{L^2}^\frac{1}{3}
    \Vert (\partial')^\frac{3}{4}v\Vert_{L^2}^\frac{2}{3}
    \\&\indeq
    +\mathcal{P}
    \left(
    \Vert \partial_{333}v\Vert_{L^2}^\frac{1}{2}
    \Vert v\Vert_{L^2}^\frac{1}{2}
    +
    \Vert \partial_{333}v\Vert_{L^2}^\frac{1}{3}
    \Vert (\partial')^\frac{3}{4}v\Vert_{L^2}^\frac{2}{3}
    \right)
    \\&\indeq
    +
    \left(
    \mathcal{P}
    \Vert \partial'v\Vert_{L^2}^\frac{1}{2}
    \Vert \partial''v\Vert_{L^2}^\frac{1}{2}
    +
    \PPs
    \Vert \partial_{3}(\partial')^\frac{1}{2}v\Vert_{L^2}
    \right)
    \Vert \partial_{333}v\Vert_{L^2}^\frac{1}{3}
    \Vert (\partial')^\frac{3}{4}v\Vert_{L^2}^\frac{2}{3}
    \\&\indeq
    +
    \left(
    \mathcal{P}
    \Vert \partial'v\Vert_{L^2}^\frac{1}{2}
    \Vert \partial''v\Vert_{L^2}^\frac{1}{2}
    +
    \PPs
    \Vert \partial_{3}(\partial')^\frac{1}{2}v\Vert_{L^2}
    \right)
    \Vert \partial_{333}v\Vert_{L^2}^\frac{2}{3}
    \Vert (\partial')^\frac{3}{2}v\Vert_{L^2}^\frac{1}{3}
    \\&\indeq
    +\mathcal{P}
    \Vert \partial_{333}v\Vert_{L^2}^\frac{2}{3}
    \Vert v\Vert_{L^2}^\frac{1}{3}
    .
    \end{split}
    \label{EQ78b}
\end{align}
At this stage, invoking Lemma~\ref{L03}, we deduce
\begin{align}
    \begin{split}
    J_8^{(3)}
    &\lesssim
    \PPs \left(
    \Vert \partial'v\Vert_{L^2}^\frac{1}{2}
    \Vert \partial''v\Vert_{L^2}^\frac{1}{2}
    +
    \Vert \partial_{333}v\Vert_{L^2}^\frac{1}{3}
    \Vert (\partial')^\frac{3}{4}v\Vert_{L^2}^\frac{2}{3}
    \right)
    \Vert \partial_{333}v\Vert_{L^2}^\frac{1}{3}
    \Vert (\partial')^\frac{3}{4}v\Vert_{L^2}^\frac{2}{3}
    \\&\indeq
    +\mathcal{P}
    \left(
    \Vert \partial_{333}v\Vert_{L^2}^\frac{1}{2}
    \Vert v\Vert_{L^2}^\frac{1}{2}
    +
    \Vert \partial_{333}v\Vert_{L^2}^\frac{1}{3}
    \Vert (\partial')^\frac{3}{4}v\Vert_{L^2}^\frac{2}{3}
    \right)
    \\&\indeq
    +
    \left(
    \mathcal{P}
    \Vert \partial'v\Vert_{L^2}^\frac{1}{2}
    \Vert \partial''v\Vert_{L^2}^\frac{1}{2}
    +
    \PPs
    \Vert \partial_{333}v\Vert_{L^2}^\frac{1}{3}
    \Vert (\partial')^\frac{3}{4}v\Vert_{L^2}^\frac{2}{3}
    \right)
    \Vert \partial_{333}v\Vert_{L^2}^\frac{1}{3}
    \Vert (\partial')^\frac{3}{4}v\Vert_{L^2}^\frac{2}{3}
    \\&\indeq
    +
    \left(
    \mathcal{P}
    \Vert \partial'v\Vert_{L^2}^\frac{1}{2}
    \Vert \partial''v\Vert_{L^2}^\frac{1}{2}
    +
    \PPs
    \Vert \partial_{333}v\Vert_{L^2}^\frac{1}{3}
    \Vert (\partial')^\frac{3}{4}v\Vert_{L^2}^\frac{2}{3}
    \right)
    \Vert \partial_{333}v\Vert_{L^2}^\frac{2}{3}
    \Vert (\partial')^\frac{3}{2}v\Vert_{L^2}^\frac{1}{3}
    \\&\indeq
    +\mathcal{P}
    \Vert \partial_{333}v\Vert_{L^2}^\frac{2}{3}
    \Vert v\Vert_{L^2}^\frac{1}{3}
    .
    \end{split}
   \llabel{EQ118}
\end{align}
\eold
Collecting like terms and then applying Young’s inequality, we arrive at
\begin{align}
    J_8^{(3)}\lesssim \epsilon\Vert\partial_{333}v\Vert_{L^2}+\mathcal{P}.
    \llabel{EQ78c}
\end{align}
Similarly, we obtain an estimate for $J^{(3)}_7$, which is similar to that for~$J^{(3)}_8$. Thus,
\begin{align}
    \begin{split}
    J_{7}^{(3)}&= \Vert\partial_3 (v_\gamma a_{j\gamma}\partial_jv_\alpha)\Vert_{L^2}
    \\&\lesssim
    \sum_{j=1,2}
    \Vert a_{j\gamma}\Vert_{L^\infty}
    \left\Vert\Vert \partial_3v\Vert_{L^\infty_z}\right\Vert_{L^4_\HH}
    \left\Vert\Vert \partial'v\Vert_{L^2_z}\right\Vert_{L^4_\HH}
    +
    \Vert a_{3\gamma}\Vert_{L^\infty}
    \left\Vert\Vert \partial_3v\Vert_{L^\infty_z}\right\Vert_{L^4_\HH}
    \left\Vert\Vert \partial_3v\Vert_{L^2_z}\right\Vert_{L^4_\HH}
    \\&\indeq
    +
    \Vert a_{j\gamma}\Vert_{L^\infty}
    \left\Vert\Vert v\Vert_{L^\infty_z}\right\Vert_{L^4_\HH}
    \left(
    \left\Vert\Vert \partial_3\partial'v\Vert_{L^2_z}\right\Vert_{L^4_\HH}
    +
    \left\Vert\Vert \partial_{33}v\Vert_{L^2_z}\right\Vert_{L^4_\HH}
    \right)
    \\&\indeq
    +\Vert \partial_3a_{j\gamma}\Vert_{L^\infty}
    \left\Vert\Vert v\Vert_{L^\infty_z}\right\Vert_{L^4_\HH}
    \left(\left\Vert\Vert \partial'v\Vert_{L^2_z}\right\Vert_{L^4_\HH}
    +
    \left\Vert\Vert \partial_3v\Vert_{L^2_z}\right\Vert_{L^4_\HH}
    \right)
    .
    \end{split}
    \llabel{EQ78d}
\end{align}
Similarly to \eqref{EQ117}, Lemma~\ref{L04} yields 
\begin{align}
    \begin{split}
    J_{7}^{(3)}
    &\lec
    \mathcal{P}
    \Vert \partial_3(\partial')^\frac{1}{2}v\Vert_{L^2}^\frac{1}{2}
    \Vert \partial_{33}(\partial')^\frac{1}{2}v\Vert_{L^2}^\frac{1}{2}
    \Vert \partial'v\Vert_{L^2}^\frac{1}{2}
    \Vert \partial''v\Vert_{L^2}^\frac{1}{2}
    \\&\indeq
    +\PPs
    \Vert \partial_3(\partial')^\frac{1}{2}v\Vert_{L^2}^\frac{1}{2}
    \Vert \partial_{33}(\partial')^\frac{1}{2}v\Vert_{L^2}^\frac{1}{2}
    \Vert\partial_3(\partial')^\frac{1}{2}v\Vert_{L^2}
    \\&\indeq
    +\PPs
    \Vert (\partial')^\frac{1}{2}v\Vert_{L^2}^\frac{1}{2}
    \Vert \partial_{3}(\partial')^\frac{1}{2}v\Vert_{L^2}^\frac{1}{2}
    \left(\Vert \partial_3(\partial')^\frac{3}{2}v\Vert_{L^2}
    +\Vert\partial_{33}(\partial')^\frac{1}{2}v\Vert_{L^2}
    \right)
    \\&\indeq
    +\PPs
    \Vert (\partial')^\frac{1}{2}v\Vert_{L^2}^\frac{1}{2}
    \Vert \partial_{3}(\partial')^\frac{1}{2}v\Vert_{L^2}^\frac{1}{2}
    \left(\Vert \partial'v\Vert_{L^2}^\frac{1}{2}
    \Vert \partial''v\Vert_{L^2}^\frac{1}{2}
    +\Vert\partial_3(\partial')^\frac{1}{2}v\Vert_{L^2}
    \right)
   .
   \llabel{EQ119}
    \end{split}
\end{align}
\eold
Again, we use Lemma~\ref{L03} to deduce
\begin{align}
    \begin{split}
       J_{7}^{(3)}
    &\lesssim
    \mathcal{P}
    \Vert \partial_{333}v\Vert_{L^2}^\frac{1}{6}
    \Vert (\partial')^\frac{3}{4}v\Vert_{L^2}^\frac{1}{3}
    \Vert \partial_{333}v\Vert_{L^2}^\frac{1}{3}
    \Vert (\partial')^\frac{3}{2}v\Vert_{L^2}^\frac{1}{6}
    \Vert \partial'v\Vert_{L^2}^\frac{1}{2}
    \Vert \partial''v\Vert_{L^2}^\frac{1}{2}
    \\&\indeq
    +\PPs
    \Vert \partial_{333}v\Vert_{L^2}^\frac{1}{6}
    \Vert (\partial')^\frac{3}{2}v\Vert_{L^2}^\frac{1}{3}
    \Vert \partial_{333}v\Vert_{L^2}^\frac{1}{3}
    \Vert \partial'''v\Vert_{L^2}^\frac{1}{6}
    \Vert\partial_{333}v\Vert_{L^2}^\frac{1}{3}
    \Vert (\partial')^\frac{3}{4}v\Vert_{L^2}^\frac{2}{3}
    \\&\indeq
    +\mathcal{P}
    \Vert v\Vert_{L^2}^\frac{1}{4}
    \Vert \partial' v\Vert_{L^2}^\frac{1}{4}
    \Vert \partial_{333}v\Vert_{L^2}^\frac{1}{6}
    \Vert (\partial')^\frac{3}{4}v\Vert_{L^2}^\frac{1}{3}
    \left(\Vert \partial_{33}\partial'v\Vert_{L^2}^\frac{1}{2}
    \Vert \partial''v\Vert_{L^2}^\frac{1}{2}
    +\Vert\partial_{333}v\Vert_{L^2}^\frac{2}{3}
    \Vert (\partial')^\frac{3}{2} v\Vert_{L^2}^\frac{1}{3}
    \right)
    \\&\indeq
    +\PPs
    \Vert v\Vert_{L^2}^\frac{1}{4}
    \Vert \partial' v\Vert_{L^2}^\frac{1}{4}
    \Vert \partial_{333}v\Vert_{L^2}^\frac{1}{6}
    \Vert (\partial')^\frac{3}{2}v\Vert_{L^2}^\frac{1}{3}
    \left(\Vert \partial'v\Vert_{L^2}^\frac{1}{2}
    \Vert \partial''v\Vert_{L^2}^\frac{1}{2}
    +\Vert\partial_{333}v\Vert_{L^2}^\frac{1}{3}
    \Vert (\partial')^\frac{3}{4}v\Vert_{L^2}^\frac{2}{3}
    \right)
    \\&\lesssim
    (\epsilon+\PPs)
    \left(\Vert\partial_{333}v\Vert_{L^2}
    \right)
    +\epsilon \Vert\partial_{33}\partial'v\Vert_{L^2}
    +\PPs\Vert\partial'''v\Vert_{L^2}
    +\PPs
    .
        \llabel{EQ78e}
    \end{split}
\end{align}
Combining this with \eqref{EQ77} and \eqref{EQ78b}, we arrive at
\begin{align}
    \begin{split}
    \Vert \partial_{333}v\Vert_{L^2}
    \lesssim&
    (\PPs+\epsilon)
    (\Vert\partial_{33}v\Vert_{L^2}
    +\Vert\partial_{33}\partial'v\Vert_{L^2}
    +\Vert\partial_{333}v\Vert_{L^2})
    \\&\indeq
    +\Vert \partial_3\partial''v\Vert_{L^2}
    +\Vert\partial_t\partial_3v\Vert_{L^2}
    +\PPs\Vert\partial'''v\Vert_{L^2}
    +\mathcal{P}
    ,
    \end{split}
    \label{EQ79a}
\end{align}
completing the estimate for $\Vert \partial_{333}v\Vert_{L^2}$.
\end{proof}

\startnewsection{Pressure estimates}{sec15}
The purpose of this section is to obtain a pressure estimate,
stated in the following lemma.

\cole%%%out
\begin{Lemma}
\label{L07}
Under the assumptions of Theorem~\ref{T01}
and with $\epsilon\in(0,1/2]$,
we have
the pressure estimate
\begin{align}
    \Vert p\Vert_{H^2({\mathbb{T}^2})}
    \lesssim  
    \Vert v\Vert_{H^3} (\PPs+\epsilon)
    +\mathcal{P}
    .
    \label{EQ97}
\end{align}
\end{Lemma}
\colb%%%out

\begin{proof}[Proof of Lemma~\ref{L07}]
Multiplying \eqref{EQ87} by $\partial_{3} \phi$, then integrating from $0$ to $1$ in the $x_{3}$ variable and noting that  ${w}=0 $, $\phi_{t}=0$,  and $\Delta_{x_{1}, x_{2}} {w} =0$ on $x_{3}=0 $, while ${v}=0$ on $x_{3}=1$ and on $x_{3}=0$, we obtain 
\begin{align}
 h \Delta_{x_{1}, x_{2}} {p} =
 {w}_{t}\big|_{z=1}
 - a_{33} \phi_{t} \partial_{3} {w}\Big|_{z=1}   
 + F 
,
\label{EQ88}      
\end{align}
where 
  \begin{align}
   F
   = - \Delta_{a}  {w}\Big|_{ z=1} +  \Delta_{a}  {w}\Big|_{ z=0} + \int_{0}^{1}\left( -  \partial_{3}\phi \nabla_{a,\HH} {v} : \nabla^{T}_{a, \HH}  {v} 
   + {w} \partial_{3}\left(\frac{1}{\partial_{3} \phi} \partial_{3} {w}\right)   
   -  2 \nabla_{a,\HH} {w} \cdot \partial_{3}  {v} 
   \right) \, dx_{3}.
   \label{EQ89}
\end{align}
Performing integration by parts on the third term of the integral
in~\eqref{EQ89}, we get
\begin{align}
    \begin{split}
        \int_0^1w\partial_3\left(\frac{1}{\partial_3\phi}\partial_3w\right)\,dx_3
        &=
        -\int_0^1\frac{1}{\partial_3\phi}(\partial_3w)^2\,dx_3
        +\frac{1}{\partial_3\phi}w\partial_3w\Big|_{z=1}
        -\frac{1}{\partial_3\phi}w\partial_3w\Big|_{z=0}
        \\&
        =
        -\int_0^1\frac{1}{\partial_3\phi}(\partial_3w)^2\,dx_3
        +\frac{1}{\partial_3\phi}w\partial_3w\Big|_{z=1}
        \\&
        =-\int_0^1\frac{1}{\partial_3\phi}(\partial_3w)^2\,dx_3
        +a_{33}\phi_t\partial_3w\Big|_{z=1}
        ,
    \end{split}
    \label{EQ90}
\end{align}
where we have used \eqref{EQ08}, \eqref{EQ09}, and~\eqref{EQ12}.
Observing that the last term in \eqref{EQ90} cancels with the second-to-last term on the right-hand side of \eqref{EQ88}, we set
\begin{align}
    \tilde{F}=F- a_{33} \phi_{t} \partial_{3} {w}\Big|_{z=1} 
   .
    \llabel{EQ90b}
\end{align}
Using \eqref{EQ07}, together with the reductions in \eqref{EQ89} and \eqref{EQ90}, we arrive at
\begin{align}
    (I-\Delta_{\HH})p=\Delta^2_{\HH}h+(h-1)\Delta_{\HH}p-\tilde{F}.
    \llabel{EQ91}
\end{align}
Hence, in the case where $h$ is close to $1$,
\begin{align}
    \Vert p\Vert_{H^2({\mathbb{T}^2})}\lesssim  \Vert h\Vert_{H^4({\mathbb{T}^2})}+\Vert \tilde{F}\Vert_{L^2({\mathbb{T}^2})}
    .
    \label{EQ92}
\end{align}

To proceed, we expand the term $F$ as
\begin{align}
\begin{split}
    \tilde{F}=&-\Delta_aw\Big|_{z=1}+\Delta_aw\Big|_{z=0}\\
    &-\int_0^1\partial_3\phi(2(a_{j2}\partial_jv_1)(a_{j1}\partial_jv_2)+(a_{j1}\partial_jv_1)^2+(a_{j2}\partial_jv_2)^2)\,dx_3\\
    &-\int_0^1\frac{1}{\partial_3\phi}(\partial_3w)^2\,dx_3-2\int_0^1\nabla_{a,\HH}w\partial_3v\,dx_3
    .
\end{split}
\llabel{EQ93}
\end{align}
We first estimate $K_1$ as
\begin{align}
\begin{split}
    K_1=
    &(-\Delta_{\HH}h_t)
    -\sum_{i,j=1}^3
    \left(a_{ji}(\partial_ja_{3i})\sum_{k=1}^2b_{3k}\partial_3v_k\right)\Big|_{z=1}\\
    &-\left(\sum_{i,j=1}^3a_{ji}a_{3i}\partial_j(\sum_{k=1}^2b_{3k}\partial_3v_k)\right)\Big|_{z=1}
    -\left(\sum_{i=1}^2a_{3i}\partial_i(\sum_{k=1}^2b_{3k}\partial_3v_k)\right)\Big|_{z=1}
    .
    \end{split}
    \llabel{EQ94}
\end{align}
In this step, \eqref{EQ05} and the divergence-free condition \eqref{EQ01} have been applied to simplify the expression. 
The following estimate relies on the application of the trace theorem along with Lemma~\ref{L01} and~\ref{L02}
\begin{align}
\begin{split}
    \Vert K_1\Vert_{L^2_\HH}&
    \lesssim
    \Vert h_t\Vert_{H^2}
    -\sum_{i,j=1}^3
    \Vert a_{ji}(\partial_ja_{3i})\big|_{z=1}\Vert_{L^\infty}
    \sum_{k=1}^2\Vert b_{3k}\big|_{z=1}\Vert_{L^\infty}
    \Vert\partial_3v_k\big|_{z=1}\Vert_{L^2}
    \\&\indeq
    -\sum_{i,j=1}^3\Vert a_{ji}|_{z=1}\Vert_{L^\infty}
    \Vert a_{3i}|_{z=1}\Vert_{L^\infty}
    \sum_{k=1}^2\left(
    \Vert\partial_jb_{3k}|_{z=1}\Vert_{L^\infty}
    \Vert\partial_3v_k|_{z=1}\Vert_{L^2}
    +
    \Vert b_{3k}|_{z=1}\Vert_{L^\infty}
    \Vert\partial_j\partial_3v_k|_{z=1}\Vert_{L^2}
    \right)
    \\&\indeq
    -\sum_{i=1}^2
    \Vert a_{3i}|_{z=1}\Vert_{L^\infty}
    \sum_{k=1}^2\left(
    \Vert \partial_ib_{3k}|_{z=1}\Vert_{L^\infty}
    \Vert\partial_3v_k|_{z=1}\Vert_{L^2}
    +
    \Vert b_{3k}|_{z=1}\Vert_{L^\infty}
    \Vert\partial_i\partial_3v_k|_{z=1}\Vert_{L^2}
    \right)
    \\&
    \lesssim \Vert v\Vert_{H^3} \PPs +\mathcal{P}
    .
    \llabel{EQ941}
    \end{split}
\end{align}
Noting that $K_2$ has more favorable properties than 
$K_1$ due to the no-slip condition \eqref{EQ09}, a further application of \eqref{L02} yields
\begin{align}
    \begin{split}
     K_1+K_2\lesssim \Vert v\Vert_{H^3} \PPs +\mathcal{P}
     .
    \end{split}
    \label{EQ95}
\end{align}
For $K_3,\ K_4$, and $K_5$, we apply Minkowski’s inequality and Lemma~\ref{L03}. As an illustration, for $K_3$, we obtain
\begin{align}
\begin{split}
    \Vert K_3\Vert_{L^2_{\HH}}
    =&\left\Vert \int_0^1\partial_3\phi(2(a_{j2}\partial_jv_1)(a_{j1}\partial_jv_2)
    +(a_{j1}\partial_jv_1)^2
    +(a_{j2}\partial_jv_2)^2)\,dx_3\right\Vert_{L^2_{\HH}}
    \\ \lesssim& 
    \int_0^1\Vert \partial_3\phi(2(a_{j2}\partial_jv_1)(a_{j1}\partial_jv_2)+(a_{j1}\partial_jv_1)^2
    +(a_{j2}\partial_jv_2)^2)\Vert_{L^2_{\HH}}\,dx_3
    \\= & 
    \left\Vert \Vert \partial_3\phi(2(a_{j2}\partial_jv_1)(a_{j1}\partial_jv_2)
    +(a_{j1}\partial_jv_1)^2
    +(a_{j2}\partial_jv_2)^2)\Vert_{L^2_{\HH}}\right\Vert_{L^1_z}
    \\ \lesssim & 
    \Vert \partial_3\phi(2(a_{j2}\partial_jv_1)(a_{j1}\partial_jv_2)
    +(a_{j1}\partial_jv_1)^2+(a_{j2}\partial_jv_2)^2)\Vert_{L^2}
    ,
    \end{split}
    \llabel{EQ96a}
\end{align}
from where
\begin{align}
    \begin{split}
      \Vert K_3\Vert_{L^2_{\HH}}
     \lesssim & 
    \PPs
    \left\Vert \Vert \partial_3v\Vert_{L^\infty_z}\right\Vert_{L^2_{\HH}}
    \left\Vert \Vert \partial_3 v\Vert_{L^2_z}\right\Vert_{L^\infty_{\HH}}
    +\PPs
    \left\Vert \Vert \partial_3v\Vert_{L^2_z}\right\Vert_{L^\infty_{\HH}}
    \left\Vert \Vert \partial'v\Vert_{L^\infty_z}\right\Vert_{L^2_{\HH}}
    \\&\indeq
    +\mathcal{P}
    \left\Vert \Vert \partial'v\Vert_{L^\infty_z}\right\Vert_{L^2_{\HH}}
    \left\Vert \Vert \partial' v\Vert_{L^2_z}\right\Vert_{L^\infty_{\HH}}
    \\\lesssim&
    \PPs \Vert\partial_3v\Vert^\frac{1}{2}_{L^2}
    \Vert\partial_{33}v\Vert^\frac{1}{2}_{L^2}
    \Vert\partial_3(\partial')^{(1-\delta)}v\Vert_{L^2}^\frac{1}{2}
    \Vert\partial_3(\partial')^{(1+\delta)}v\Vert_{L^2}^\frac{1}{2}
    \\&+\PPs 
    \Vert\partial_3(\partial')^{(1-\delta)}v\Vert_{L^2}^\frac{1}{2}
    \Vert\partial_3(\partial')^{(1+\delta)}v\Vert_{L^2}^\frac{1}{2}
    \Vert\partial'v\Vert_{L^2}^\frac{1}{2}
    \Vert\partial'\partial_3v\Vert_{L^2}^\frac{1}{2}
   \\&\indeq
    +\mathcal{P} 
    \Vert\partial'v\Vert_{L^2}^\frac{1}{2}
    \Vert\partial'\partial_3v\Vert_{L^2}^\frac{1}{2}
    \Vert\partial'v\Vert_{L^2}^\frac{1}{2}
    \Vert\partial'''v\Vert_{L^2}^\frac{1}{2}
    \\\lesssim&
    \PPs
    \left(
    \Vert v\Vert_{L^2}^\frac{1}{3}
    \Vert \partial_{333}v\Vert_{L^2}^\frac{1}{6}
    \right)
    \left(
    \Vert v\Vert_{L^2}^\frac{1}{6}
    \Vert \partial_{333}v\Vert_{L^2}^\frac{1}{3}
    \right)
   \\&\indeq\indeq\indeq\indeq\times
    \left(
    \Vert (\partial')^{\frac{3}{2}(1-\delta)}v\Vert_{L^2}^\frac{1}{3}
    \Vert \partial_{333}v\Vert^\frac{1}{6}_{L^2}
    \right)
    \left(
    \Vert (\partial')^{\frac{3}{2}(1+\delta)}v\Vert_{L^2}^\frac{1}{3}
    \Vert \partial_{333}v\Vert^\frac{1}{6}_{L^2}
    \right)
    \\&
    +\PPs
    \left(
    \Vert (\partial')^{\frac{3}{2}(1-\delta)}v\Vert_{L^2}^\frac{1}{3}
    \Vert \partial_{333}v\Vert^\frac{1}{6}_{L^2}
    \right)
    \left(
    \Vert (\partial')^{\frac{3}{2}(1+\delta)}v\Vert_{L^2}^\frac{1}{3}
    \Vert \partial_{333}v\Vert^\frac{1}{6}_{L^2}
    \right)
    \Vert\partial'v\Vert_{L^2}^\frac{1}{2}
    \left(
    \Vert v\Vert_{L^2}^\frac{1}{6}
    \Vert \partial_{333}v\Vert_{L^2}^\frac{1}{3}
    \right)
    \\&
    +\mathcal{P}
    \Vert\partial'v\Vert_{L^2}
    \left(
    \Vert\partial''v\Vert_{L^2}^\frac{1}{4}
    \Vert\partial_{33}v\Vert_{L^2}^\frac{1}{4}
    \right)
    \Vert\partial'''v\Vert_{L^2}^\frac{1}{2}
    ;
    \end{split}
    \label{EQ96b}
\end{align}
in this step
$\delta$ can be chosen arbitrarily in $\left(0,\frac{1}{3}\right)$.
Upon collecting similar terms and invoking Young’s inequality, we observe that $\Vert\partial_3(\partial')^{(1\pm\delta)}v\Vert_{L^2}$ is bounded by
$\mathcal{P}$, we deduce the following estimate
\begin{align}
    \begin{split}
    \Vert K_3\Vert_{L^2_{\HH}}
     \lesssim & \PPs 
     \Vert (\partial')^{\frac{3}{2}(1-\delta)}v\Vert_{L^2}^\frac{1}{3}  
     \Vert (\partial')^{\frac{3}{2}(1+\delta)}v\Vert_{L^2}^\frac{1}{3}
     \Vert v\Vert_{L^2}^\frac{1}{2}
     \Vert\partial_{333}v\Vert^\frac{5}{6}
     \\&
     +\PPs 
     \Vert (\partial')^{\frac{3}{2}(1-\delta)}v\Vert_{L^2}^\frac{1}{3}  
     \Vert (\partial')^{\frac{3}{2}(1+\delta)}v\Vert_{L^2}^\frac{1}{3}
     \Vert v\Vert_{L^2}^\frac{1}{6}
     \Vert \partial'v\Vert_{L^2}^\frac{1}{2}
     \Vert\partial_{333}v\Vert^\frac{2}{3}
     \\&
     +\mathcal{P}
     \Vert\partial'v\Vert_{L^2}
    \Vert\partial''v\Vert_{L^2}^\frac{1}{4}
    \Vert\partial_{33}v\Vert_{L^2}^\frac{1}{4}
    \Vert\partial'''v\Vert_{L^2}^\frac{1}{2}
    \\\lesssim&
    (\PPs+\epsilon)\Vert v\Vert_{H^3} 
    +\mathcal{P}
    .
    \end{split}
    \label{EQ96c}
\end{align}
The treatment of $K_4$ and $K_5$ relies on the incompressibility condition
\begin{align}
    \begin{split}
    \Vert K_4\Vert_{L^2_{\HH}}
    =&
    \left\Vert-\int_0^1\frac{1}{\partial_3\phi}(\partial_3w)^2\,dx_3\right\Vert_{L^2_\HH}
    =
    \left\Vert\int_0^1a_{33}(\partial_1v_1+\partial_2v_2+b_{31}\partial_3v_1+b_{32}\partial_3v_2)^2\,dx_3\right\Vert_{L^2_\HH}
    \\\lesssim &
    \left\Vert a_{33}(\partial_1v_1+\partial_2v_2+b_{31}\partial_3v_1+b_{32}\partial_3v_2)^2\right\Vert_{L^2}
    .
    \end{split}
    \llabel{EQ96d}
\end{align}
After this step, we obtain a similar estimate as for~\eqref{EQ96b}
\begin{align}
    \begin{split}
    \Vert K_4\Vert_{L^2_{\HH}}
    \lesssim & 
    \PPs
    \left\Vert \Vert \partial_3v\Vert_{L^\infty_z}\right\Vert_{L^2_{\HH}}
    \left\Vert \Vert \partial_3 v\Vert_{L^2_z}\right\Vert_{L^\infty_{\HH}}
    +\PPs
    \left\Vert \Vert \partial_3v\Vert_{L^2_z}\right\Vert_{L^\infty_{\HH}}
    \left\Vert \Vert \partial'v\Vert_{L^\infty_z}\right\Vert_{L^2_{\HH}}
    \\&
    +\mathcal{P}
    \left\Vert \Vert \partial'v\Vert_{L^\infty_z}\right\Vert_{L^2_{\HH}}
    \left\Vert \Vert \partial' v\Vert_{L^2_z}\right\Vert_{L^\infty_{\HH}}
    \Vert K_3\Vert_{L^2_{\HH}}
     \\\lesssim & \PPs 
     \Vert (\partial')^{\frac{3}{2}(1-\delta)}v\Vert_{L^2}^\frac{1}{3}  
     \Vert (\partial')^{\frac{3}{2}(1+\delta)}v\Vert_{L^2}^\frac{1}{3}
     \Vert v\Vert_{L^2}^\frac{1}{2}
     \Vert\partial_{333}v\Vert^\frac{5}{6}
     \\&
     +\PPs 
     \Vert (\partial')^{\frac{3}{2}(1-\delta)}v\Vert_{L^2}^\frac{1}{3}  
     \Vert (\partial')^{\frac{3}{2}(1+\delta)}v\Vert_{L^2}^\frac{1}{3}
     \Vert v\Vert_{L^2}^\frac{1}{6}
     \Vert \partial'v\Vert_{L^2}^\frac{1}{2}
     \Vert\partial_{333}v\Vert^\frac{2}{3}
     \\&
     +\mathcal{P}
     \Vert\partial'v\Vert_{L^2}
    \Vert\partial''v\Vert_{L^2}^\frac{1}{4}
    \Vert\partial_{33}v\Vert_{L^2}^\frac{1}{4}
    \Vert\partial'''v\Vert_{L^2}^\frac{1}{2}
    \\\lesssim&
    (\PPs+\epsilon)\Vert v\Vert_{H^3} 
    +\mathcal{P}
    .
    \end{split}
    \label{EQ96e}
\end{align}
For $K_5$ we need to integrate by parts and apply Lemma~\ref {L04}
\begin{align}
    \begin{split}
    \Vert K_5\Vert_{L^2_{\HH}}=&
    \left\Vert-2\int_0^1\nabla_{a,\HH}w\partial_3v\,dx_3\right\Vert_{L^2}\lesssim
    \left\Vert2\int_0^1\nabla_\HH w\partial_3v\,dx_3\right\Vert_{L^2}
    +\left\Vert2\sum_{k=1}^2\int_0^1a_{3k}\partial_3w\partial_3v_k\,dx_3\right\Vert_{L^2}
    \\ \lesssim&
    \left\Vert\nabla_{\HH}(\partial_1v_1+\partial_2v_2+b_{31}\partial_3v_1+b_{32}\partial_3v_2)v
    \right\Vert_{L^2}
    \\&
    +\left\Vert \sum_{k=1}^2a_{3k}(\partial_1v_1+\partial_2v_2+b_{31}\partial_3v_1+b_{32}\partial_3v_2)\partial_3v_k\right\Vert_{L^2}
    \\ \lesssim&
    \Vert \nabla_\HH(\partial_1v_1+\partial_2v_2+b_{31}\partial_3v_1+b_{32}\partial_3v_2)\Vert_{L^2}\Vert v\Vert_{L^\infty}
    +\PPs
    \left\Vert \Vert \partial_3v\Vert_{L^\infty_z}\right\Vert_{L^2_{\HH}}
    \left\Vert \Vert \partial_3 v\Vert_{L^2_z}\right\Vert_{L^\infty_{\HH}}
    \\\lesssim&
    (\Vert\partial''v\Vert_{L^2}+
    \PPs\Vert\partial_{333}v\Vert^\frac{1}{3}
    \Vert(\partial')^\frac{3}{2}v\Vert^\frac{2}{3})
    \Vert v\Vert_{L^2}^\frac{1}{4}
    \Vert \bard''v\Vert_{L^2}^\frac{1}{4}
    \Vert \partial_3v\Vert_{L^2}^\frac{1}{4}
    \Vert \partial_3 \bard''v\Vert_{L^2}^\frac{1}{4}
    \\&+\PPs
    \left\Vert \Vert \partial_3v\Vert_{L^\infty_z}\right\Vert_{L^2_{\HH}}
    \left\Vert \Vert \partial_3 v\Vert_{L^2_z}\right\Vert_{L^\infty_{\HH}}
    .
    \llabel{EQ96f}
    \end{split}
\end{align}
The first term can again be handled by Young’s inequality, while the second term has already been treated in~\eqref{EQ96b} and~\eqref{EQ96c}. Therefore, we have
\begin{align}
    \Vert K_5\Vert_{L^2_{\HH}}\lesssim (\PPs+\epsilon)\Vert v\Vert_{H^3} 
    +\mathcal{P}
    .
    \label{EQ96g}
\end{align}
Combining the estimates \eqref{EQ92}, \eqref{EQ95}, \eqref{EQ96c}, \eqref{EQ96e}, and
\eqref{EQ96g}, we obtain the pressure estimate~\eqref{EQ97}.
\end{proof}

\startnewsection{Time derivative estimates}{sec16}
In this section, we obtain the following time derivative estimate.

\cole%%%out
\begin{Lemma}
\label{L08}
Under the assumptions of Theorem~\ref{T01}
and with $\epsilon\in(0,1/2]$,
we have
the time derivative estimate
\begin{align}
\begin{split}
       &\frac{1}{2}\int J \partial_tv_{\alpha}\partial_tv_{\alpha}
        +\iint J\nablah\partial_tv_\alpha \nablah\partial_tv_\alpha
        +\iint\left(\frac{1+b_{31}^2+b_{32}^2}{J}\right)\partial_3\partial_tv_\alpha
        \partial_3\partial_tv_\alpha
        \\&\indeq\indeq
        +\frac{1}{2}\int (\partial_{tt}h)^2
        +\frac{1}{2}\int (\Delta_{\HH}h_t)^2
        \\&\indeq
        \lesssim
        \frac{1}{2}\int J(0) \partial_tv_{\alpha}(0)
        \partial_tv_{\alpha}(0)
        +\frac{1}{2}\int (\partial_{tt}h)^2(0)
        +\frac{1}{2}\int (\Delta_{\HH}h_t)^2(0)
        \\&\indeq\indeq
        +\int_0^t\mathcal{P}
        +\int_0^t(\epsilon+\PPs)(\Vert v_t\Vert_{H^1}^2+\Vert v\Vert_{H^3}^2)
        +\int_0^t\epsilon(\Vert p\Vert ^2_{H^2}+\Vert v\Vert_{H^2}^2)
        +\int \partial_tb_{k\alpha}(0)p(0)\partial_k v_\alpha(0)
        .
\end{split}
\label{EQ103}
\end{align}
\end{Lemma}
\colb%%%out

\begin{proof}[Proof of Lemma~\ref{L08}]
We apply $\partial_t$ to the equation \eqref{EQ50} and then take the  $L^2$-inner product with $\partial_tv_\alpha$. Summing $\alpha$ from 1 to 2, the energy estimate for $\partial_tv_\alpha$ can be written as
\begin{align}
\begin{split}
       &\frac{1}{2}\frac{d}{dt}\int J \partial_tv_{\alpha}\partial_tv_{\alpha}
        +\int J\nablah\partial_tv_\alpha \nablah\partial_tv_\alpha
        +\int\frac{1+b_{31}^2+b_{32}^2}{J}\partial_3\partial_tv_\alpha
        \partial_3\partial_tv_\alpha
        \\&\indeq\indeq
        +\frac{1}{2}\frac{d}{dt}\int (\partial_{tt}h)^2
        +\frac{1}{2}\frac{d}{dt}\int (\Delta_{\HH}h_t)^2
        \\&\indeq
        = -\int \nabla_{\HH}J\nabla_{\HH}\partial_tv_\alpha
            \partial_tv_\alpha
            +\int J_t\Delta_{\HH}v_\alpha
            \partial_tv_\alpha
        \\&\indeq\indeq
        +2\int\partial_tb_{3\beta}\partial_{3\beta} v_\alpha
        \partial_tv_\alpha
        +2\int b_{3\beta}\partial_{3\beta}\partial_tv_\alpha
        \partial_tv_\alpha
        \\&\indeq\indeq
        +\int \partial_t(b_{kl}\partial_ka_{3l})\partial_3v_\alpha
        \partial_tv_\alpha
        +
        \int b_{kl}\partial_ka_{3l}\partial_3\partial_tv_\alpha
        \partial_tv_\alpha
        \\&\indeq\indeq
        +\int \partial_t\left(\frac{1+b_{31}^2+b_{32}^2}{J}\right)\partial_3^2v_\alpha
        \partial_tv_\alpha
        -\int \partial_3\left(\frac{1+b_{31}^2+b_{32}^2}{J}\right)\partial_3\partial_tv_\alpha
        \partial_tv_\alpha
        \\&\indeq\indeq
        -\int \partial_t v_\gamma b_{j\gamma}\partial_jv_\alpha
        \partial_tv_\alpha
        -\int v_\gamma \partial_tb_{j\gamma}\partial_jv_\alpha
        \partial_tv_\alpha
        \\&\indeq\indeq
        -\int \partial_t(v_3-\phi_t)\partial_3v_\alpha
        \partial_tv_\alpha
        \\&\indeq\indeq
        +\int \partial_tb_{k\alpha}p\partial_k\partial_tv_\alpha
        +\int \partial_tb_{k\alpha}\partial_tp\partial_kv_\alpha
        \\&\indeq
        =\sum_{i=1}^{13}T_i
        .
\end{split}
\label{EQ98}
\end{align}
We estimate the terms above in order. For the first two terms, we
invoke Lemma~\ref{L01} to estimate
\begin{align}
\begin{split}
    T_{1}+T_{2}
    &= -\int \nabla_{\HH}J\nabla_{\HH}\partial_tv_\alpha
            \partial_tv_\alpha
        +\int J_t\Delta_{\HH}v_\alpha
        \partial_tv_\alpha
    \\&\lesssim 
        \Vert\nabla_{\HH}J\Vert_{L^\infty}
        \Vert \nabla_{\HH}\partial_tv_\alpha\Vert_{L^2}
        \Vert\partial_tv_\alpha\Vert_{L^2}
        +
        \Vert J_t\Vert_{L^3}
        \Vert\Delta_{\HH}v_\alpha\Vert_{L^6}
        \Vert\partial_tv_\alpha\Vert_{L^2}
    \\&\lesssim
        \mathcal{P}+\epsilon\Vert\partial_tv\Vert_{H^1}^2
        .
        \llabel{EQ991}
\end{split}
\end{align}
The third term on the right-hand side of \eqref{EQ98} can be estimated by Hölder’s inequality,
\begin{align}
\begin{split}
    T_{3}
    &
    =2\int\partial_tb_{3\beta}\partial_{3\beta} v_\alpha
        \partial_tv_\alpha
    \lesssim
    \Vert \partial_t b\Vert_{L^6}
    \Vert \partial_{3\beta} v_\alpha\Vert_{L^3}
    \Vert \partial_tv_\alpha \Vert_{L^2}
    \\&\lesssim
        \mathcal{P}+\epsilon\Vert v\Vert_{H^3}^2
    .
    \llabel{EQ992}
\end{split}
\end{align}
The fourth term is treated by an integration by parts argument supplemented Lemma~\ref{L02}, because the highest-order derivative involved exceeds the level controlled by our previous estimates,
  \begin{align}
    \begin{split}
    T_4&=2\int b_{3\beta}\partial_{3\beta}\partial_tv_\alpha
        \partial_tv_\alpha
    \\&=-2\int \partial_\beta b_{3\beta}
        \partial_3\partial_tv_\alpha
        \partial_tv_\alpha
        -2\int b_{3\beta}
        \partial_3\partial_tv_\alpha
        \partial_\beta\partial_tv_\alpha
    \\&\lesssim
    \Vert \partial_\beta b_{3\beta} \Vert_{L^\infty}
    \Vert \partial_3\partial_tv_\alpha\Vert_{L^2}
    \Vert \partial_tv_\alpha \Vert_{L^2}
    +
    \Vert b_{3\beta} \Vert_{L^\infty}
    \Vert \partial_3\partial_tv_\alpha\Vert_{L^2}
    \Vert \partial_\beta\partial_tv_\alpha \Vert_{L^2}
    \\&\lesssim
    \mathcal{P}+(\PPs+\epsilon)\Vert v_t\Vert_{H^1}^2
        .
        \llabel{EQ993}
    \end{split}
\end{align}
The combined contribution of the fifth and sixth terms can be controlled by Hölder’s inequality
\begin{align}
    \begin{split}
        T_{5}+T_{6}&=
        \int \partial_t(b_{kl}\partial_ka_{3l})\partial_3v_\alpha
        \partial_tv_\alpha
        +
        \int b_{kl}\partial_ka_{3l}\partial_3\partial_tv_\alpha
        \partial_tv_\alpha
        \\&\lesssim 
        \Vert \partial_tb_{kl}\Vert_{L^6}
        \Vert \partial_ka_{3l}\|_{L^\infty}
        \Vert \partial_3v_\alpha\Vert_{L^3}
        \Vert \partial_t v_\alpha\Vert_{L^2}
        +
        \Vert b_{kl}\Vert_{L^\infty}
        \Vert \partial_t\partial_ka_{3l}\|_{L^3}
        \Vert \partial_3v_\alpha\Vert_{L^6}
        \Vert \partial_t v_\alpha\Vert_{L^2}
        \\&\indeq+\Vert b_{kl}\Vert_{L^\infty}
        \Vert\partial_k a_{3l}\Vert_{L^\infty}
        \Vert\partial_3\partial_t v_\alpha\Vert_{L^2}
        \Vert \partial_t v_\alpha\Vert_{L^2}
        \\&\lesssim
        \mathcal{P}+\epsilon(\Vert v\Vert_{H^3}^2+\Vert v_t\Vert_{H^1}^2)
        .
        \llabel{EQ994}
    \end{split}
\end{align}
For the seventh and eighth terms, the very same argument used for the first two terms yields the bounds
\begin{align}
    \begin{split}
        T_{7}+T_{8}\lesssim \mathcal{P}+\epsilon(\Vert v\Vert_{H^3}^2+\Vert v_t\Vert_{H^1}^2)
    .
    \llabel{EQ995}
    \end{split}
\end{align}
The nonlinear terms are treated in a similar manner using Hölder’s inequality. As an example, we present the estimate for the ninth term
\begin{align}
    \begin{split}
        T_{9}&
        =-\int \partial_t v_\gamma 
        b_{j\gamma}
        \partial_jv_\alpha
        \partial_tv_\alpha
        \lesssim
        \Vert\partial_tv_\gamma\Vert_{L^3}
        \Vert b_{j\gamma}\Vert_{L^\infty}
        \Vert \partial_j v_\alpha\Vert_{L^6}
        \Vert \partial_tv_\alpha\Vert_{L^2}
        \\&\lesssim
        \mathcal{P}+\epsilon(\Vert v_t\Vert^2_{H^1}+\Vert v\Vert_{H^3}^2)
        .
        \llabel{EQ996}
    \end{split}
\end{align}
To estimate the tenth term, we make use of the embedding inequality~\eqref{EQ151}, which follows from the anisotropic estimate,
\begin{align}
    \begin{split}
    T_{10}&=
    -\int v_\gamma 
        \partial_tb_{j\gamma}
        \partial_jv_\alpha
        \partial_tv_\alpha
    \lesssim
        \Vert v_\gamma\Vert_{L^\infty}
        \Vert \partial_t b_{j\gamma}\Vert_{L^3}
        \Vert \partial_jv_\alpha\Vert_{L^6}
        \Vert \partial_t v_\alpha\Vert_{L^2}
    \\&\lesssim
        \Vert h_t\Vert_{H^1}
        \Vert v\Vert_{L^2}^\frac{1}{4}
        \Vert \bard''v\Vert_{L^2}^\frac{1}{4}
        \Vert \partial_3v\Vert_{L^2}^\frac{1}{4}
        \Vert \partial_3 \bard''v\Vert_{L^2}^\frac{1}{4}
        \Vert v\Vert_{H^3}
        \Vert \partial_tv\Vert_{L^2}
    \lesssim
        \mathcal{P}+\epsilon(\Vert v_t\Vert^2_{H^1}+\Vert v\Vert_{H^3}^2)
        .
        \llabel{EQ997}
    \end{split}
\end{align}
In a similar manner, the eleventh term is estimated by invoking the anisotropic interpolation inequality in Lemma~\ref{L03},
\begin{align}
    \begin{split}
        T_{11}&=-\partial_t(w-\phi_t)\partial_3v_\alpha\partial_tv_\alpha
        \\&
        \lesssim
        \left\Vert \Vert \partial_tw\Vert_{L^\infty_z}\right\Vert_{L^2_{\HH}}
        \left\Vert \Vert \partial_3v_\alpha\Vert_{L^2_z}\right\Vert_{L^4_{\HH}}
        \left\Vert \Vert \partial_tv_\alpha\Vert_{L^2_z}\right\Vert_{L^4_{\HH}}
        +\Vert \phi_{tt}\Vert_{L^3}
        \Vert \partial_3v_\alpha\Vert_{L^6}
        \Vert \partial_tv_\alpha\Vert_{L^2}
        \\&
        \lesssim
        \Vert \partial_t\nabla v\Vert_{L^2}
        \Vert \partial_{333}v\Vert ^\frac{1}{3}_{L^2}
        \Vert (\partial')^\frac{3}{4}v\Vert ^\frac{2}{3}_{L^2}
        \Vert v_t\Vert_{L^2}^\frac{1}{2}
        \Vert \partial_t\partial'v\Vert_{L^2}^\frac{1}{2}
        \mathcal{P}
        +\Vert v\Vert_{H^1}^\frac{1}{2}
        \Vert v\Vert_{H^2}^\frac{1}{2}
        \mathcal{P}
        \\&
        \lesssim
        \mathcal{P}+\epsilon (\Vert v\Vert_{H^3}^2+\Vert v_t\Vert_{H^1}^2)
        .
    \end{split}
    \llabel{EQ100}
\end{align}
The estimate for $T_{12}$ only requires integration by parts with respect to $\partial_k$, for $k$ is summed from $1$ to $3$, so that we can apply Piola's identity and write
\begin{align}
    \begin{split}
        \int \partial_tb_{k\alpha}p\partial_k\partial_tv_\alpha
        &=
        -\int\partial_t\partial_kb_{k\alpha}p\partial_t v_\alpha
        -\int\partial_tb_{k\alpha}\partial_k p \partial_tv_\alpha
        =
        -\int\partial_tb_{k\alpha}\partial_k p \partial_tv_\alpha
        \\&\lesssim
        \Vert \partial_tb_{k\alpha}\Vert_{L^3}
        \Vert\partial_kp\Vert_{L^6}
        \Vert \partial_tv_\alpha\Vert_{L^2}
        \lesssim
        \mathcal{P}+\epsilon\Vert p\Vert_{H^2}^2
        ,
    \end{split}
   \llabel{EQ134}
\end{align}
whereas $T_{13}$, after the same operation, additionally requires integration by parts in time,
\begin{align}
  \begin{split}
    -\int_0^t\int_{\Omega}
    \partial_tb_{k\alpha} \partial_tp\partial_k v_\alpha
    &=-\int_\Omega \partial_tb_{k\alpha}p\partial_k v_\alpha\Big|_0^t
    +\int_0^t\int_{\Omega}\partial_{tt}b_{k\alpha}p\partial_kv_\alpha
    +\int_0^t\int_{\Omega}\partial_t b_{k\alpha}p\partial_k\partial_tv_\alpha
    \\&= \left(\tilde{T}_{13}^{(1)}(t)-\tilde{T}_{13}^{(1)}(0)\right)+\tilde{T}_{13}^{(2)}+\tilde{T}_{13}^{(3)}
    .
  \end{split}
    \llabel{EQ101}
\end{align}
Integration by parts and Piola’s identity can be applied to estimate all three terms, when the first integral is evaluated at the time $t$,
\begin{align}
    \begin{split}
        \tilde{T}_{13}^{(1)}(t)&=-\int_\Omega \partial_tb_{k\alpha}p\partial_k v_\alpha
        =\int \partial_t\partial_kb_{k\alpha}
        p
        v_\alpha
        +\int \partial_tb_{k\alpha}
        \partial_k p
        v_\alpha
        =\int \partial_tb_{k\alpha}
        \partial_k p
        v_\alpha
        \\&\lesssim
        \Vert \partial_tb_{k\alpha}\Vert_{L^3}
        \Vert \partial_k p \Vert_{L^6}
        \Vert v_\alpha\Vert_{L^2}
        \lesssim
        \mathcal{P}+\epsilon\Vert p\Vert_{H^2}^2
        .
        \llabel{EQ102a}
    \end{split}
\end{align}
In the estimate of the second term, since $h_{tt}$appears only in the $L^2$-norm within the energy estimate, we need to integrate by parts twice and write
\begin{align}
    \begin{split}
    \tilde{T}_{13}^{(2)}&=\int_0^t\int_{\Omega}\partial_{tt}b_{k\alpha}p\partial_kv_\alpha
    =-\int_0^t\int_\Omega \partial_{tt}(\partial_kb_{k\alpha})
    pv_{\alpha}
    -\int_0^t\int_\Omega\partial_{tt}b_{k\alpha}\partial_kpv_\alpha
    \\&
    =-\sum_{k=1,2}
    \int_0^t\int_\Omega\partial_{tt}b_{k\alpha}\partial_kpv_\alpha
    \\&
    =-\sum_{k=1,2}\sum_{\alpha=1,2}
    \int_0^t\int_\Omega \partial_{tt}\partial_3\phi\partial_kpv_\alpha
    =\sum_{k=1,2}\sum_{\alpha=1,2}
    \int_0^t\int_\Omega \partial_{tt}\phi \partial_kp\partial_3v_\alpha
    \\&
    \lesssim \sum_{k=1,2}
    \int_0^t \left\Vert \Vert \partial_{tt}\phi\Vert_{L^2_z}\right\Vert_{L^4_{\HH}}
    \left\Vert \Vert \partial_kp\Vert_{L^\infty_z}\right\Vert_{L^4_{\HH}}
    \Vert v\Vert_{L^2}^\frac{1}{2}
    \Vert \partial_{33}v\Vert_{L^2}^\frac{1}{2}
    \\&
    \lesssim \int_0^t\mathcal{P}
    +\epsilon\int_0^t(\Vert p\Vert ^2_{H^2}+\Vert v\Vert_{H^2}^2)
    .
    \end{split}
    \llabel{EQ102b}
\end{align}
For the third term, it suffices to use the anisotropic estimate
\begin{align}
    \begin{split}
        \tilde{T}_{13}^{(3)}&=\int_0^t\int_{\Omega}\partial_t b_{k\alpha}p\partial_k\partial_tv_\alpha
        =-\int_0^t\int_\Omega\partial_t\partial_kb_{k\alpha}p\partial_tv_\alpha
        -\int_0^t\int_\Omega\partial_tb_{k\alpha}\partial_kp\partial_tv_\alpha
        \\&=\sum_{k=1,2}-\int_0^t\int_\Omega\partial_tb_{k\alpha}\partial_kp\partial_tv_\alpha
        \lesssim
        \int_0^t\Vert\partial_tb_{k\alpha}\Vert_{L^3}
        \Vert\partial_kp\Vert_{L^6}
        \Vert v_\alpha\Vert_{L^2}
        \\&\lesssim
        \mathcal{P}+\epsilon\Vert p\Vert_{H^2}^2
        .
        \llabel{EQ102c}
    \end{split}
\end{align}
In conclusion, we derive the time derivative estimate,
\begin{align}
\begin{split}
       &\frac{1}{2}\int J \partial_tv_{\alpha}\partial_tv_{\alpha}
        +\iint J\nablah\partial_tv_\alpha \nablah\partial_tv_\alpha
        +\iint\left(\frac{1+b_{31}^2+b_{32}^2}{J}\right)\partial_3\partial_tv_\alpha
        \partial_3\partial_tv_\alpha
        \\&\indeq\indeq
        +\frac{1}{2}\int (\partial_{tt}h)^2
        +\frac{1}{2}\int (\Delta_{\HH}h_t)^2
        \\&\indeq
        \lesssim
        \frac{1}{2}\int J(0) \partial_tv_{\alpha}(0)
        \partial_tv_{\alpha}(0)
        +\frac{1}{2}\int (\partial_{tt}h)^2(0)
        +\frac{1}{2}\int (\Delta_{\HH}h_t)^2(0)
        \\&\indeq\indeq
        +\int_0^t\mathcal{P}
        +\int_0^t(\epsilon+\PPs)(\Vert v_t\Vert_{H^1}^2+\Vert v\Vert_{H^3}^2)
        +\int_0^t\epsilon(\Vert p\Vert ^2_{H^2}+\Vert v\Vert_{H^2}^2)
        +\int \partial_tb_{k\alpha}(0)p(0)\partial_k v_\alpha(0)
        .
\end{split}
\llabel{EQ103}
\end{align}
In conclusion, we derive the time derivative estimate~\eqref{EQ103}.
\end{proof}

\startnewsection{Proof of the main theorem}{sec17}
We are finally ready to conclude the proof of the main theorem.

\begin{proof}[Proof of Theorem~\ref{T01}]
By~\eqref{EQ79}, and after absorbing $\Vert \partial_{333}v\Vert_{L^2}$, we deduce the estimate valid when $t$ is close to zero:
\begin{align}
    \begin{split}
        \Vert \partial_{333}v\Vert_{L^2}
        \lesssim
        (\Vert v\Vert_{H^1}+
        \Vert \partial'v\Vert_{H^1}+
        \Vert \partial''v\Vert_{H^1}
        +\Vert v_t\Vert_{H^1})
        \mathcal{P}
        +(\Vert \partial_{33}v\Vert_{L^2}
        +\Vert \partial_{33}\partial'v\Vert_{L^2})\PPs
        .
    \end{split}
    \label{EQ104}
\end{align}
Combining \eqref{EQ69} and \eqref{EQ74}, we also obtain
\begin{align}
    \Vert \partial_{33}\partial'v\Vert_{L^2}+\Vert \partial_{33}v\Vert_{L^2}
    \lesssim
    (\Vert v\Vert_{H^1}+
        \Vert \partial'v\Vert_{H^1}+
        \Vert \partial''v\Vert_{H^1}
        +\Vert v_t\Vert_{H^1})
        \mathcal{P}
        +\Vert p\Vert_{H^2}
        \mathcal{P}
        +\mathcal{P}
        .
        \label{EQ105}
\end{align}
Keeping in mind that the elliptic estimate of $\Vert
\partial_{333}v\Vert_{L^2}$ is independent of $\Vert p\Vert_{H^2}$, by
substituting \eqref{EQ104} and \eqref{EQ105} into \eqref{EQ97}, we deduce
\begin{align}
    \Vert p\Vert_{H^2}\lesssim
    (\Vert v\Vert_{H^1}+
        \Vert \partial'v\Vert_{H^1}+
        \Vert \partial''v\Vert_{H^1}
        +\Vert v_t\Vert_{H^1})
        (\PPs+\epsilon)
        +\mathcal{P}
        .
        \label{EQ106}
\end{align}
Inserting \eqref{EQ104}, \eqref{EQ105}, and \eqref{EQ106} into \eqref{EQ64},
we arrive at a simpler tangential estimate
\begin{align}
    \begin{split}
        &\frac{1}{2}\frac{d}{dt}\int J \partial''v_{\alpha}\partial''v_{\alpha}
        +\int J\nablah\partial''v_\alpha\nablah\partial''v_\alpha
        +\int\left(\frac{1+b_{31}^2+b_{32}^2}{J}\right)\partial_3\partial''v_\alpha\partial_3\partial''v_\alpha
        \\&\indeq
        +\frac{1}{2}\frac{d}{dt}\int (\partial'' \partial_th)^2
        +\frac{1}{2}\frac{d}{dt}\int (\partial''\Delta_{\HH}h)^2
	\\&\indeq
        \lesssim 
        \mathcal{P}
        +\epsilon\Vert v_t\Vert_{H^1}^2
        +\epsilon(\Vert v\Vert_{H^1}+
        \Vert \partial'v\Vert_{H^1}+
        \Vert \partial''v\Vert_{H^1}
        +\Vert v_t\Vert_{H^1})^2
        \PPs
        .
    \end{split}
    \llabel{EQ107}
\end{align}
An analogous argument applies to the time estimate
\begin{align}
\begin{split}
       &\frac{1}{2}\int J \partial_tv_{\alpha}\partial_tv_{\alpha}
        +\iint J\nablah\partial_tv_\alpha \nablah\partial_tv_\alpha
        +\iint\left(\frac{1+b_{31}^2+b_{32}^2}{J}\right)\partial_3\partial_tv_\alpha
        \partial_3\partial_tv_\alpha
        \\&\indeq
        +\frac{1}{2}\int (\partial_{tt}h)^2
        +\frac{1}{2}\int (\Delta_{\HH}h_t)^2
        \\&\indeq
        \lesssim
        \mathcal{P}_0
        +\int_0^t\mathcal{P}
        +\int_0^t
        (\Vert v\Vert_{H^1}+
        \Vert \partial'v\Vert_{H^1}+
        \Vert \partial''v\Vert_{H^1})^2
        \PPs
        .
\end{split}
\llabel{EQ108}
\end{align}
Choosing $\epsilon$ sufficiently small so that $\Vert v_t\Vert_{H^1}$
can be absorbed, we obtain the final a~priori estimate
\begin{align}
    \begin{split}
        &\frac{1}{2}\int J \partial''v_{\alpha}\partial''v_{\alpha}
        +\iint J\nablah\partial''v_\alpha\nablah\partial''v_\alpha
        +\iint\left(\frac{1+b_{31}^2+b_{32}^2}{J}\right)\partial_3\partial''v_\alpha\partial_3\partial''v_\alpha
        \\&\indeq\indeq
        +\frac{1}{2}\int J \partial_tv_{\alpha}\partial_tv_{\alpha}
        +\iint J\nablah\partial_tv_\alpha \nablah\partial_tv_\alpha
        +\iint\left(\frac{1+b_{31}^2+b_{32}^2}{J}\right)\partial_3\partial_tv_\alpha
        \partial_3\partial_tv_\alpha
        \\&\indeq\indeq
        +\frac{1}{2}\int (\partial'' \partial_th)^2
        +\frac{1}{2}\int (\partial''\Delta_{\HH}h)^2
        +\frac{1}{2}\int (\partial_{tt}h)^2
        \\&\indeq
        \lesssim
        \mathcal{P}_0
        +\int_0^t\mathcal{P}
        .
    \end{split}
    \label{EQ109}
\end{align}
In order to control all spatial derivatives up to second order, we need to introduce the $L^2$ energy estimate and then apply interpolation.
Arguing as in the previous tangential estimates, we deduce
\begin{align}
    \begin{split}
       &\frac{1}{2}\frac{d}{dt}\int J v_{\alpha}v_{\alpha}
        +\int J\nablah v_\alpha\nablah v_\alpha
        +\int\frac{1+b_{31}^2+b_{32}^2}{J}\partial_3v_\alpha\partial_3v_\alpha
        +\frac{1}{2}\frac{d}{dt}\int ( \partial_th)^2
        +\frac{1}{2}\frac{d}{dt}\int (\Delta_{\HH}h)^2
	\\&\indeq
	=-\int\nablah J\nablah v_\alpha v_\alpha
        -\int\partial_3\left(\frac{1+b_{31}^2+b_{32}^2}{J}\right)\partial_3 v_\alpha v_\alpha
	.
    \end{split}
    \label{EQ110}
\end{align}
By directly applying Hölder’s inequality together with Young’s inequality, we obtain
\begin{align}
    \begin{split}
        &-\int\nablah J\nablah v_\alpha v_\alpha
        -\int\partial_3\left(\frac{1+b_{31}^2+b_{32}^2}{J}\right)\partial_3 v_\alpha v_\alpha
        \\&\indeq
        \lesssim \Vert\nablah J\Vert_{L^\infty}
        \Vert\nablah v\Vert_{L^2}
        \Vert v\Vert_{L^2}
        +\Vert \partial_3\left(\frac{1+b_{31}^2+b_{32}^2}{J}\right)\Vert_{L^\infty}
        \Vert\partial_3v\Vert_{L^2}
        \Vert v\Vert_{L^2}
        \\&\indeq
        \lesssim \mathcal{P}+\epsilon\Vert v\Vert_{H^1}^2
        .
        \label{EQ111}
    \end{split}
\end{align}
After integrating both sides of inequalities~\eqref{EQ110} and~\eqref{EQ111} with respect to time from 0 to $t$ and adding them to~\eqref{EQ109}, and by applying Lemma~\ref{L02} we obtain
\begin{align}
    \begin{split}
        &\int  \partial''v_{\alpha}\partial''v_{\alpha}
        +\iint \nablah\partial''v_\alpha\nablah\partial''v_\alpha
        +\iint \partial_3\partial''v_\alpha\partial_3\partial''v_\alpha
        \\&\indeq\indeq
        +\int  \partial_tv_{\alpha}\partial_tv_{\alpha}
        +\iint \nablah\partial_tv_\alpha \nablah\partial_tv_\alpha
        +\iint \partial_3\partial_tv_\alpha
        \partial_3\partial_tv_\alpha
        \\&\indeq\indeq
        +\int  v_{\alpha}v_{\alpha}
        +\iint \nablah v_\alpha\nablah v_\alpha
        +\iint \partial_3v_\alpha\partial_3v_\alpha
        \\&\indeq\indeq
        +\int (\partial'' \partial_th)^2
        +\int (\partial''\Delta_{\HH}h)^2
        +\int (\partial_{tt}h)^2
        +\int ( \partial_th)^2
        +\int (\Delta_{\HH}h)^2
        \\&\indeq
        \lesssim
        \mathcal{P}_0
        +\int_0^t\mathcal{P}
        .
    \end{split}
    \label{EQ112}
\end{align}
In view of 
\begin{align}
    \Vert\partial'v\|^2_{L^2}
    \lesssim
    \Vert v\Vert_{L^2}
    \Vert\partial''v\Vert_{L^2}
    \lesssim
    \Vert v\Vert_{L^2}^2
    +\Vert\partial''v\Vert_{L^2}^2
    ,
  \label{EQ113}  
\end{align}
 the polynomial $\mathcal{P}$ can be controlled by a polynomial expression $\tilde{\mathcal{P}}$ that does not depend on~$\Vert\partial'v\Vert_{L^2}$.
For $\mathcal{P}_0$, an analogous quantity $\tilde{\mathcal{P}_0}$ can be defined in the same manner. Applying a Gronwall argument on~\eqref{EQ112}, we obtain
an estimate
\begin{align}
\begin{split}
    &\sup_{0\leq t\leq \TT} \left(\|\partial''v\|^2_{L^2}
    +\|v\|^2_{L^2}+\|v_t\|^2_{L^2}
    +\|h_{tt}\|^2_{L^2}+\|h_t\|_{H^2}^2+\|h\|^2_{H^4}\right)
    \\&\indeq
    +\int_0^{\TT} \left(\|\partial''v\|^2_{H^1}+\|\partial'\partial_3v\|^2_{L^2}
    +\|\partial_3v\|^2_{L^2}+\|v_t\|^2_{H^1}\right)\, dt
    \leq K
    ,
\end{split}
   \llabel{EQ114}
\end{align}
where $K$ is an explicit polynomial of~$M_0$, for $\tilde{T}$
depending on~$M_0$. By using \eqref{EQ113}, it follows that 
\begin{align}
    \sup_{0\leq t\leq \TT}\Vert\partial'v\Vert_{L^2} \leq K
   \llabel{EQ135}
\end{align}
as well, and the proof of the theorem is concluded.
\end{proof}

\section*{Acknowledgments}
\rm
IK and QX were supported in part by the NSF grant DMS-2205493,
while MH acknowledges the support by DFG through Reseach Unit FOR5528.

\colb

\colb

\ifnum\sketches=1
\newpage
\begin{center}
  \bf   Notes?\rm
\end{center}
\huge
\colb

\fi


\small
\begin{thebibliography}{10}
\bibitem{AS} %MR3656704
D.M.~Ambrose and M.~Siegel,
  \emph{Well-posedness of two-dimensional hydroelastic waves},
  Proc. Roy. Soc. Edinburgh Sect.~A~\textbf{147} (2017), no.~3, 529--570.
  %\MR{3656704}
  

\bibitem{B} %MR2027753
H.~Beir\~ao~da Veiga,
  \emph{On the existence of strong solutions to a coupled fluid-structure evolution problem},
  J.~Math. Fluid Mech. \textbf{6} (2004), no.~1, 21--52.
  %\MR{2027753}

\bibitem{BKMT}
A.~Balakrishna, I.~Kukavica, B.~Muha, and A.~Tuffaha,
\emph{Inviscid fluid interacting with a nonlinear two-dimensional plate},
Interfaces Free Bound.~\textbf{27} (2025), 141–-175.


\bibitem{BH} %MR4396798
T.~Binz and M.~Hieber,
  \emph{Global wellposedness of the primitive equations with nonlinear equation of state in critical spaces},
  J. Math. Fluid Mech.~\textbf{24} (2022), no.~2, Paper No.~36,~18.
  %\MR{4396798}


\bibitem{BS} %MR3766983
D.~Breit and S.~Schwarzacher, 
\emph{Compressible fluids interacting with a linear-elastic shell}, 
Arch.\ Ration.\ Mech.\ Anal.~\textbf{228} (2018), no.~2, 495--562. 
  


\bibitem{CDEG} %MR2166981
A.~Chambolle, B.~Desjardins, M.J.~Esteban, and C.~Grandmont,
  \emph{Existence of weak solutions for the unsteady interaction of a viscous fluid with an elastic plate},
  J.~Math.\ Fluid Mech.~\textbf{7} (2005), no.~3, 368--404.
   %\MR{2166981}

\bibitem{CINT} %MR3339156
C.~Cao, S.~Ibrahim, K.~Nakanishi, and E.S. Titi,
  \emph{Finite-time blowup for the inviscid primitive equations of oceanic and   atmospheric dynamics},
  Comm. Math. Phys.~\textbf{337} (2015), no.~2, 473--482.
    %\MR{3339156}

\bibitem{CLT1} %MR3237881
C.~Cao, J.~Li, and E.~S.~Titi,
  \emph{Local and global well-posedness of strong solutions to the 3{D} primitive equations with vertical eddy diffusivity},
  Arch. Ration. Mech. Anal.~\textbf{214} (2014), no.~1, 35--76.
  %\MR{3237881}

\bibitem{CLT2} %MR3264417
C.~Cao, J.~Li, and E.S.~Titi,
  \emph{Global well-posedness of strong solutions to the 3{D} primitive equations with horizontal eddy diffusivity},
  J.~Differential Equations \textbf{257} (2014), no.~11, 4108--4132.
    %\MR{3264417}


\bibitem{CT} %MR2342696
C.~Cao and E.S. Titi,
  \emph{Global well-posedness of the three-dimensional viscous primitive equations of large scale ocean and atmosphere dynamics},
  Ann. of Math.~(2)~\textbf{166} (2007), no.~1, 245--267.
  %\MR{2342696}

\bibitem{CCS} %MR2349865
C.H.~Arthur Cheng, D.~Coutand, and S.~Shkoller,
  \emph{Navier-Stokes equations interacting with a nonlinear elastic biofluid shell},
  SIAM J.~Math.\ Anal.~\textbf{39} (2007), no.~3, 742--800.
  %\MR{2349865}

\bibitem{CS} %MR2644917
C.H.~Arthur Cheng and S.~Shkoller,
  \emph{The interaction of the 3{D} Navier-Stokes equations with a moving nonlinear {K}oiter elastic shell},
  SIAM J.~Math.\ Anal.~\textbf{42} (2010), no.~3, 1094--1155.
  %\MR{2644917}


\bibitem{DEGL} %MR1871311
B.~Desjardins, M.J.~Esteban, C.~Grandmont, and P.~Le~Tallec,
\emph{Weak solutions for a fluid-elastic structure interaction model},
Rev.\ Mat.\ Complut.~\textbf{14} (2001), no.~2, 523--538.




\bibitem{GM} %MR1763528
C.~Grandmont and Y.~Maday,
\emph{Existence for an unsteady fluid-structure interaction problem},
M2AN Math.\ Model.\ Numer.\ Anal.~\textbf{34} (2000), no.~3, 609--636.



  
\bibitem{GH} %MR3466847
C.~Grandmont and M.~Hillairet,
  \emph{Existence of global strong solutions to a beam-fluid interaction system},
  Arch.\ Ration.\ Mech.\ Anal.~\textbf{220} (2016), no.~3, 1283--1333.
    %\MR{3466847}

\bibitem{GHL} %MR3955112
C.~Grandmont, M.~Hillairet, and J.~Lequeurre,
  \emph{Existence of local strong solutions to fluid-beam and fluid-rod interaction systems},
  Ann. Inst. H.~Poincar\'{e} C Anal.\ Non Lin\'{e}aire~\textbf{36} (2019), no.~4, 1105--1149.  
    %\MR{3955112}


 


\bibitem{HC} %HaddaraCao96
M.R.~Haddara and S.~Cao,
\emph{A study of the dynamic response of submerged rectangular flat plates},
{Marine Structures}~\textbf{9} (1996), no.~10, 913--933.


\bibitem{H} %Hibler79
W.D.~Hibler~III,
\emph{A dynamic thermodynamic sea ice model},
{J.~Phys.~Oceanogr.}~\textbf{9} (1979), 817--846.

\bibitem{HIRZ} %MR4906321
M.~Hieber, Y.~Iida, A.~Roy, and T.~Z\"ochling,
  \emph{The hydrostatic {L}agrangian approach to the compressible primitive equations},
  Math. Ann.~\textbf{392} (2025), no.~2, 2277--2308.
    %\MR{4906321}


\bibitem{HK} %MR3508997
M.~Hieber and T.~Kashiwabara,
  \emph{Global strong well-posedness of the three dimensional primitive equations in {$L^p$}-spaces},
  Arch. Ration. Mech. Anal.~\textbf{221} (2016), no.~3, 1077--1115.
  %\MR{3508997}


\bibitem{Hu} %Hutter75
K.~Hutter,
\emph{Floating sea ice plates and the significance of the dependence of the Poisson ratio on brine content},
{Proc.~R.~Soc.~Lond.~A}~\textbf{343} (1975), 85--108.

\bibitem{HRLTCT} %Huang19
L.~Huang, K.~Ren, M.~Li, Ž.~Tukovi\'c, P.~Cardiff, and G.~Thomas,
\emph{Fluid-structure interaction of a large ice sheet in waves},
{Ocean Eng.}~\textbf{182} (2019), 102--111.


\bibitem{Hun} %Hunke01
E.C.~Hunke,
\emph{Viscous-plastic sea ice dynamics with the EVP model: linearization issues},
{J.~Comput.~Phys.}~\textbf{170} (2001), 18--38.

\bibitem{J.}
J. H. Jungclaus et al,
\emph{The ICON Earth System Model Version 1.0},
Journal Advances in Modeling Earth Sciences~\textbf{14} (2022)

\bibitem{K} %Kakinuma01
T.~Kakinuma,
\emph{A nonlinear numerical model for the interaction of surface and internal waves with very large floating or submerged flexible platforms},
{Trans.\ Built Environ.}~\textbf{56} (2001), WIT Press, \url{https://www.witpress.com}, ISSN~1743-3509.


\bibitem{KPV} %Korobkin11
A.~Korobkin, E.~I.~Parau, and J.~Vanden-Broeck,
\emph{The mathematical challenges and modelling of hydroelasticity},
{Philos.~Trans.~R.~Soc.~A}~\textbf{369} (2011), 2803--2812.

\bibitem{KLT1} %MR4819916
I.~Kukavica, L.~Li, and A.~Tuffaha,
  \emph{On the local existence of solutions to the fluid-structure interaction problem with a free interface},
  Appl. Math. Optim.~\textbf{90} (2024), no.~3, Paper No.~53,~31.
    %\MR{4819916}

\bibitem{KLT2} %MR4718632
I.~Kukavica, L.~Li, and A.~Tuffaha,
  \emph{On the local existence of solutions to the compressible {N}avier-{S}tokes-wave system with a free interface},
  J.~Math. Fluid Mech.~\textbf{26} (2024), no.~2, Paper No. 25, 37.
  %\MR{4718632}

\bibitem{KMVW} %MR3284569
I.~Kukavica, N.~Masmoudi, V.~Vicol, and T.K.~Wong,
  \emph{On the local well-posedness of the {P}randtl and hydrostatic {E}uler equations with multiple monotonicity regions},
  SIAM J.~Math. Anal.~\textbf{46} (2014), no.~6, 3865--3890.
  %\MR{3284569}

\bibitem{KNT}
I.~Kukavica, \v S\'arka Ne\v casov\'a, and Amjad Tuffaha,
\emph{Compressible Euler Equations in an elastic domain}
{Indiana Univ.\ Math.~J.}
(to appear),
arXiv:2311.08731.

\bibitem{KT1} %MR4795678
I.~Kukavica and A.~Tuffaha,
  \emph{A free boundary inviscid model of flow-structure interaction},
  J.~Differential Equations \textbf{413} (2024),
  851--912.
  %\MR{4795678}

\bibitem{KT2} %MR4596023
I.~Kukavica and A.~Tuffaha,
  \emph{An inviscid free boundary fluid-wave model},
  J.~Evol. Equ. \textbf{23} (2023), no.~2, Paper No. 41, 18.
  %\MR{4596023}

\bibitem{L} %MR3233099
D.~Lengeler,
  \emph{Weak solutions for an incompressible, generalized {N}ewtonian fluid interacting with a linearly elastic {K}oiter type shell},
  SIAM J.~Math.\ Anal.~\textbf{46} (2014), no.~4, 2614--2649.
    %\MR{3233099}


\bibitem{LR} %MR3147436
D.~Lengeler and M.~R\r{u}\v{z}i\v{c}ka,
  \emph{Weak solutions for an incompressible {N}ewtonian fluid interacting with a {K}oiter type shell},
  Arch.\ Ration.\ Mech.\ Anal.~\textbf{211} (2014), no.~1, 205--255.
    %\MR{3147436}
 
\bibitem{L1} %MR2765696
J.~Lequeurre,
\emph{Existence of strong solutions to a fluid-structure system},
SIAM J.~Math.\ Anal.~\textbf{43} (2011), no.~1, 389--410.


\bibitem{L2} %MR3061763
J.~Lequeurre,
\emph{Existence of strong solutions for a system coupling the Navier-Stokes equations and a damped wave equation},
J.~Math.\ Fluid Mech.~\textbf{15} (2013), no.~2, 249--271.

\bibitem{LL} %MR4794622
H.-L.~Li and C.~Liang,
  \emph{Global existence and large-time behavior for primitive equations with the free boundary},
  Sci. China Math.~\textbf{67} (2024), no.~10, 2303--2330.
  %\MR{4794622}

\bibitem{IKZ} %MR3050590
M.~Ignatova, I.~Kukavica, and M.~Ziane,
  \emph{Local existence of solutions to the free boundary value problem for the primitive equations of the ocean},
  J.~Math. Phys.~\textbf{53} (2012), no.~10, 103101, 17.
  %\MR{3050590}

\bibitem{KTVZ} %MR2737223
I.~Kukavica, R.~Temam, V.C.~Vicol, and M.~Ziane,
  \emph{Local existence and uniqueness for the hydrostatic {E}uler equations on a bounded domain},
  J.~Differential Equations~\textbf{250} (2011), no.~3, 1719--1746.
  %\MR{2737223}

\bibitem{KZ} %MR2368323
I.~Kukavica and M.~Ziane,
  \emph{On the regularity of the primitive equations of the ocean},
  Nonlinearity~\textbf{20} (2007), no.~12, 2739--2753.
  %\MR{2368323}

\bibitem{LT1} %MR3592073
J.~Li and E.S.~Titi,
  \emph{Existence and uniqueness of weak solutions to viscous primitive equations for a certain class of discontinuous initial data},
  SIAM J. Math. Anal.~\textbf{49} (2017), no.~1, 1--28.
  %\MR{3592073}

\bibitem{LT2} %MR3926040
J.~Li and E.S.~Titi,
  \emph{The primitive equations as the small aspect ratio limit of the {N}avier-{S}tokes equations: rigorous justification of the hydrostatic approximation},
  J.~Math. Pures Appl.~(9)~\textbf{124} (2019), 30--58.
  %\MR{3926040}

\bibitem{LTW1} %MR1158375
J.-L.~Lions, R.~Temam, and S.H.~Wang,
  \emph{New formulations of the primitive equations of atmosphere and applications},
  Nonlinearity~\textbf{5} (1992), no.~2, 237--288.
  %\MR{1158375}

\bibitem{LTW2} %MR1187737
J.-L.~Lions, R.~Temam, and S.H.~Wang,
  \emph{On the equations of the large-scale ocean},
  Nonlinearity~\textbf{5} (1992), no.~5, 1007--1053.
  %\MR{1187737}

\bibitem{LTW3} %MR1252502
J.-L.~Lions, R.~Temam, and S.~Wang,
  \emph{Models for the coupled atmosphere and ocean. ({CAO} {I},{II})},
  Comput. Mech. Adv.~\textbf{1} (1993), no.~1, 120.
  %\MR{1252502}

\bibitem{LTW4} %MR1325825
J.-L.~Lions, R.~Temam, and S.~H.~Wang,
  \emph{Mathematical theory for the coupled atmosphere-ocean models. ({CAO} {III})},
  J.~Math. Pures Appl.~(9)~\textbf{74} (1995), no.~2, 105--163.
  %\MR{1325825}

\bibitem{LA} %MR3608168
S.~Liu and D.M.~Ambrose,
  \emph{Well-posedness of two-dimensional hydroelastic waves with mass},
  J.~Differential Equations~\textbf{262} (2017), no.~9, 4656--4699.
  %\MR{3608168}



\bibitem{MRR} %MR4189724
  D.~Maity, J.-P.~Raymond, and A.~Roy,
  \emph{Maximal-in-time existence and uniqueness of strong solution of a 3{D} fluid-structure interaction model},
  SIAM J.~Math.\ Anal.~\textbf{52} (2020), no.~6, 6338--6378.
    %\MR{4189724}
 
\bibitem{MW} %MR2898740
N.~Masmoudi and T.K.~Wong,
  \emph{On the {$H^s$} theory of hydrostatic {E}uler equations},
  Arch. Ration. Mech. Anal.~\textbf{204} (2012), no.~1, 231--271.
  %\MR{2898740}
  

\bibitem{Mi} %MR4152298
S.~Mitra,
  \emph{Local existence of strong solutions of a fluid-structure interaction model},
  J.~Math. Fluid Mech.~\textbf{22} (2020), no.~4, Paper No.~60,~38.
  %\MR{4152298}

\bibitem{MHJLR} %Marsland03
S.J.~Marsland, H.~Haak, J.H.~Jungclaus, M.~Latif, and F.~R\"oske,
\emph{The Max-Planck-Institute global ocean/sea ice model with orthogonal curvilinear coordinates},
{Ocean Model.}~\textbf{5} (2003), no.~2, 91--127.

\bibitem{MC1} %MR3017292
B.~Muha and S.~\v Cani\'{c},
\emph{Existence of a weak solution to a nonlinear fluid-structure interaction problem modeling the flow of an incompressible, viscous fluid in a cylinder with deformable walls},
  Arch.\ Ration.\ Mech.\ Anal.~\textbf{207} (2013), no.~3, 919--968.
  %\MR{3017292}


\bibitem{MC2} %MR3482692
B.~Muha and S.~\v{C}ani\'{c},
  \emph{Existence of a weak solution to a fluid-elastic structure interaction problem with the Navier-slip boundary condition},
  J.~Differential Equations~\textbf{260} (2016), no.~12, 8550--8589.
    %\MR{3482692}


\bibitem{MS} %MR4540753
  B.~Muha and S.~Schwarzacher,
  \emph{Existence and regularity of weak solutions for a fluid interacting with a non-linear shell in three dimensions},
  Ann.\ Inst.\ H.~Poincar\'{e} C Anal.~Non Lin\'{e}aire~\textbf{39} (2022), no.~6, 1369--1412.
  %\MR{4540753}


\bibitem{PB} %Papathanasiou14
T.K.~Papathanasiou and K.A.~Belibassakis,
\emph{Hydroelastic analysis of VLFS based on a consistent coupled-mode system and FEM},
{IES J.~Part~A:~Civil \& Structural Eng.}~\textbf{7} (2014), no.~3, 195--206.


  %\MR{2812947}


 


\bibitem{PV} %Parau11
E.I.~P\u{a}r\u{a}u and J.-M.~Vanden-Broeck,
\emph{Three-dimensional waves beneath an ice sheet due to a steadily moving pressure},
{Philos.~Trans.~R.~Soc.~A}~\textbf{369} (2011), 2973--2988.



\bibitem{SS} %MR4450291
S.~Schwarzacher and M.~Sroczinski,
  \emph{Weak-strong uniqueness for an elastic plate interacting with the {N}avier-{S}tokes equation},
  SIAM J.~Math. Anal.~\textbf{54} (2022), no.~4, 4104--4138.
  %\MR{4450291}


 
\bibitem{SSu}
S.~Schwarzacher and P.~Su,
\emph{Existence of strong solutions for a perfect elastic beam interacting with Navier-Stokes equations},
arXiv:2308.04253.

\bibitem{PTZ} %MR2454281
M.~Petcu, R.M.~Temam, and M.~Ziane,
  \emph{Some mathematical problems in geophysical fluid dynamics},
  Handbook of numerical analysis.   {V}ol. {XIV}. {S}pecial volume: computational methods for the atmosphere and
  the oceans, Handb. Numer. Anal., vol.~14, Elsevier/North-Holland, Amsterdam,
  2009, pp.~577--750.
  %\MR{2454281}

\bibitem{P} %Porter19
R.~Porter,
\emph{The coupling between ocean waves and rectangular ice sheets},
{J.~Fluids Struct.}~\textbf{84} (2019), 171--181.

\bibitem{R} %MR2563627
M.~Renardy,
  \emph{Ill-posedness of the hydrostatic {E}uler and {N}avier-{S}tokes equations},
  Arch. Ration. Mech. Anal.~\textbf{194} (2009), no.~3, 877--886.
    %\MR{2563627}

\bibitem{S} %Squire20
V.A.~Squire,
\emph{Ocean wave interactions with sea ice: a reappraisal},
{Annu.~Rev.~Fluid Mech.}~\textbf{52} (2020), 37--60.

\bibitem{SDWRL} %Squire95
V.A.~Squire, G.P.~Dugan, G.P.~Wadhams, P.J.~Rottier, and A.K.~Liu,
\emph{Of ocean waves and sea ice},
{Annu.~Rev.~Fluid Mech.}~\textbf{27} (1995), 115--168.


\bibitem{THAB} %Tavakoli22
S.~Tavakoli, L.~Huang, F.~Azhari, and A.V.~Babanin,
\emph{Viscoelastic wave-ice interactions: a computational fluid-solid dynamic approach},
{J.~Mar.~Sci.~Eng.}~\textbf{10} (2022), 1220.


 

\bibitem{Tr} %MR4526391
S.~Trifunovi\'c,
  \emph{Compressible fluids interacting with plates: regularity and weak-strong uniqueness},
  J. Math. Fluid Mech.~\textbf{25} (2023), no.~1, Paper No.~13, 28.
  %\MR{4526391}

\bibitem{TW} %MR4042350
S.~Trifunovi\'c and Y.-G.~Wang,
  \emph{Existence of a weak solution to the fluid-structure interaction problem in 3{D}},
  J.~Differential Equations~\textbf{268} (2020), no.~4, 1495--1531.
  %\MR{4042350}


\bibitem{W} %MR3293727
T.K.~Wong,
  \emph{Blowup of solutions of the hydrostatic {E}uler equations},
  Proc. Amer. Math. Soc.~\textbf{143} (2015), no.~3, 1119--1125.
  %\MR{3293727}

\bibitem{WY} %MR4104949
Z.~Wang and J.~Yang,
  \emph{Energy estimates and local well-posedness of 3{D} interfacial hydroelastic waves between two incompressible fluids},
  J.~Differential Equations~\textbf{269} (2020), no.~7, 6055--6087.
  %\MR{4104949}

  

\bibitem{XEK} %Xia08
D.~Xia, R.~C.~Ertekin, and J.~W.~Kim,
\emph{Fluid-structure interaction between a two-dimensional mat-type VLFS and solitary waves by the Green--Naghdi theory},
{J.~Fluids Struct.}~\textbf{24} (2008), no.~4, 527--540.


\bibitem{XW} %Xu22
P.~Xu and P.R.~Wellens,
\emph{Theoretical analysis of nonlinear fluid--structure interaction between large-scale polymer offshore floating photovoltaics and waves},
{Ocean Eng.}~\textbf{249} (2022), 110829.

\end{thebibliography}
\end{document}